\documentclass{article}[11pt]
\usepackage{amsmath}
\usepackage{latexsym}
\usepackage{amssymb}
\usepackage{amsthm}
\usepackage[english]{babel}
\usepackage{blindtext}
\usepackage{mathtools}
\usepackage{mathrsfs}
\usepackage{cite}
\usepackage{hyperref}
\usepackage{lmodern}
\usepackage{graphicx}
\usepackage{microtype}

\allowdisplaybreaks

\newtheorem{theorem}{Theorem}[section]
\newtheorem{lemma}{Lemma}[section]
\newtheorem{definition}{Definition}[section]
\newtheorem{corollary}{Corollary}[section]
\newtheorem{remark}{Remark}[section]
\newtheorem{example}{Example}[section]

\newcommand{\lcm}{\mathrm{lcm}}

\newcommand{\es}{\varnothing}

\newcommand{\ExpecCond}[2]{\mathbb{E}[#1 \kern0.1em|\kern0.1em #2]}
\newcommand{\bigExpecCond}[2]{\mathbb{E}\bigl[#1 \kern-0.1em \bigm| \kern-0.1em #2\bigr]}
\newcommand{\BigExpecCond}[2]{\mathbb{E}\Bigl[#1 \kern-0.1em \Bigm| \kern-0.1em #2\Bigr]}
\newcommand{\biggExpecCond}[2]{\mathbb{E}\biggl[#1 \kern-0.1em \biggm| \kern-0.1em #2\biggr]}
\newcommand{\BiggExpecCond}[2]{\mathbb{E}\Biggl[#1 \kern-0.1em \Biggm| \kern-0.1em #2\Biggr]}

\newcommand{\ExpecCondwrt}[3]{\mathbb{E}_{#1}[#2 \kern0.1em|\kern0.1em #3]}
\newcommand{\bigExpecCondwrt}[3]{\mathbb{E}_{#1}\bigl[#2 \kern-0.1em \bigm| \kern-0.1em #3\bigr]}
\newcommand{\BigExpecCondwrt}[3]{\mathbb{E}_{#1}\Bigl[#2 \kern-0.1em \Bigm| \kern-0.1em #3\Bigr]}
\newcommand{\biggExpecCondwrt}[3]{\mathbb{E}_{#1}\biggl[#2 \kern-0.1em \biggm| \kern-0.1em #3\biggr]}
\newcommand{\BiggExpecCondwrt}[3]{\mathbb{E}_{#1}\Biggl[#2 \kern-0.1em \Biggm| \kern-0.1em #3\Biggr]}

\newcommand{\OneTo}[1]{[#1]}

\newcommand{\eqdef}{\triangleq} 
\newcommand{\card}[1]{|#1|}
\newcommand{\bigcard}[1]{\bigl|#1\bigr|}
\newcommand{\gbinom}[3]{\Bigl[ \hspace*{-0.2cm} \begin{array}{c} {\hspace*{0.2cm} #1} \\ {\hspace*{-0.3cm} #2} \hspace*{-0.5cm} \end{array}\Bigr]_{#3}}  

\RequirePackage{bbm}

\DeclareMathOperator{\Vertex}{\mathsf{V}}
\DeclareMathOperator{\Edge}{\mathsf{E}}
\DeclareMathOperator{\Adjacency}{\mathbf{A}}

\DeclareMathOperator{\Identity}{\mathbf{I}}
\DeclareMathOperator{\Independentset}{\set{I}}
\DeclareMathOperator{\Cliqueset}{\set{C}}
\DeclareMathOperator{\Kneser}{\mathsf{K}}

\DeclareMathOperator{\Complete}{\mathsf{K}}

\DeclareMathOperator{\Path}{\mathsf{P}}
\DeclareMathOperator{\Cycle}{\mathsf{C}}
\DeclareMathOperator{\SRG}{\mathsf{srg}}

\DeclareMathOperator{\Clique}{\omega}
\DeclareMathOperator{\Chromatic}{\chi}
\DeclareMathOperator{\Cliquecover}{\sigma}

\newcommand{\naturals}{\ensuremath{\mathbb{N}}}
\newcommand{\Rationals}{\ensuremath{\mathbb{Q}}}
\newcommand{\Reals}{\ensuremath{\mathbb{R}}}

\newcommand{\set}{\ensuremath{\mathcal}}

\newcommand{\Gr}[1]{\mathsf{#1}}                          
\newcommand{\CGr}[1]{\overline{\mathsf{#1}}}              
\newcommand{\V}[1]{\Vertex(#1)}                           
\newcommand{\E}[1]{\Edge(#1)}                             
\newcommand{\A}{\Adjacency}                               
\newcommand{\Id}[1]{\Identity_{#1}}                       
\newcommand{\indset}[1]{\Independentset(#1)}              
\newcommand{\indnum}[1]{\alpha(#1)}                       
\newcommand{\indnumbig}[1]{\alpha\bigl(#1\bigr)}

\newcommand{\indnumBigg}[1]{\alpha\Biggl(#1\Biggr)}
\newcommand{\findnum}[1]{\alpha_{\mathrm{f}}(#1)}         
\newcommand{\clset}[1]{\Cliqueset(#1)}                    
\newcommand{\clnum}[1]{\Clique(#1)}                       

\newcommand{\fclnum}[1]{\Clique_{\mathrm{f}}(#1)}         
\newcommand{\chrnum}[1]{\Chromatic(#1)}                   

\newcommand{\fchrnum}[1]{\Chromatic_{\mathrm{f}}(#1)}     

\newcommand{\clcnum}[1]{\Cliquecover(#1)}                 
\newcommand{\fclcnum}[1]{\Cliquecover_{\mathrm{f}}(#1)}   

\newcommand{\CoG}[1]{\Complete_{#1}}       
\newcommand{\KG}[2]{\Kneser(#1,#2)}        
\newcommand{\qKG}[3]{\Kneser_{#3}(#1,#2)}  
\newcommand{\PathG}[1]{\Path_{#1}}         

\newcommand{\CG}[1]{\Cycle_{#1}}           
\newcommand{\srg}[4]{\SRG(#1,#2,#3,#4)}    
\newcommand{\TG}[2]{\mathrm{T}({#1},{#2})} 

\DeclareMathOperator{\Degree}{\text{d}}

\newcommand{\dgr}[1]{\Degree_{#1}}   

\MHInternalSyntaxOn

\renewcommand{\dcases}

{
  \MT_start_cases:nnnn
    {\quad}
    {$\m@th\displaystyle##$\hfil}
    {$\m@th\displaystyle##$\hfil}
    {\lbrace}
}

\makeatletter
\makeatletter
\renewcommand{\@fnsymbol}[1]{%
  \ifcase#1\or
    *\or
    **%
  \else
    \@ctrerr
  \fi}
\makeatother

\MHInternalSyntaxOff

\title{\bf{\huge{Advances in the Shannon Capacity of Graphs}}}

\author{%
Nitay~Lavi\thanks{%
Nitay Lavi is with the Viterbi Faculty of Electrical and Computer Engineering, Technion-Israel
Institute of Technology, Haifa 3200003, Israel.
}
\and
Igal~Sason\thanks{%
Igal Sason is with the Viterbi Faculty of Electrical and Computer Engineering and the Faculty
of Mathematics, Technion-Israel Institute of Technology, Haifa 3200003, Israel. \newline
{\bf{Corresponding author}}: Igal Sason, Email: eeigal@technion.ac.il; Tel: +97248294699. \\[0.1cm]
\hspace*{0.2cm} $\bigstar$ \bf{Final version. Accepted for publication in {\em AIMS Mathematics},
Special Issue: {\em Mathematical Foundations of Information Theory}, January 2026.}}
}

\begin{document}
\maketitle

\vspace*{-0.6cm}
\begin{abstract}
We derive exact values and new bounds for the Shannon capacity of two families of graphs: the
$q$-Kneser graphs and the tadpole graphs. We also construct a countably infinite family of connected
graphs whose Shannon capacity is not attained by the independence number of any finite strong power.
Building on recent work of Schrijver, we establish sufficient conditions under which the Shannon capacity
of a polynomial in graphs, formed via disjoint unions and strong products, equals the corresponding
polynomial of the individual capacities, thereby reducing the evaluation of such capacities to that
of their components.
Finally, we prove an inequality relating the Shannon capacities of the strong product of graphs and their
disjoint union, which yields alternative proofs of several known bounds as well as new tightness conditions.
In addition to contributing to the computation of the Shannon capacity of graphs, this paper is intended
to serve as an accessible entry point to those wishing to work in this area.
\end{abstract}


\vspace*{0.2cm}

\section{Introduction}\label{section:introduction}
The Shannon capacity of a graph, introduced by Claude Shannon in 1956 \cite{Shannon56}, is a central concept bridging
zero-error information theory and graph theory. It characterizes the maximum rate of zero-error
communication over a noisy channel when the channel is modeled by a graph whose vertices represent the input symbols, and
two vertices are adjacent if and only if the corresponding symbols can be confused by the channel with positive probability. This
notion establishes a fundamental link between zero-error problems in information theory and graph theory, as surveyed extensively
in \cite{Alon02,Alon19,Jurkiewicz14,KornerO98,CsiszarK01}.

The calculation of the Shannon capacity of a graph is notoriously difficult (see, e.g., \cite{Alon98,AlonL06,BocheD25,GuoW90}), and only the
capacity of a some families of graphs are known (e.g., Kneser graphs \cite{Lovasz79}). The computational complexity of the Shannon capacity has
extended the research towards finding computable bounds on the capacity and exploring its properties \cite{Alon98,BuysPZ25,BoerBZ24,BiT19,BukhC19,Haemers79,HuTS18,Knuth94,Lovasz79,CsonkaS24,Simonyi21}.

While exploring properties of the Shannon capacity, it is worth keeping in mind the more general concept of the asymptotic spectrum of graphs,
introduced by Zuiddam \cite{Zuiddam19,WigdersonZ23,BuysPZ25,BoerBZ24}, which delineates a space of graph parameters that remain invariant under
graph isomorphisms. This space is characterized by the following unique properties: additivity under disjoint union of graphs, multiplicativity
under strong product of graphs, normalization for a simple graph with a single vertex, and monotonicity under graph complement homomorphisms.
Building upon Strassen's theory of asymptotic spectra \cite{Strassen88}, a novel dual characterization of the Shannon capacity of a graph
is derived in \cite{Zuiddam19}, expressed as the minimum over the elements of its asymptotic spectrum. By confirming that various graph invariants,
including the Lov\'{a}sz $\vartheta$-function \cite{Lovasz79} and the fractional Haemers' bound \cite{BukhC19}, are elements of the asymptotic
spectrum of a graph (spectral points), it can be deduced that these elements indeed serve as upper bounds on the Shannon capacity of a graph.
For further exploration, the comprehensive paper by Wigderson and Zuiddam \cite{WigdersonZ23} provides a survey on Strassen’s theory of the
asymptotic spectra and its application areas, including the Shannon capacity of graphs.

Several further properties of the Shannon capacity have been studied. For example, the recent work \cite{AbiadDF26} analyzes the Shannon capacity
of the distance-$k$ power of a graph using tools from spectral graph theory and linear optimization methods. Despite this progress, many fundamental
questions remain open. Among the most basic is the determination of the Shannon capacity of odd cycle graphs of length greater than~5
\cite{Lovasz79,Bohman05b,BoerBZ24,PolakS19}.

This paper addresses several research directions in the study of the Shannon capacity of graphs, and it is organized as follows.
\begin{itemize}
\item Section~\ref{section:preliminaries} provides preliminaries that are required for the analysis in this paper.
Its focus is on graph invariants, and classes of graphs that are used throughout this paper.
\item Section~\ref{section:polynomial_decomposition} builds on a recent paper by Schrijver \cite{Schrijver23}, and
it explores conditions under which, for a family of graphs, the Shannon capacity of any polynomial in these graphs
equals the corresponding polynomial of their individual Shannon capacities. This equivalence can substantially simplify
the computation of the Shannon capacity for some of structured graphs. Two sufficient conditions are presented, followed
by a comparison of their differences and illustrative examples of their use.
\item Section~\ref{section:tadpole} explores Tadpole graphs. Exact values and bounds on the capacity of Tadpole graphs are
derived, and a direct relation between the capacity of odd-cycles and the capacity of a countably infinite subfamily of
the Tadpole graphs is proved, providing an important property of that subfamily that is further discussed in the following section.
\item Section~\ref{section: unattainability of Shannon capacity at any finite power} determines sufficient conditions for
the unattainability of the Shannon capacity by the independence number of any finite strong power of a graph. It first presents
in an alternative streamlined way an approach by Guo and Watanabe \cite{GuoW90}. It then introduces two other original approaches.
One of the novelties in this section is the construction of a countably infinite family of {\em connected} graphs whose capacity is
unattainable by any finite strong power of these graphs.
\item Section~\ref{section:q-kneser} determines the Shannon capacity of the $q$-Kneser graphs in a manner analogous to the Lov\'{a}sz
calculation of the Shannon capacity of the classical Kneser graphs \cite{Lovasz79}. The derivation relies on two known results, namely
a generalized Erd\H{o}s–Ko–Rado theorem for finite vector spaces and characterizations of the spectrum of the $q$-Kneser graphs.
This broadens the class of graphs for which the exact Shannon capacity is explicitly determined.
\item Section~\ref{section:inequality} introduces a new inequality that relates the Shannon capacity of the strong product of graphs to that 
of their disjoint union, and it identifies several conditions under which equality is attained. As an illustration of the applicability of this 
inequality, the section also derives an alternative proof of a lower bound on the Shannon capacity of the disjoint union of a 
graph and its complement, originally due to Alon~\cite{Alon98}, together with new sufficient conditions for its attainability, as well as 
analogous bounds for the Lov\'{a}sz $\vartheta$-function.
\item Section~\ref{section: outlook} outlines directions for future research that are largely related to the results of this paper.
\end{itemize}


\section{Preliminaries}\label{section:preliminaries}

\subsection{Basic definitions and graph families}\label{subsection:basic_properties}

\subsubsection{Terminology}

Let $\Gr{G} = (\Vertex, \Edge)$ be a graph, where $\Vertex=\V{\Gr{G}}$ is the {\em vertex set} of $\Gr{G}$,
and $\Edge=\E{\Gr{G}}$ is the {\em edge set} of $\Gr{G}$.
\begin{itemize}
\item An {\em undirected graph} is a graph whose edges are undirected.
\item A {\em self-loop} is an edge that connects a vertex to itself.
\item A {\em simple graph} is a graph having no self-loops and no multiple edges between any pair of vertices.
\item A {\em finite graph} is a graph with a finite number of vertices.
\item The {\em order} of a finite graph is the number of its vertices, $\card{\V{\Gr{G}}} = n$.
\item The {\em size} of a finite graph is the number of its edges, $\card{\E{\Gr{G}}} = m$.
\item Vertices $i, j \in \V{\Gr{G}}$ are {\em adjacent} if they are the endpoints of an edge
in $\Gr{G}$, which is denoted by $\{i,j\} \in \E{\Gr{G}}$ or $i \sim j$.
\item An {\em empty graph} is a graph without edges, so its size is equal to zero.
\item The {\em degree of a vertex} $v$ in $\Gr{G}$ is the
number of adjacent vertices to $v$ in $\Gr{G}$, denoted by $\dgr{v} = \dgr{v}(\Gr{G})$.
\item A graph is {\em regular} if all its vertices have an identical degree.
\item A {\em $d$-regular} graph is a regular graph whose all vertices have a fixed degree~$d$.
\item A {\em walk} is a sequence of vertices in a graph $\Gr{G}$, where
every two consecutive vertices in the sequence are adjacent in $\Gr{G}$.
\item A {\em path} is a walk with no repeated vertices.
\item A {\em cycle} $\Cycle$ is obtained from a path $\Path$ by adding an edge connecting
the two endpoints of $\Path$ (i.e., an edge connecting the two degree-1 vertices of $\Path$).

\item The {\em length of a path or a cycle} is equal to its number of edges. A {\em triangle}
is a cycle of length~3.

\item A path on $n$ vertices is denoted by $\PathG{n}$, and its size is equal to $n-1$.

\item A cycle on $n$ vertices is called an $n$-cycle, and it
is denoted by $\CG{n}$ with an integer $n \geq 3$. The order and size of $\CG{n}$ are both equal to $n$.

\item A {\em connected graph} is a graph where every two distinct vertices are connected by a path.

\item An {\em $r$-partite graph} is a graph whose vertex set is a disjoint union of $r$ subsets
such that no two vertices in the same subset are adjacent. If $r=2$, then $\Gr{G}$ is a {\em bipartite graph}
(e.g., $\CG{n}$ is a bipartite graph if and only if $n \geq 4$ is even).

\item A {\em complete graph} on $n$ vertices, denoted by $\CoG{n}$, is a graph whose all $n$
distinct vertices are pairwise adjacent. Hence, $\CoG{n}$ is an $(n-1)$-regular graph of order $n$.

\item A {\em complete $r$-partite graph}, denoted by $\Complete_{n_1, \ldots, n_r}$ with
$n_1, \ldots n_r \in \naturals$, is an $r$-partite graph
whose vertex set is partitioned into $r$ disjoint subsets of cardinalities $n_1, \ldots, n_r$,
such that every two vertices in the same subset are not adjacent, and every two vertices in
distinct subsets are adjacent.
\end{itemize}
Throughout this paper, graphs are finite, simple, and undirected.
We also use the standard notation $\OneTo{n} \eqdef \{1, \ldots, n\}$ for all $n \in \naturals$.

\begin{definition}[Subgraphs and graph connectivity]
\label{definition:subgraphs}
A graph $\Gr{F}$ is a {\em subgraph} of a graph $\Gr{G}$, and it is
denoted by $\Gr{F} \subseteq \Gr{G}$, if
$\V{\Gr{F}} \subseteq \V{\Gr{G}}$ and $\E{\Gr{F}} \subseteq \E{\Gr{G}}$.
\begin{itemize}
\item A {\em spanning subgraph} of $\Gr{G}$ is obtained by edge deletions
from $\Gr{G}$, while its vertex set is left unchanged.
A {\em spanning tree} in $\Gr{G}$ is a spanning subgraph of $\Gr{G}$ that forms a tree.
\item An {\em induced subgraph} is obtained by removing vertices
from the original graph, followed by the deletion of their incident edges.
\end{itemize}
\end{definition}

\begin{definition}[Isomorphic graphs]
\label{definition:isomorphic graphs}
Graphs $\Gr{G}$ and $\Gr{H}$ are {\em isomorphic} if there exists a bijection
$f \colon \V{\Gr{G}} \to \V{\Gr{H}}$ (i.e., a one-to-one and onto mapping) such that
$\{i,j\} \in \E{\Gr{G}}$ if and only if $\{f(i), \, f(j)\} \in \E{\Gr{H}}$.
It is denoted by $\Gr{G} \cong \Gr{H}$, and $f$ is said to be an {\em isomorphism}
from $\Gr{G}$ to $\Gr{H}$.
\end{definition}

\begin{definition}[Complement and self-complementary graphs]
\label{definition:complement and s.c. graphs}
The {\em complement} of a graph $\Gr{G}$, denoted by $\CGr{G}$, is a graph
whose vertex set is $\V{\Gr{G}}$, and its edge set is the complement set $\overline{\E{\Gr{G}}}$.
Every vertex in $\V{\Gr{G}}$ is nonadjacent to itself in $\Gr{G}$ and $\CGr{G}$, so
$\{i,j\} \in \E{\CGr{G}}$ if and only if $\{i, j\} \notin \E{\Gr{G}}$ with $i \neq j$.
A graph $\Gr{G}$ is {\em self-complementary} if $\Gr{G} \cong \CGr{G}$ (i.e., $\Gr{G}$
is isomorphic to $\CGr{G}$).
\end{definition}

\begin{example}
\label{example: self-complementary}
It can be verified that $\PathG{4}$ and $\CG{5}$ are self-complementary graphs.
\end{example}

\subsubsection{Graph operations}
This subsection presents the basic graph operations used throughout this paper.

\begin{definition}[Strong product of graphs]
\label{definition:strong_product}
Let $\Gr{G}$ and $\Gr{H}$ be simple graphs. The strong product $\Gr{G}\boxtimes \Gr{H}$ is a graph whose vertices set is $\V{\Gr{G}}\times \V{\Gr{H}}$,
and two distinct vertices $(g_1,h_1),(g_2,h_2)$ are adjacent if, in each coordinate, they are either equal or adjacent. This means that one of the following
three conditions is needed to be satisfied:
\begin{enumerate}
  \item $g_1 = g_2$ and $\{h_1,h_2\}\in \E{\Gr{H}}$,
  \item $\{g_1,g_2\}\in \E{\Gr{G}}$ and $h_1 = h_2$,
  \item $\{g_1,g_2\}\in \E{\Gr{G}}$ and $\{h_1,h_2\}\in \E{\Gr{H}}$.
\end{enumerate}
\end{definition}
The interested reader is referred to \cite{HammackIK11} for an extensive textbook on graph products and their properties.

Define the {\em $k$-fold strong power} of $\Gr{G}$ as the strong product of $k$ copies of $\Gr{G}$:
\begin{align}
  \Gr{G}^k \triangleq \Gr{G} \boxtimes \ldots \boxtimes \Gr{G}.
\end{align}

Throughout this paper, we also use the disjoint union of graphs as the addition operation on graphs.
\begin{definition}[Disjoint union of graphs]
\label{definition:disjoint_union}
Let $\Gr{G}$ and $\Gr{H}$ be two simple graphs. The disjoint union $\Gr{G}+\Gr{H}$ is a graph whose vertices set is
$\V{\Gr{G}}\cup \V{\Gr{H}}$, and the edges set is $\E{\Gr{G}}\cup \E{\Gr{H}}$.
Also, for $m\in\naturals$, let
\begin{align}
  m \Gr{G} \triangleq \Gr{G} + \ldots + \Gr{G},
\end{align}
which is defined as the disjoint union of $m$ copies of $\Gr{G}$.
\end{definition}

\subsubsection{Basic graph invariants under isomorphism}

\begin{definition}[Independent sets]
\label{definition:independent_set}
  Let $\Gr{G}$ be a simple graph. Define
  \begin{itemize}
    \item A set $\Independentset\subseteq\V{\Gr{G}}$ is an independent set in $\Gr{G}$ if every pair of vertices in $\Independentset$ are nonadjacent in $\Gr{G}$.
    \item The set $\indset{\Gr{G}}$ is the set of independent sets in $\Gr{G}$.
  \end{itemize}
\end{definition}
Now we define the independence number of a graph.
\begin{definition}[Independence number]
\label{definition:independence_number}
  The independence number of a graph $\Gr{G}$, denoted by $\indnum{\Gr{G}}$, is the order of a largest independent set in $\Gr{G}$, i.e.
  \begin{align}\label{eq:independence_number}
    \indnum{\Gr{G}}\triangleq\max\{\card{\Independentset}:\Independentset\in\indset{\Gr{G}}\}.
  \end{align}
\end{definition}

\begin{definition}[Cliques]
\label{definition:clique}
  Let $\Gr{G}$ be a simple graph. Define
  \begin{itemize}
    \item A set $\Cliqueset\subseteq\V{\Gr{G}}$ is a clique in $\Gr{G}$ if every pair of vertices in $\Cliqueset$ are adjacent in $\Gr{G}$.
    \item The set $\clset{\Gr{G}}$ is the set of cliques in $\Gr{G}$.
  \end{itemize}
\end{definition}
\begin{definition}[Clique number]
\label{definition:clique_number}
  The clique number of a graph $\Gr{G}$, denoted by $\clnum{\Gr{G}}$, is the order of a largest clique in $\Gr{G}$, i.e.
  \begin{align}\label{eq:clique_number}
    \clnum{\Gr{G}}\triangleq\max\{\card{\Cliqueset}:\Cliqueset\in\clset{\Gr{G}}\}.
  \end{align}
\end{definition}

\begin{definition}[Chromatic number]
\label{definition:chromatic_number}
  The chromatic number of a graph $\Gr{G}$, denoted by $\chrnum{\Gr{G}}$, is the smallest cardinality of a partition of the vertex set of
  $\Gr{G}$ into independent sets. Equivalently, it is the minimum number of colors needed to color the vertices of $\Gr{G}$ such that the
  endpoints of every edge receive distinct colors.
\end{definition}

\begin{definition}[Clique-cover number]
\label{definition:clique-cover_number}
  The clique-cover number of $\Gr{G}$, denoted by $\clcnum{\Gr{G}}$, is the smallest number of cliques needed to cover all the vertices of $\Gr{G}$.
  Hence, $\clcnum{\Gr{G}} = \chrnum{\CGr{G}}$.
\end{definition}

Next, we provide required properties of the independence number of a graph.

\begin{theorem}\label{theorem: independence_number_strong_product}
  Let $\Gr{G}$ and $\Gr{H}$ be simple graphs. Then,
  \begin{align}
    & \indnum{\Gr{G}\boxtimes\Gr{H}}\geq\indnum{\Gr{G}} \, \indnum{\Gr{H}},   \label{eq:independence_number_strong_product} \\[0.1cm]
    & \indnum{\Gr{G}+\Gr{H}}=\indnum{\Gr{G}}+\indnum{\Gr{H}}. \label{eq:independence_number_disjoint_union}
  \end{align}
\end{theorem}

\begin{theorem}\label{theorem: independence_number_subgraphs}
  Let $\Gr{G}$ be a simple graph, and let $\Gr{H}_1$ and $\Gr{H}_2$ be induced and spanning subgraphs of $\Gr{G}$, respectively. Then,
  \begin{align}\label{eq:independence_number_subgraphs}
    \indnum{\Gr{H}_1} \leq \indnum{\Gr{G}} \leq \indnum{\Gr{H}_2}.
  \end{align}
  Furthermore, for every $k\in\naturals$,
  \begin{align}
    \indnum{\Gr{H}_1^k}\leq\indnum{\Gr{G}^k}\leq\indnum{\Gr{H}_2^k}.
  \end{align}
\end{theorem}

\subsubsection{Fractional invariants of graphs}
To properly define four fractional invariants of a graph, the following sets of functions are first introduced.
\begin{definition}\label{definition:fractional_sets}
  Let $\Gr{G}$ be a simple graph. Define four sets of functions as follows:
  \begin{itemize}
    \item $\set{F}_I(\Gr{G})$ is the set of non-negative functions, $f \colon \V{\Gr{G}} \to \Reals$, such that for
    every $\set{I} \in \indset{\Gr{G}}$,
        \begin{align}\label{eq:fractionl_set_independent_small}
          \sum_{v \in \set{I}} f(v) \leq 1.
        \end{align}
    \item $\set{F}_C(\Gr{G})$ is the set of non-negative functions, $f \colon \V{\Gr{G}} \to \Reals$, such that for
    every $\set{C} \in \clset{\Gr{G}}$,
        \begin{align}\label{eq:fractionl_set_clique_small}
          \sum_{v \in \set{C}} f(v) \leq 1.
        \end{align}
    \item $\set{G}_I(\Gr{G})$ is the set of non-negative functions, $g \colon \indset{\Gr{G}} \to \Reals$, such
    that for every $v \in \V{\Gr{G}}$,
        \begin{align}\label{eq:fractionl_set_independent_big}
          \sum_{\set{I} \in \indset{\Gr{G}}: \, v \in \set{I}} g(\set{I}) \geq 1,
        \end{align}
        where the summation on the left-hand side of \eqref{eq:fractionl_set_independent_big} is over all independent
        sets in $\Gr{G}$ that include the vertex $v$.
    \item $\set{G}_C(\Gr{G})$ is the set of non-negative functions, $g \colon \clset{\Gr{G}}\to\Reals$, such that for
    every $v\in\V{\Gr{G}}$,
        \begin{align}\label{eq:fractionl_set_clique_big}
          \sum_{\set{C} \in \clset{\Gr{G}}: \, v \in \set{C}} g(\set{C}) \geq 1,
        \end{align}
        where the summation on the left-hand side of \eqref{eq:fractionl_set_clique_big} is over all cliques in $\Gr{G}$
        that include the vertex $v$.
  \end{itemize}
\end{definition}

Next, four basic fractional invariants of graphs are defined by using linear programming.
\begin{definition}[Fractional invariants of graphs]
\label{definition:fractional_invariants}
  For a finite, undirected, and simple graph $\Gr{G}$, fractional invariants of graphs are defined as follows.
  \begin{itemize}
    \item The fractional independence number of $\Gr{G}$ is
    \begin{align}\label{eq:fractional_independence_number}
      \findnum{\Gr{G}} = \sup \Bigg\{\sum_{v \in \V{\Gr{G}}} f(v): \, f \in \set{F}_C(\Gr{G}) \Bigg\}.
    \end{align}
    Namely, $\findnum{\Gr{G}}$ is the supremum of the total weight assigned to the vertices of $\Gr{G}$,
    where nonnegative real weights are assigned so that for every clique of $\Gr{G}$, the sum of the weights
    of the vertices in that clique is at most~1.
    \item The fractional clique-cover number of $\Gr{G}$ is
    \begin{align}\label{eq:fractional_clique_cover_number}
      \fclcnum{\Gr{G}} = \inf \Bigg\{\sum_{\set{C} \in \clset{\Gr{G}}} g(\set{C}): \, g \in \set{G}_C(\Gr{G}) \Bigg\}.
    \end{align}
    Namely, $\fclcnum{\Gr{G}}$ is the infimum of the total weight assigned to the cliques of $\Gr{G}$, where nonnegative
    real weights are assigned so that for every vertex of $\Gr{G}$, the sum of the weights of all cliques containing that
    vertex is at least~1.
    \item The fractional clique number of $\Gr{G}$ is
    \begin{align}\label{eq:fractional_clique_number}
      \fclnum{\Gr{G}} = \sup \Bigg\{\sum_{v \in \V{\Gr{G}}} f(v): \, f \in \set{F}_I(\Gr{G}) \Bigg\}.
    \end{align}
    Namely, $\fclnum{\Gr{G}}$ is the supremum of the total weight assigned to the vertices of $\Gr{G}$, where nonnegative
    real weights are assigned so that for every independent set of $\Gr{G}$, the sum of the weights of the vertices in that
    independent set is at most~1.
    \item The fractional chromatic number of $\Gr{G}$, denoted by $\fchrnum{\Gr{G}}$, is given by
    \begin{align}\label{eq:fractional_chromatic_number}
      \fchrnum{\Gr{G}} = \inf \Bigg\{\sum_{\set{I} \in \indset{\Gr{G}}} g(\set{I}): \, g\in\set{G}_I(\Gr{G})\Bigg\}.
    \end{align}
    Namely, $\fchrnum{\Gr{G}}$ is the infimum of the total weight assigned to the independent sets of $\Gr{G}$, where
    nonnegative real weights are assigned so that for every vertex of $\Gr{G}$, the sum of the weights of all independent
    sets containing that vertex is at least~1.
  \end{itemize}
\end{definition}

The next theorem follows from strong duality in linear programming.
\begin{theorem}\label{theorem: fractional_invariants_strong_duality}
\cite{ScheinermanU13}
  Let $\Gr{G}$ be a simple graph. Then,
  \begin{align}
    & \findnum{\Gr{G}} = \fclcnum{\Gr{G}}, \\
    & \fchrnum{\Gr{G}} = \fclnum{\Gr{G}}.
  \end{align}
\end{theorem}
From now on, we will primarily use the fractional independence number and the fractional chromatic number.
When computing these quantities, we will use both of their equivalent linear programming formulations.

Another useful property, which is immediate from Definition~\ref{definition:fractional_invariants}, is given next.
\begin{theorem}\label{theorem: fractional_invariants_complement_equality}
\cite{ScheinermanU13}
  For a simple graph $\Gr{G}$, the following holds:
  \begin{align}
    \findnum{\Gr{G}} = \fclcnum{\Gr{G}} = \fchrnum{\CGr{G}} = \fclnum{\CGr{G}}.
  \end{align}
\end{theorem}

Some properties of the fractional independence number are presented next.
\begin{theorem}\label{theorem: fractional_independence_strong_product}
  \cite{Hales73} Let $\Gr{G}$ and $\Gr{H}$ be simple graphs. Then,
  \begin{align}\label{eq:fractional_independence_strong_product}
    & \findnum{\Gr{G}\boxtimes\Gr{H}} = \findnum{\Gr{G}} \, \findnum{\Gr{H}}, \\[0.1cm]
    & \indnum{\Gr{G}\boxtimes\Gr{H}}\leq\findnum{\Gr{G}} \, \indnum{\Gr{H}}.
    \end{align}
\end{theorem}

\subsubsection{Graph spectrum}
\label{subsection: Graph spectrum}

\begin{definition}[Adjacency matrix]
\label{definition: adjacency matrix}
Let $\Gr{G}$ be a simple undirected graph on $n$ vertices. The {\em adjacency matrix} of $\Gr{G}$,
denoted by $\A = \A(\Gr{G})$, is an $n \times n$ symmetric matrix $\A = (\mathrm{A}_{i,j})$
where $\mathrm{A}_{i,j} = 1$ if $\{i,j\} \in \E{\Gr{G}}$, and $\mathrm{A}_{i,j}=0$ otherwise
(so the entries in the principal diagonal of $\A$ are zeros).
\end{definition}

\begin{definition}[Graph spectrum]
\label{definition: graph spectrum}
Let $\Gr{G}$ be a simple undirected graph on $n$ vertices. The {\em spectrum} of $\Gr{G}$ is defined
as the multiset of eigenvalues of the adjacency matrix of $\Gr{G}$.
\end{definition}

\subsubsection{Some structured families of graphs}
Vertex- and edge-transitivity, defined as follows, play an important role in characterizing graphs.

\begin{definition}[Automorphism]
\label{definition:automorphism}
An {\em automorphism} of a graph $\Gr{G}$ is an isomorphism from $\Gr{G}$ to itself.
\end{definition}

\begin{definition}[Vertex-transitivity]
\label{definition:vertex-transitive graphs}
A graph $\Gr{G}$ is said to be {\em vertex-transitive} if for every two vertices $i, j \in \V{\Gr{G}}$,
there is an automorphism $f \colon \V{\Gr{G}} \to \V{\Gr{G}}$ such that $f(i) = j$.
\end{definition}

\begin{definition}[Edge-transitivity]
\label{definition:edge-transitive graphs}
A graph $\Gr{G}$ is {\em edge-transitive} if for every two edges $e_1, e_2 \in \E{\Gr{G}}$,
there is an automorphism $f \colon \V{\Gr{G}} \to \V{\Gr{G}}$ that maps the endpoints of $e_1$ to
the endpoints of $e_2$.
\end{definition}

\begin{definition}[Kneser graphs]
\label{definition:kneser_graph}
  Let $\OneTo{n}$ be the set with natural numbers from $1$ to $n$, and let $1\leq r \leq n$.
  The Kneser graph $\KG{n}{r}$ is the graph whose vertex set is composed of the different
  $r$-subsets of $\OneTo{n}$, and every two vertices $u,v$ are adjacent if and only if the
  respective $r$-subsets are disjoint.
\end{definition}
Kneser graphs are vertex- and edge-transitive.

\begin{definition}[Perfect graphs]
\label{definition:perfect_graph}
    A graph $\Gr{G}$ is perfect if for every induced subgraph $\Gr{H}$ of $\Gr{G}$,
    \begin{align} \label{eq: perfect graphs}
      \clnum{\Gr{H}}=\chrnum{\Gr{H}}.
    \end{align}
\end{definition}

\begin{definition}[Universal graphs]
\label{definition:universal_graph}
    A graph $\Gr{G}$ is universal if for every graph~$\Gr{H}$,
    \begin{align}  \label{eq: universal graphs}
      \indnum{\Gr{G}\boxtimes\Gr{H}}=\indnum{\Gr{G}} \, \indnum{\Gr{H}}.
    \end{align}
\end{definition}

\begin{lemma}\label{lemma:universal_power}
    If $\Gr{G}$ is a universal graph, then $\Gr{G}^k$ is universal for all $k\in\naturals$.
\end{lemma}
\begin{proof}
This follows easily by Definition~\ref{definition:universal_graph} and mathematical induction on~$k$.
\end{proof}

A corollary by Hales \cite{Hales73} regarding the connection between the graph universality and fractional independence number is presented next.
\begin{theorem}\label{theorem: universal_findnum}
  \cite{Hales73} A graph $\Gr{G}$ is universal if and only if $\indnum{\Gr{G}}=\findnum{\Gr{G}}$.
\end{theorem}

\begin{definition}[Strongly regular graphs]
\label{definition:strongly_regular_graph}
  A graph $\Gr{G}$ is strongly regular with parameters $\srg{n}{d}{\lambda}{\mu}$ if it satisfies the following:
  \begin{itemize}
    \item The order of $\Gr{G}$ is $n$.
    \item $\Gr{G}$ is $d$-regular.
    \item Every pair of adjacent vertices has exactly $\lambda$ common neighbors.
    \item Every pair of distinct, nonadjacent vertices has exactly $\mu$ common neighbors.
  \end{itemize}
\end{definition}

\begin{definition}[Paley graphs]
\label{definition:paley_graph}
  Let $q=p^n$ be a prime power with $p$ prime, $n\in\naturals$, and $q\equiv 1\text{ mod 4}$.
  The Paley graph of order $q$, denoted by $P(q)$, is defined as follows:
  \begin{itemize}
    \item The vertex set of $P(q)$ is $\mathbb{F}_{q} = \{0,1,\ldots,q-1\}$.
    \item Two distinct vertices $a, b \in \mathbb{F}_{q}$ are adjacent if $a-b$ is a square in $\mathbb{F}_{q}$.
  \end{itemize}
\end{definition}

\begin{theorem}\label{theorem: paley_graph_properties}  \cite{GodsilR}
  Let $q\equiv 1 \text{ mod 4}$ be a prime power, and let $P(q)$ be a Paley graph. Then, the following hold:
  \begin{itemize}
    \item $P(q)$ is a self-complementary graph.
    \item $P(q)$ is strongly regular with parameters $\srg{q}{\frac{q-1}{2}}{\frac{q-5}{4}}{\frac{q-1}{4}}$.
    \item $P(q)$ is vertex-transitive.
    \item $P(q)$ is edge-transitive.
  \end{itemize}
\end{theorem}

\subsection{The Shannon capacity of graphs}
\label{subsection:shannon_capacity}
The Shannon capacity of a graph $\Gr{G}$ was introduced in \cite{Shannon56} as a measure of the largest effective
alphabet size permitting zero-error communication over a given noisy channel. A discrete memoryless
channel (DMC) with finite input and output alphabets is specified by a finite input set $\set{X}$, a finite output set $\set{Y}$,
and a non-empty fan-out set $\set{S}_x \triangleq \bigl\{y \in \set{Y}: \, p_{Y|X}(y|x) > 0 \bigr\}$ for each $x \in \set{X}$.
In each channel use, the sender transmits an input $x \in \set{X}$, and the receiver observes an arbitrary output
in the corresponding fan-out set $\set{S}_x$.
The channel can be represented by a \textit{confusion graph} $\Gr{G}$ that is defined as follows:
\begin{itemize}
  \item $\V{\Gr{G}} = \set{X}$ represents the set of symbols of the channel’s input alphabet.
  \item $\E{\Gr{G}}$ is the edge set of $\Gr{G}$, where two distinct vertices are adjacent if the corresponding
  two input symbols from $\set{X}$ are not distinguishable by the channel, i.e., they can produce an identical output symbol with positive probability.
  In other words, two vertices in $\Gr{G}$ are adjacent if the fan-out sets of the corresponding input symbols intersect, so
  \begin{align}
  \E{\Gr{G}} = \Bigl\{\{x,x'\}: \, x, x' \in \set{X}, \; x \neq x', \; \set{S}_x \cap \set{S}_{x'} \neq \es \Bigr\}.
  \end{align}
\end{itemize}

The largest cardinality of a set of input symbols that can be communicated without error in a single use of the channel equals the independence
number $\indnum{\Gr{G}}$. Indeed, in this single-use setting, the sender and the receiver agree in advance on an independent set $\set{I}$ of a
maximum size $\indnum{\Gr{G}}$, the sender transmits only inputs in $\set{I}$, every received output is in the fan-out set of exactly one input
in $\set{I}$, and the receiver can correctly determine the transmitted input.

Consider the transmission of $k$-length strings over a channel, where the channel is used $k\geq1$ times. The sender transmits a sequence
$x_1, \ldots, x_k$, and the receiver observes an output sequence $y_1, \ldots, y_k$, where $y_i \in \set{S}_{x_i}$ for all $i \in \OneTo{k}$.
In this setup, $k$ uses of the channel are viewed as a single use of a super-channel whose input alphabet is $\set{X}^k$, its output alphabet is $\set{Y}^k$,
and the fan-out set of $(x_1, \ldots, x_k) \in \set{X}^k$ is given by the Cartesian product $\set{S}_{x_1} \times \ldots \times\set{S}_{x_k}$.
Note that two distinct sequences $(x_1, \ldots, x_k), (x_1', \ldots, x_k')\in\set{X}^k$ are distinguishable by the channel if and only if
\begin{align}
  S_{x_i}\cap S_{x_i'}=\es
\end{align}
for {\em some} index $i \in \OneTo{k}$.
Thus, it is possible to represent the super-channel by the $k$th confusion graph that is defined as the $k$-fold strong power of $\Gr{G}$.
Using the $k$th confusion graph, the largest amount of information that can be sent by $k$ uses of the channel with error-free communication
is given by the independence number of the $k$th confusion graph, i.e., it is equal to $\indnum{\Gr{G}^k}$.
The maximum information rate per symbol that is achieved by $k$ uses of the channel, and with zero-error communication, is equal to
\begin{align}
  \frac{1}{k} \log \indnum{\Gr{G}^k} = \log\sqrt[k]{\indnum{\Gr{G}^k}}, \quad k \in \naturals.
\end{align}
Finally, by omitting the logarithm (as a monotonically increasing function) and supremizing over $k \in \naturals$, the Shannon capacity of the graph $\Gr{G}$
is defined as follows.

\begin{definition}[Shannon capacity]
\label{definition:shannon_capacity}
  Let $\Gr{G}$ be a simple graph. The Shannon capacity of $\Gr{G}$ is defined as
  \begin{align}
    \Theta(\Gr{G}) & \triangleq \sup_{k \in \naturals}\sqrt[k]{\indnum{\Gr{G}^k}} \label{eq:shannon_capacity} \\
    &= \lim_{k \to \infty} \sqrt[k]{\indnum{\Gr{G}^k}}, \label{eq2:shannon_capacity}
  \end{align}
  where equality \eqref{eq2:shannon_capacity} holds by inequality~\eqref{eq:independence_number_strong_product} and Fekete's lemma.
\end{definition}

Next, we provide several properties of the Shannon capacity, which are used throughout this paper.
\begin{theorem}\label{theorem: shannon_power}
  Let $\Gr{G}$ be a simple graph and let $m \in \naturals$. Then,
  \begin{align}
    \Theta(\Gr{G}^m) = \Theta(\Gr{G})^m. \label{eq:shannon_power}
  \end{align}
\end{theorem}
\begin{proof}
Equality~\eqref{eq:shannon_power} holds by \eqref{eq2:shannon_capacity}.
\end{proof}

\begin{theorem}\label{theorem: shannon_strong_product}
  Let $\Gr{G}_1$ and $\Gr{G}_2$ be simple graphs. Then,
  \begin{align}\label{eq:shannon_strong_product}
    \Theta(\Gr{G}_1\boxtimes\Gr{G}_2) \geq \Theta(\Gr{G}_1) \, \Theta(\Gr{G}_2).
  \end{align}
\end{theorem}

\begin{theorem}[Shannon's inequality]
\label{theorem: shannon_disjoint_union}
  \cite{Shannon56} Let $\Gr{G}_1$ and $\Gr{G}_2$ be simple graphs. Then,
  \begin{align}\label{eq:shannon_disjoint_union}
    \Theta(\Gr{G}_1 + \Gr{G}_2) \geq \Theta(\Gr{G}_1) + \Theta(\Gr{G}_2).
  \end{align}
\end{theorem}
\begin{proof}
An elegant proof of Shannon's inequality \eqref{eq:shannon_disjoint_union} is presented in \cite{Schrijver23}.
\end{proof}

In light of Definition~\ref{definition:disjoint_union}, if $\Gr{G}$ and $\Gr{H}$ are the confusion graphs of two
channels, then their disjoint union represents the sum of the channels corresponding to the situation where either
one of the two channels may be used, a new choice being made for each transmitted symbol. The Shannon capacity of the disjoint union
of graphs may be strictly larger than the sum of their capacities \cite{Alon98}, disproving Shannon's conjecture in the following
strong sense.
\begin{theorem}[The Shannon capacity of disjoint union of graphs] \cite[Theorem~1.1]{Alon98}
\label{theorem: capacity of union}
For every $k \in \naturals$, there is a graph $\Gr{G}$ so that the Shannon capacity of the graph and that of its complement $\CGr{G}$
satisfy $\Theta(\Gr{G}) \leq k$ and $\Theta(\CGr{G}) \leq k$, whereas
\begin{align}
\Theta(\Gr{G} + \CGr{G}) \geq k^{(1+o(1)) \, \frac{\ln k}{8 \ln \ln k}},
\end{align}
and the $o(1)$-term tends to zero as we let $k$ tend to infinity.
\end{theorem}
The following known result is central in this paper, so its short and nice proof from \cite{Schrijver23} is included here.
\begin{theorem}[Duality theorem]
\label{theorem: shannon_union_equivalent_product}
  \cite{WigdersonZ23,Schrijver23} Let $\Gr{G}$ and $\Gr{H}$ be simple graphs. Then,
  \begin{align}\label{eq:shannon_union_equivalent_product}
    \Theta(\Gr{G}+\Gr{H}) = \Theta(\Gr{G}) + \Theta(\Gr{H}) \iff \Theta(\Gr{G}\boxtimes \Gr{H}) = \Theta(\Gr{G}) \, \Theta(\Gr{H}).
  \end{align}
\end{theorem}
\begin{proof}
See Appendix~\ref{appendix: proof of duality theorem}.
\end{proof}

\begin{theorem}\label{theorem: shannon_scalar}
  Let $\Gr{G}$ be a simple graph and let $m \in \naturals$. Then,
  \begin{align}
    \Theta(m \Gr{G}) = m \, \Theta(\Gr{G}).  \label{eq:shannon_scalar}
  \end{align}
\end{theorem}
\begin{proof}
Equality~\eqref{eq:shannon_scalar} follows from Theorem~\ref{theorem: shannon_union_equivalent_product} by
relying on equality \eqref{eq:shannon_power}.
\end{proof}

\begin{theorem}\label{theorem: shannon_capacity_subgraphs}
  Let $\Gr{G}$ be a simple graph, and let $\Gr{H}_1$ and $\Gr{H}_2$ be induced and spanning subgraphs of $\Gr{G}$, respectively. Then,
  \begin{align}\label{eq:shannon_capacity_subgraphs}
    \Theta(\Gr{H}_1) \leq \Theta(\Gr{G}) \leq \Theta(\Gr{H}_2).
  \end{align}
\end{theorem}
\begin{proof}
This holds by the definitions of induced and spanning subgraphs, together with \eqref{eq:shannon_capacity}.
\end{proof}

\subsection{The Lov\'{a}sz \texorpdfstring{$\vartheta$}{theta}-function of graphs}\label{subsection:lovasz_function}
Before we present the Lov\'{a}sz $\vartheta$-function of graphs, we define an orthogonal representation of a graph \cite{Lovasz79} (see also \cite[Chapter~11]{Lovasz19}).
\begin{definition}[Orthogonal representations]
\label{definition:orthogonal_representation}
  Let $\Gr{G}$ be a simple graph. An {\em orthogonal representation} of $\Gr{G}$ in $\Reals^d$ assigns each vertex $i\in \V{\Gr{G}}$ to a nonzero vector $u_i\in\Reals^d$ such that, for every distinct nonadjacent vertices $i,j\in \V{\Gr{G}}$, the vectors $u_i,u_j$ are orthogonal.
  An orthogonal representation is called an {\em orthonormal representation} if all the representing vectors of $\Gr{G}$ have a unit length.
\end{definition}
In an orthogonal representation of a graph $\Gr{G}$, distinct nonadjacent vertices are mapped to orthogonal vectors, but adjacent vertices may not be necessarily mapped to non-orthogonal vectors. If the latter condition also holds, then it is called a {\em faithful} orthogonal representation.

\begin{definition}[Lov\'{a}sz $\vartheta$-function]
\label{definition:lovasz_function}
  Let $\Gr{G}$ be a simple graph of order $n$. The Lov\'{a}sz $\vartheta$-function of $\Gr{G}$ is defined as
  \begin{align}\label{eq:lovasz_function}
    \vartheta(\Gr{G}) \triangleq \min_{u,\mathbf{c}}\max_{1\leq i \leq n}\frac{1}{(\mathbf{c}^{\mathrm{T}} \mathbf{u}_i)^2},
  \end{align}
  where the minimum is taken over all orthonormal representations $\{\mathbf{u}_i:i\in \V{\Gr{G}}\}$ of $\Gr{G}$ and
  all unit vectors $\mathbf{c}$.
  The unit vector $\mathbf{c}$ attaining the minimum is called the handle of the orthonormal representation.
\end{definition}
An orthonormal representation of the pentagon $\CG{5}$, along with its handle $\mathbf{c}$, is shown in Figure~\ref{figure: Lovasz umbrella}.
This figure is reproduced, with permission of the author \cite{Lovasz26}, by combining Figures~1.3 and~1.4 in \cite{Lovasz19}.

\begin{figure}[ht]
      \centering
      \begin{minipage}[b]{0.3\textwidth}
        \centering
        \includegraphics[width=1\textwidth]{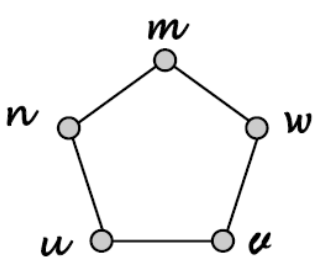}
      \end{minipage}
      \hspace*{0.3cm}
      \begin{minipage}[b]{0.5\textwidth}
        \centering
        \includegraphics[width=1\textwidth]{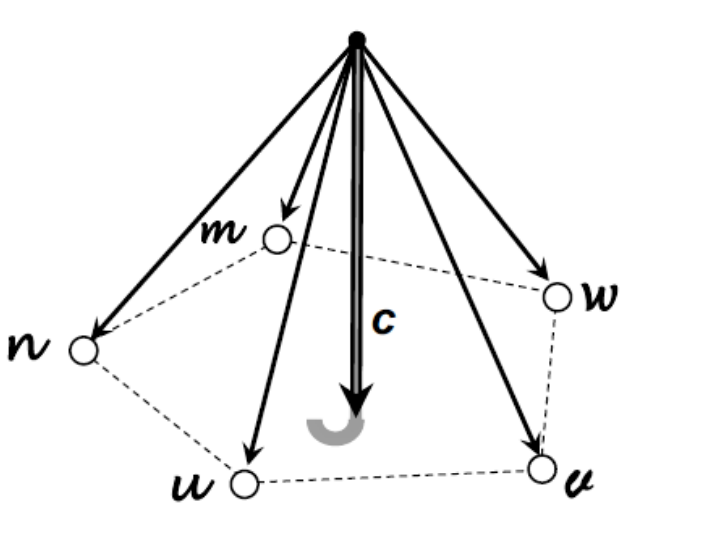}
      \end{minipage}
      \caption{A $5$-cycle graph and its orthonormal representation (Lov\'{a}sz umbrella).}
      \label{figure: Lovasz umbrella}
    \end{figure}

The Lov\'{a}sz $\vartheta$-function can be expressed as a solution of a semidefinite programming (SDP)
problem. To that end, let ${\bf{A}} = (A_{i,j})$ be the $n \times n$ adjacency matrix of $\Gr{G}$ with
$n \triangleq \card{\V{\Gr{G}}}$. The Lov\'{a}sz
$\vartheta$-function $\vartheta(\Gr{G})$ can be expressed by the following convex optimization problem:
\vspace*{0.1cm}
\begin{eqnarray}
\label{eq: SDP problem - Lovasz theta-function}
\mbox{\fbox{$
\begin{array}{l}
\text{maximize} \; \; \mathrm{Tr}({\bf{B}} \, {\bf{J}}_n)  \\
\text{subject to} \\
\begin{cases}
{\bf{B}} \succeq 0, \\
\mathrm{Tr}({\bf{B}}) = 1, \\
A_{i,j} = 1  \; \Rightarrow \;  B_{i,j} = 0, \quad i,j \in \OneTo{n}.
\end{cases}
\end{array}$}}
\end{eqnarray}
The SDP formulation in \eqref{eq: SDP problem - Lovasz theta-function} yields
the existence of an algorithm that computes $\vartheta(\Gr{G})$, for every graph $\Gr{G}$, with a
precision of $r$ decimal digits, and a computational complexity that is polynomial in $n$ and $r$.
Thus, the Lov\'{a}sz $\vartheta$-function can be computed in polynomial time in $n$, and by \cite{Lovasz79},
it is an upper bound on the Shannon capacity, whose computation requires the computation of an infinite
series of independence number, which is a known NP-hard problem.

Adding the inequality constraints $B_{i,j} \geq 0$ for all $i,j \in \OneTo{n}$ to \eqref{eq: SDP problem - Lovasz theta-function}
yields the Schrijver $\vartheta$-function of $\Gr{G}$, denoted by $\vartheta'(\Gr{G})$, which therefore yields
\begin{align}\label{eq:lovasz_bigger_schrijver}
  \vartheta'(\Gr{G})\leq\vartheta(\Gr{G}).
\end{align}
In light of \eqref{eq:lovasz_bigger_schrijver}, one may ask whether the Schrijver $\vartheta$-function can serve as a tighter upper bound on the Shannon capacity of graphs.
However, a recent paper by the second author \cite{Sason25} shows that the Schrijver $\vartheta$-function does not, in general, provide an upper bound on the Shannon capacity:
an explicit graph is constructed whose Shannon capacity is strictly larger than its Schrijver $\vartheta$-value.
Nevertheless, the Schrijver $\vartheta$-function is an upper bound on the independence number, and can therefore be used to derive an improved bound on the independence number.

Next, we provide an alternative representation of the Lov\'{a}sz $\vartheta$-function of a graph \cite{Lovasz79}.
\begin{theorem}\label{theorem: lovasz_equivalent_complement_formula}
  \cite{Lovasz79} Let $(\mathbf{v}_1,\cdots,\mathbf{v}_n)$ range over all orthonormal representations of $\CGr{G}$,
  and let $\mathbf{d}$ range over all unit vectors. Then,
  \begin{align}\label{eq:lovasz_equivalent_complement_formula}
    \vartheta(\Gr{G}) = \max\sum_{i=1}^{n}(\mathbf{d}^{\mathrm{T}}\mathbf{v}_i)^2.
  \end{align}
\end{theorem}

Next, we provide several properties of the Lov\'{a}sz $\vartheta$-function, regarding graph operations and subgraphs.
\begin{theorem}\label{theorem: lovasz_strong_product}
  \cite{Lovasz79} Let $\Gr{G}_1$ and $\Gr{G}_2$ be simple graphs. Then,
  \begin{align}\label{eq:lovasz_strong_product}
    \vartheta(\Gr{G}_1\boxtimes\Gr{G}_2) = \vartheta(\Gr{G}_1) \, \vartheta(\Gr{G}_2).
  \end{align}
\end{theorem}

The following result was first stated and proved by Knuth (Section~18 of \cite{Knuth94}).
We suggest an alternative elementary proof in Appendix~\ref{appendix: original proof of Knuth}.
\begin{theorem}\label{theorem: lovasz_disjoint_union}
  \cite{Knuth94} Let $\Gr{G}$ and $\Gr{H}$ be simple graphs. Then,
  \begin{align}\label{eq:lovasz_disjoint_union}
    \vartheta(\Gr{G} + \Gr{H}) = \vartheta(\Gr{G}) + \vartheta(\Gr{H}).
  \end{align}
\end{theorem}

The following result is an easy consequence of Definition~\ref{definition:lovasz_function},
and it is an analog of Theorem~\ref{theorem: shannon_capacity_subgraphs}.
\begin{theorem}\label{theorem: lovasz_function_subgraphs}
    Let $\Gr{G}$ be a simple graph, and let $\Gr{H}_1$ and $\Gr{H}_2$ be induced and spanning subgraphs of $\Gr{G}$, respectively. Then,
  \begin{align}\label{eq:lovasz_function_subgraphs}
    \vartheta(\Gr{H}_1) \leq \vartheta(\Gr{G}) \leq \vartheta(\Gr{H}_2).
  \end{align}
\end{theorem}

Next, we present a few known formulas and bounds on the Lov\'{a}sz function.

\begin{theorem}\label{theorem: lovasz_bigger_n}
  \cite{Lovasz79,Sason23} Let $\Gr{G}$ be a simple graph on $n$ vertices. Then,
  \begin{align}\label{eq6:28.11.2024}
    \vartheta(\Gr{G}) \, \vartheta(\CGr{G}) \geq n,
  \end{align}
  with an equality in \eqref{eq6:28.11.2024} if $\Gr{G}$ is a vertex-transitive or strongly regular graph.
\end{theorem}

\begin{theorem}\label{theorem: lovasz_hoffman_bound}
  \cite{Lovasz79} Let $\Gr{G}$ be a $d$-regular graph of order $n$, and let $\lambda_n$ be its smallest eigenvalue. Then,
  \begin{align}\label{eq:lovasz_hoffman_bound}
    \vartheta(\Gr{G})\leq-\frac{n\lambda_n}{d - \lambda_n},
  \end{align}
  with an equality in \eqref{eq:lovasz_hoffman_bound} if $\Gr{G}$ is an edge-transitive graph.
\end{theorem}

\subsection{Concluding preliminaries} \label{subsection:additional_preliminaries}

We provide a useful lower bound for the fractional independence number and the fractional chromatic number (see
\cite[Proposition~3.1.1]{ScheinermanU13}).
\begin{theorem}\label{theorem: lower_bound_fractional_number}
  Let $\Gr{G}$ be a simple graph of order $n$. Then,
  \begin{align}
    \findnum{\Gr{G}} & \geq \frac{n}{\clnum{\Gr{G}}}, \\
    \fchrnum{\Gr{G}} & \geq \frac{n}{\indnum{\Gr{G}}}.
  \end{align}
  Both inequalities hold with equality for vertex-transitive graphs.
\end{theorem}

The following sandwich theorem for the Lov\'{a}sz $\vartheta$-function is remarkable in that it provides a computable bound on
graph invariants that are NP-hard to compute.
\begin{theorem} \label{theorem: sandwich_theorem}
  Let $\Gr{G}$ be a simple graph. Then,
  \begin{align}  \label{eq: sandwich_theorem}
    \indnum{\Gr{G}} \leq \Theta(\Gr{G}) \leq \vartheta(\Gr{G}) \leq \findnum{\Gr{G}} \leq \chrnum{\CGr{G}} = \clcnum{\Gr{G}}.
  \end{align}
\end{theorem}

In continuation to Definition~\ref{definition:kneser_graph}, we provide some known invariants of Kneser graphs.
\begin{theorem}\label{theorem: kneser_graph_invariants}
  Let $\Gr{G} = \KG{n}{r}$ be a Kneser graph with $n \geq 2r$. The invariants of $\Gr{G}$ are
  \begin{align}
    \indnum{\Gr{G}} & = \binom{n-1}{r-1}, \\
    \clnum{\Gr{G}} & = \bigg\lfloor\frac{n}{r}\bigg\rfloor, \\
    \Theta(\Gr{G}) & = \binom{n-1}{r-1}, \label{eq: Shannon capacity of Kneser graphs} \\
    \vartheta(\Gr{G}) & = \binom{n-1}{r-1}, \\
    \vartheta(\CGr{G}) & = \frac{n}{r}, \\
    \findnum{\Gr{G}} & = \frac{\binom{n}{r}}{\bigl\lfloor\frac{n}{r}\bigr\rfloor}, \\[0.1cm]
    \fchrnum{\Gr{G}} & = \frac{n}{r}, \\
    \chrnum{\Gr{G}} & = n - 2r + 2, \\
    \clcnum{\Gr{G}} & = \Bigg\lceil\frac{\binom{n}{r}}{\big\lfloor\frac{n}{r}\big\rfloor}\Bigg\rceil.
  \end{align}
\end{theorem}
\begin{proof}
  See \cite[Theorem~13]{Lovasz79}, \cite{BrouwerS79}, and Theorem~\ref{theorem: lower_bound_fractional_number}.
\end{proof}

\begin{theorem}\label{theorem: path_invariants}
  The path $\PathG{\ell}$, for $\ell \in \naturals$, is a universal graph with parameters
  \begin{align}\label{eq:path_invariants}
    \indnum{\PathG{\ell}} = \Theta(\PathG{\ell}) = \vartheta(\PathG{\ell}) = \findnum{\PathG{\ell}} = \clcnum{\PathG{\ell}} = \left\lceil\frac{\ell}{2}\right\rceil.
  \end{align}
\end{theorem}

\begin{theorem}\label{theorem: cycle_invariants}
  Let $k \geq 4$. Then, the cycle graph $\CG{k}$ has
  \begin{itemize}
    \item If $k$ is an even number, then $\CG{k}$ is a universal graph with
    \begin{align}
      \indnum{\CG{k}} = \Theta(\CG{k}) = \vartheta(\CG{k}) =\frac{k}{2}.
    \end{align}
    \item If $k \geq 5$ is an odd number, then
    \begin{align}
      \indnum{\CG{k}} = \left\lfloor\frac{k}{2}\right\rfloor, \quad \vartheta(\CG{k}) = \frac{k}{1+\sec\frac{\pi}{k}}.
    \end{align}
  \end{itemize}
\end{theorem}

\begin{lemma}\label{lemma:universal_capacity_independent}
    If $\Gr{G}$ is a universal graph, then $\Theta(\Gr{G}) = \indnum{\Gr{G}}$.
\end{lemma}
\begin{proof}
    If $\Gr{G}$ is a universal graph, then by Theorems \ref{theorem: universal_findnum} and \ref{theorem: sandwich_theorem},
    \begin{align}
      \Theta(\Gr{G}) = \indnum{\Gr{G}}.
    \end{align}
\end{proof}

\begin{remark}
    The converse of Lemma~\ref{lemma:universal_capacity_independent} is in general false, see Example~\ref{example:petersen_not_universal}.
\end{remark}

\begin{corollary} \label{corollary: when Kneser graphs are universal}
    Let $\Gr{G} = \KG{n}{r}$ be a Kneser graph with $n\geq2r$. Then, $\Gr{G}$ is universal if and only if $r \mid n$.
\end{corollary}
\begin{proof}
  By Theorem~\ref{theorem: kneser_graph_invariants},
  \begin{align}
    \findnum{\Gr{G}} = \frac{\binom{n}{r}}{\lfloor\frac{n}{r}\rfloor}.
  \end{align}
  For $\Gr{G}$ to be a universal graph, the equality $\indnum{\Gr{G}} = \findnum{\Gr{G}}$ should hold
  (see Theorem~\ref{theorem: universal_findnum}). Thus, $\Gr{G}$ is universal if and only if
  \begin{align}
    \binom{n-1}{r-1} = \frac{\binom{n}{r}}{\lfloor\frac{n}{r}\rfloor},
  \end{align}
  which holds if and only if $r \mid n$, as required.
\end{proof}

\begin{example}\label{example:petersen_not_universal}
    The Petersen graph $\Gr{G}=\KG{5}{2}$ satisfies $\indnum{\Gr{G}}=4=\Theta(\Gr{G})$, yet it is not universal
    by Corollary~\ref{corollary: when Kneser graphs are universal}.
\end{example}


\section{The Shannon capacity of polynomials in graphs}\label{section:polynomial_decomposition}

The Shannon capacity of the disjoint union of graphs is at least the sum of their individual capacities \cite{Shannon56};
however, it may in fact be significantly larger \cite{Alon98}. This phenomenon is strikingly demonstrated by a result of Alon \cite[Theorem~1.1]{Alon98},
which we restate here as Theorem~\ref{theorem: capacity of union}. Moreover, Alon showed in \cite[Theorem~1.2]{Alon98} that if we let $\Gr{G}$ be the
Schl\"{a}fli graph, then
\begin{align} \label{eq: Noga}
\Theta(\Gr{G}) + \Theta(\CGr{G}) \leq 10 < 2\sqrt{27} \leq \Theta(\Gr{G}+\CGr{G}).
\end{align}
We note that the simple counterexample constructed by Haemers~\cite{Haemers79} plays a key role in this argument and was later used by Alon~\cite{Alon98}
to establish~\eqref{eq: Noga}.

We provide sufficient conditions on a sequence of graphs under which the Shannon capacity of their disjoint union equals the sum
of their individual capacities. Our results rely on a recent result by Schrijver \cite{Schrijver23}, which we include with proof
due to its central role (see Theorem~\ref{theorem: shannon_union_equivalent_product}). It states that for any two simple graphs $\Gr{G}$ and $\Gr{H}$,
\begin{align}
  \Theta(\Gr{G}\boxtimes\Gr{H}) = \Theta(\Gr{G}) \, \Theta(\Gr{H}) \iff \Theta(\Gr{G} + \Gr{H}) = \Theta(\Gr{G}) + \Theta(\Gr{H}).
\end{align}
This result was independently proved by Wigderson and Zuiddam \cite{WigdersonZ23}, with credit to Holzman for personal communication.

Let $\naturals[x_1,\ldots,x_\ell]$ be the set of nonzero polynomials in the variables $x_1,\ldots,x_\ell$ with nonnegative integer coefficients.
A polynomial in graphs, $p(\Gr{G}_1,\ldots,\Gr{G}_\ell)$, is defined by interpreting graph addition as the disjoint union and graph multiplication
as the strong product.

In this section, we derive sufficient conditions for a sequence of graphs $\Gr{G}_1, \ldots, \Gr{G}_\ell$ to satisfy the property that, for every polynomial
$p \in \naturals[x_1, \ldots, x_\ell]$, the following equality holds:
\begin{align}\label{eq1:20.4.25}
  \Theta(p(\Gr{G}_1,\Gr{G}_2,\ldots,\Gr{G}_\ell)) = p(\Theta(\Gr{G}_1),\Theta(\Gr{G}_2) , \ldots,\Theta(\Gr{G}_\ell)).
\end{align}
By a corollary of the duality theorem (Theorem~\ref{theorem: shannon_union_equivalent_product}), which was proved by Schrijver
\cite{Schrijver23}, the following surprising result holds.
\begin{lemma}\label{lemma:structured_graphs_sum_polynomial}
  Let $\Gr{G}_1,\Gr{G}_2,\ldots,\Gr{G}_\ell$ be a sequence of simple graphs. Then, equality \eqref{eq1:20.4.25}
  holds for every $p\in\naturals[x_1,\ldots,x_\ell]$ if and only if
  \begin{align}\label{eq2:20.4.25}
  \Theta(\Gr{G}_1 + \Gr{G}_2 + \ldots + \Gr{G}_\ell) = \Theta(\Gr{G}_1) + \Theta(\Gr{G}_2) + \ldots + \Theta(\Gr{G}_\ell).
  \end{align}
\end{lemma}
By Lemma~\ref{lemma:structured_graphs_sum_polynomial} (see \cite[Theorem~3]{Schrijver23}), we focus on finding sufficient conditions for the satisfiability of \eqref{eq2:20.4.25}.
In this section, we present two main results that provide such sufficient conditions, discuss the differences between them, and illustrate their application.
We start by providing the following result.
\begin{theorem}\label{theorem: capacity_lovasz_bound}
  Let $\Gr{G}_1,\ldots,\Gr{G}_\ell$ be simple graphs, with $\ell\in\naturals$. Then, for every $p\in\naturals[x_1,\ldots,x_\ell]$,
  \begin{align}\label{eq1:9.12.2024}
    p(\Theta(\Gr{G}_1),\ldots,\Theta(\Gr{G}_\ell))\leq\Theta(p(\Gr{G}_1,\ldots,\Gr{G}_\ell))\leq p(\vartheta(\Gr{G}_1),\ldots,\vartheta(\Gr{G}_\ell)).
  \end{align}
\end{theorem}
\begin{proof}
  By Theorems~\ref{theorem: shannon_strong_product} and~\ref{theorem: shannon_disjoint_union},
  \begin{align}  \label{eq1: 26.01.26}
     p(\Theta(\Gr{G}_1),\ldots,\Theta(\Gr{G}_\ell))\leq\Theta(p(\Gr{G}_1,\ldots,\Gr{G}_\ell)),
  \end{align}
  and by Theorems~\ref{theorem: lovasz_strong_product} and \ref{theorem: lovasz_disjoint_union},
  \begin{align}
    \Theta(p(\Gr{G}_1,\ldots,\Gr{G}_\ell)) & \leq \vartheta(p(\Gr{G}_1,\ldots,\Gr{G}_\ell)) \nonumber \\
    & = p(\vartheta(\Gr{G}_1),\ldots,\vartheta(\Gr{G}_\ell)).  \label{eq2: 26.01.26}
  \end{align}
  Combining \eqref{eq1: 26.01.26} and \eqref{eq2: 26.01.26} gives both inequalities in \eqref{eq1:9.12.2024}.
\end{proof}

\begin{corollary}\label{corollary:capacity_lovasz_polynomial_eq}
  Let $\Gr{G}_1, \ldots, \Gr{G}_\ell$ be simple graphs, and let $p\in\naturals[x_1, \ldots, x_\ell]$ with $\ell \in \naturals$.
  If $\Theta(\Gr{G}_i) = \vartheta(\Gr{G}_i)$ for every $i \in \OneTo{\ell}$, then
  \begin{align}\label{eq2:9.12.2024}
    \Theta(p(\Gr{G}_1, \ldots, \Gr{G}_\ell)) = p\bigl(\vartheta(\Gr{G}_1), \ldots, \vartheta(\Gr{G}_\ell)\bigr).
  \end{align}
\end{corollary}
\begin{proof}
  By assumption, it follows that
  \begin{align}
    p(\Theta(\Gr{G}_1), \ldots, \Theta(\Gr{G}_\ell)) = p\bigl(\vartheta(\Gr{G}_1), \ldots, \vartheta(\Gr{G}_\ell) \bigr).
  \end{align}
  Thus, equality \eqref{eq2:9.12.2024} holds by Theorem~\ref{theorem: capacity_lovasz_bound}.
\end{proof}

The next result readily follows.
\begin{theorem}\label{theorem: structured_first_result}
  Let $\Gr{G}_1, \ldots, \Gr{G}_\ell$ be simple graphs, with $\ell \in \naturals$. If $\Theta(\Gr{G}_i) = \vartheta(\Gr{G}_i)$
  for every $i \in \OneTo{\ell}$, then equality \eqref{eq1:20.4.25} holds for every $p \in \naturals[x_1, \ldots, x_\ell]$.
\end{theorem}
\begin{proof}
  This holds by Corollary~\ref{corollary:capacity_lovasz_polynomial_eq} upon replacing $\Theta(\Gr{G}_i)$ with $\vartheta(\Gr{G}_i)$
  for all $i \in \OneTo{\ell}$.
\end{proof}

\begin{example}
  If $\Gr{G}_1,\ldots,\Gr{G}_\ell$ are all Kneser graphs or self-complementary vertex-transitive graphs, then by Theorem~\ref{theorem: structured_first_result},
  equality~\eqref{eq1:20.4.25} holds for every polynomial $p\in\naturals[x_1,\ldots,x_\ell]$.
  In particular, if $\Gr{G}_i = \KG{n_i}{r_i}$ (with $n_i\geq 2r_i$) for every $i\in\OneTo{\ell}$, then by Theorem~\ref{theorem: kneser_graph_invariants},
  \begin{align}
    \Theta(\KG{n_1}{r_1} + \ldots + \KG{n_\ell}{r_\ell}) = \binom{n_1 -1}{r_1 -1} + \ldots + \binom{n_\ell -1}{r_\ell -1}.
  \end{align}
  Similarly, if $\Gr{G}_i$ is a self-complementary and vertex-transitive graph for all $i\in\OneTo{\ell}$, then by \cite[Theorem~3.26]{Sason24},
  \begin{align}
    \Theta(\Gr{G}_1 + \ldots + \Gr{G}_\ell) = \sqrt{n_1} + \ldots + \sqrt{n_\ell},
  \end{align}
  where $n_i$ is the order of $\Gr{G}_i$.
\end{example}

Next, we give a second sufficient condition for equality~\eqref{eq1:20.4.25} to hold.
\begin{lemma}\label{lemma:structured_second_result}
  Let $\Gr{G}_1$ and $\Gr{G}_2$ be simple graphs. If $\Theta(\Gr{G}_1)=\findnum{\Gr{G}_1}$, then
  \begin{align}\label{eq:lemma_structured_second_result}
    \Theta(\Gr{G}_1 + \Gr{G}_2)=\Theta(\Gr{G}_1) + \Theta(\Gr{G}_2).
  \end{align}
\end{lemma}
\begin{proof}
  By Theorem~\ref{theorem: shannon_strong_product},
  \begin{align}\label{eq1:17.8.25}
    \Theta(\Gr{G}_1) \, \Theta(\Gr{G}_2)\leq\Theta(\Gr{G}_1\boxtimes\Gr{G}_2),
  \end{align}
  and by Theorem~\ref{theorem: fractional_independence_strong_product},
  \begin{align}
    \Theta(\Gr{G}_1\boxtimes\Gr{G}_2) & = \lim_{k\to\infty} \sqrt[k]{\indnum{(\Gr{G}_1\boxtimes\Gr{G}_2)^k}} \nonumber \\
    & \leq \lim_{k\to\infty} \sqrt[k]{\findnum{\Gr{G}_1^k} \, \indnum{\Gr{G}_2^k}} \nonumber \\
    & = \findnum{\Gr{G}_1} \, \Theta(\Gr{G}_2),   \label{eq2:17.8.25}
  \end{align}
  where \eqref{eq2:17.8.25} holds since the equality $\findnum{\Gr{G}_1} = \sqrt[k]{\findnum{\Gr{G}_1^k}}$ is satisfied for all $k \in \naturals$,
  and by definition $\Theta(\Gr{G}_2) = \underset{k \to \infty}{\lim} \sqrt[k]{\indnum{\Gr{G}_2^k} }$.
  Combining inequalities \eqref{eq1:17.8.25} and \eqref{eq2:17.8.25} gives
  \begin{align}
    \Theta(\Gr{G}_1) \, \Theta(\Gr{G}_2) & \leq\Theta(\Gr{G}_1\boxtimes\Gr{G}_2) \nonumber \\
    & \leq\findnum{\Gr{G}_1} \, \Theta(\Gr{G}_2).
  \end{align}
  Hence, by the assumption,
  \begin{align}
    \Theta(\Gr{G}_1) \, \Theta(\Gr{G}_2)=\Theta(\Gr{G}_1\boxtimes\Gr{G}_2),
  \end{align}
  and equality~\eqref{eq:lemma_structured_second_result} follows from Theorem~\ref{theorem: shannon_union_equivalent_product}.
\end{proof}

\begin{theorem}\label{theorem: structured_second_result}
  Let $\Gr{G}_1,\Gr{G}_2,\ldots,\Gr{G}_\ell$ be simple graphs. If $\Theta(\Gr{G}_i)=\findnum{\Gr{G}_i}$ for (at least) $\ell - 1$ of these graphs, then
  equality~\eqref{eq1:20.4.25} holds for every $p\in\naturals[x_1,\ldots,x_\ell]$.
\end{theorem}
\begin{proof}
  Without loss of generality, assume $\Theta(\Gr{G}_i)=\findnum{\Gr{G}_i}$ for every $i\in\OneTo{\ell-1}$.
  By a recursive application of Lemma~\ref{lemma:structured_second_result},
  \begin{align}
    \Theta\left(\sum_{i=1}^{\ell}{\Gr{G}_i}\right) = \Theta(\Gr{G}_1) + \Theta\left(\sum_{i=2}^{\ell}{\Gr{G}_i}\right) = \ldots = \sum_{i=1}^{\ell}{\Theta(\Gr{G}_i)}.
  \end{align}
  Thus, by Lemma~\ref{lemma:structured_graphs_sum_polynomial}, equality~\eqref{eq1:20.4.25} holds for every $p\in\naturals[x_1,\ldots,x_\ell]$.
\end{proof}
It is worth noting that our sufficient conditions for equality~\eqref{eq1:20.4.25} to hold are less restrictive than those hinted at in Shannon's paper \cite{Shannon56},
which are presented next. The following theorem is derived directly from Lemma~\ref{lemma:structured_second_result} and the sandwich theorem. We provide the original
proof in Appendix~\ref{appendix: original proof by Shannon}.
\begin{theorem}\label{theorem: shannon_condition}
  Let $\Gr{G}_1$ and $\Gr{G}_2$ be simple graphs with $\indnum{\Gr{G}_1} = \chrnum{\CGr{G}_1}$. Then,
  \begin{align}
    \Theta(\Gr{G}_1 + \Gr{G}_2)=\Theta(\Gr{G}_1) + \Theta(\Gr{G}_2).
  \end{align}
\end{theorem}
\begin{remark}
The sufficient condition in Theorem~\ref{theorem: shannon_condition} is more restrictive than the one in
Lemma~\ref{lemma:structured_second_result}. This holds since by Theorem~\ref{theorem: sandwich_theorem},
  \begin{align}
    \indnum{\Gr{G}} \leq \Theta(\Gr{G})\leq \findnum{\Gr{G}} \leq \chrnum{\CGr{G}}.
  \end{align}
\end{remark}
The next example suggests a family of graphs for which $\Theta(\Gr{G})=\findnum{\Gr{G}} < \chrnum{\CGr{G}}$ holds for every graph $\Gr{G}$ in that family,
thus showing the possible applicability of Lemma~\ref{lemma:structured_second_result} in cases where the conditions in Theorem~\ref{theorem: shannon_condition}
are not satisfied.
\begin{example}
  Let $\Gr{G}=\overline{\KG{n}{r}}$ be the complement of a Kneser graph where $n>2r$, $r>1$, and $r \mid n$. It is claimed that the graph $\Gr{G}$ satisfies
  \begin{align}
    \Theta(\Gr{G}) = \findnum{\Gr{G}} < \chrnum{\CGr{G}}.
  \end{align}
  Indeed, the invariants of the Kneser graph and its complement are known, and by Theorem~\ref{theorem: kneser_graph_invariants} (since $\indnum{\Gr{G}} = \clnum{\CGr{G}}$, where $\CGr{G} = \KG{n}{r}$ with $r$ and $n$ as above),
  \begin{align}
    \indnum{\Gr{G}} & = \frac{n}{r}, \\[0.1cm]
    \findnum{\Gr{G}} & = \frac{n}{r}, \\
    \chrnum{\CGr{G}} & = n - 2r + 2.
  \end{align}
  Since the independence number and the fractional independence number coincide and are both equal to $\frac{n}{r}$, then $\Theta(\Gr{G})=\frac{n}{r}$.
  Hence, the assumption $n>2r$ implies that
  \begin{align}
    \Theta(\Gr{G}) = \findnum{\Gr{G}} = \frac{n}{r} < n - 2r + 2 = \chrnum{\CGr{G}}.
  \end{align}
\end{example}

\begin{remark}\label{remark:structured_results_comparison}
  The sufficient conditions provided by Theorems~\ref{theorem: structured_first_result} and~\ref{theorem: structured_second_result}
  do not supersede each other (see Examples~\ref{example:structured_first_over_second} and \ref{example:structured_second_over_first}).
  More explicitly, for every graph $\Gr{G}$,
  \begin{align}
    \Theta(\Gr{G}) \leq \vartheta(\Gr{G}) \leq \findnum{\Gr{G}},
  \end{align}
  so the condition $\Theta(\Gr{G}) = \findnum{\Gr{G}}$ implies that the equality $\Theta(\Gr{G}) = \vartheta(\Gr{G})$ holds.
  Hence, Theorem~\ref{theorem: structured_second_result} imposes a stronger condition than the one in Theorem~\ref{theorem: structured_first_result}
  on $\ell - 1$ of the graphs, while no condition is imposed on the $\ell$th graph.
\end{remark}

The next two examples (Examples~\ref{example:structured_first_over_second} and~\ref{example:structured_second_over_first}) show that
neither Theorem~\ref{theorem: structured_first_result} nor Theorem~\ref{theorem: structured_second_result} implies the other.
\begin{example} \label{example:structured_first_over_second}
  Let $\Gr{G}_1,\ldots,\Gr{G}_\ell$ be Kneser graphs, where
  \begin{align}
    \Gr{G}_i = \KG{n_i}{r_i}, \quad n_i \geq 2 r_i.
  \end{align}
  For these non-empty graphs, by Theorem~\ref{theorem: kneser_graph_invariants},
  \begin{align}  \label{eq2: 30.12.25}
    \indnum{\Gr{G}_i} = \Theta(\Gr{G}_i) = \vartheta(\Gr{G}_i) = \binom{n_i - 1}{r_i -1}, \quad \forall \, i \in \OneTo{\ell},
  \end{align}
  and
  \begin{align}\label{eq10:18.5.25}
    \findnum{\Gr{G}_i} = \frac{\binom{n_i}{r_i}}{\bigl\lfloor\frac{n_i}{r_i}\bigr\rfloor}, \quad \forall \, i \in \OneTo{\ell}.
  \end{align}
  If $r_i \nmid n_i$, then it follows from \eqref{eq2: 30.12.25} and \eqref{eq10:18.5.25} that
  \begin{align}\label{eq11:18.5.25}
    \findnum{\Gr{G}_i} > \frac{\binom{n_i}{r_i}}{\frac{n_i}{r_i}} = \binom{n_i-1}{r_i-1} = \Theta(\Gr{G}_i),
  \end{align}
  so, if $r_i\nmid n_i$, then $\findnum{\Gr{G}_i} > \Theta(\Gr{G}_i)$.

  Let the parameters $n_1, \ldots, n_\ell$ and $r_1, \ldots, r_\ell$ be chosen such that $r_i \nmid n_i$
  for at least two of the Kneser graphs $\{\Gr{G}_i\}_{i=1}^{\ell}$. For each such index $i \in \OneTo{\ell}$,
  we have $\findnum{\Gr{G}_i} > \Theta(\Gr{G}_i)$, which violates the sufficient conditions of
  Theorem~\ref{theorem: structured_second_result}. By \eqref{eq2: 30.12.25}, it follows from
  Theorem~\ref{theorem: structured_first_result} that
  \begin{align}  \label{eq: 30.12.25}
    \Theta(\Gr{G}_1 + \ldots + \Gr{G}_\ell) = \sum_{i=1}^{\ell}\binom{n_i - 1}{r_i - 1},
  \end{align}
  whereas equality \eqref{eq: 30.12.25} does not follow from Theorem~\ref{theorem: structured_second_result}.
\end{example}

\begin{example} \label{example:structured_second_over_first}
  Let $\Gr{G}_i$ be a perfect graph for every $i\in\OneTo{\ell - 1}$, and let $\Gr{G}_\ell$
  be the complement of the Schl\"{a}fli graph. Then,
  \begin{itemize}
    \item $\Theta(\Gr{G}_i) = \findnum{\Gr{G}_i}$ for every $i \in \OneTo{\ell - 1}$.
    \item $\Theta(\Gr{G}_\ell) < \vartheta(\Gr{G}_\ell)$ (by \cite{Haemers79}).
  \end{itemize}
  Hence, the sufficient conditions of Theorem~\ref{theorem: structured_second_result} hold,
  in contrast to those of Theorem~\ref{theorem: structured_first_result}.
\end{example}


\section{The Shannon capacity of Tadpole graphs} \label{section:tadpole}

The present section explores Tadpole graphs, where exact values and bounds on their Shannon capacity are derived.

\begin{definition}[Tadpole graphs]
\label{definition:tadpole_graph}
  Let $k, \ell \in \naturals$ with $k \geq 3$ and $\ell\geq1$. The graph $\TG{k}{\ell}$, called the Tadpole graph of order $(k, \ell)$,
  is obtained by taking a cycle $\CG{k}$ of order $k$ and a path $\PathG{\ell}$ of order $\ell$, and then joining one pendant vertex of
  $\PathG{\ell}$ (i.e., one of its two vertices of degree~1) to a vertex of $\CG{k}$ by an edge.
  For completeness, if $\ell=0$, the Tadpole graph is defined trivially as a cycle graph of order $k$, $\TG{k}{0}=\CG{k}$.
\end{definition}

\begin{figure}[ht]
  \centering
  \includegraphics[width=0.6\textwidth]{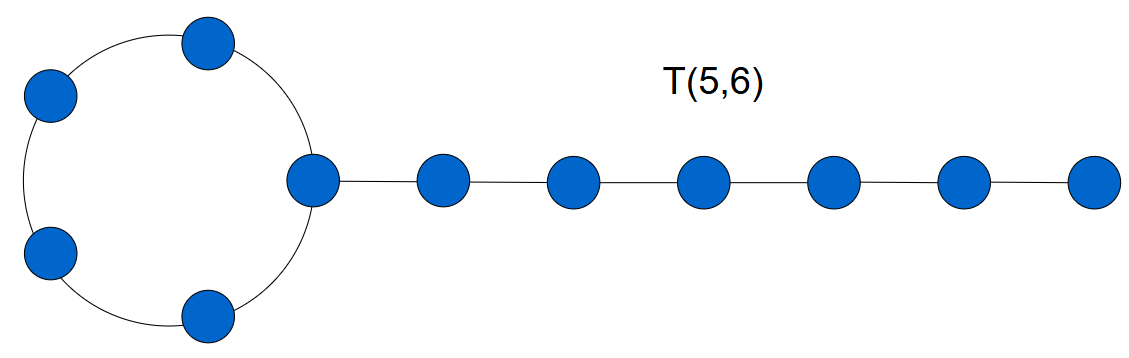}
  \caption{The Tadpole graph $\TG{5}{6}$.}
  \label{fig:tadpole}
\end{figure}

Every Tadpole graph is a connected graph with
\begin{align}
  \card{\V{\TG{k}{\ell}}} = \card{\E{\TG{k}{\ell}}} = k + \ell.
\end{align}
The Tadpole graph $\TG{k}{\ell}$ with $\ell\geq1$ is irregular since it has one vertex of degree~$3$, $k+\ell - 2$ vertices of degree~$2$, and one vertex of degree~$1$
(see Figure~\ref{fig:tadpole}).

The motivation for studying Tadpole graphs stems from their similarity to cycle graphs, both as graphs and in terms of their equivalent DMCs, as illustrated in Figure~\ref{fig:tadpole-DMCs}.
\begin{figure}[ht]
  \centering
  \includegraphics[width=0.9\textwidth]{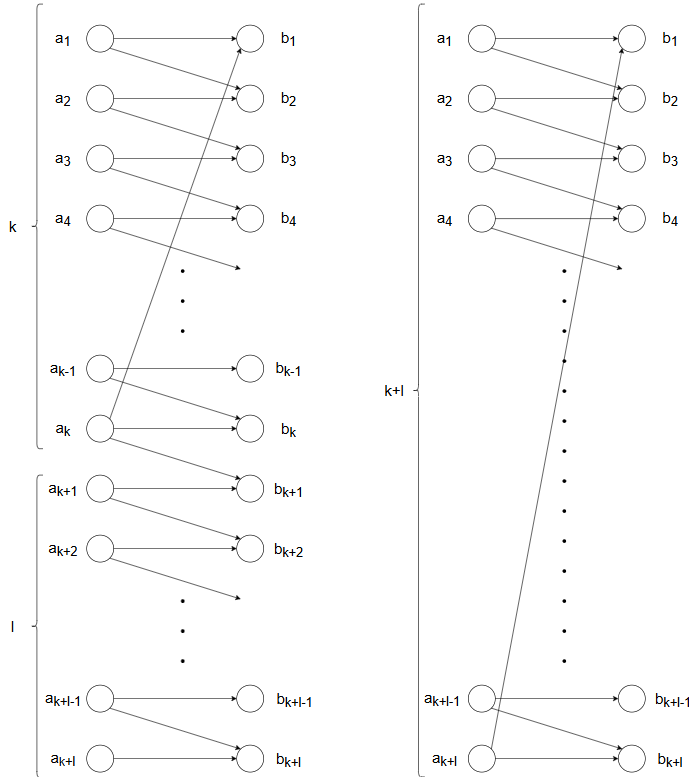}
  \caption{The DMCs of $\TG{k}{\ell}$ (left plot) and $\CG{k+\ell}$ (right plot).}
  \label{fig:tadpole-DMCs}
\end{figure}
\begin{lemma}\label{lemma:independence_number_tadpole_graph}
  Let $k\geq3$ and $\ell\geq0$. The independence number of the Tadpole graph $\TG{k}{\ell}$ is given by
  \begin{align}\label{eq:independence_number_tadpole_graph}
    \indnum{\TG{k}{\ell}} = \left\lfloor\frac{k}{2}\right\rfloor + \left\lceil\frac{\ell}{2}\right\rceil.
  \end{align}
\end{lemma}
\begin{proof}
  The independence number of $\CG{k}$ is $\left\lfloor\frac{k}{2}\right\rfloor$, and the independence number of $\PathG{\ell}$ is $\left\lceil\frac{\ell}{2}\right\rceil$.
  One can select a maximal independent set by excluding the vertex of $\CG{k}$ that is adjacent to the vertex of $\PathG{\ell}$, which gives \eqref{eq:independence_number_tadpole_graph}.
\end{proof}

\begin{lemma}\label{lemma:lovasz_tadpole_graph}
  Let $k\geq3$ and $\ell\geq0$.
  \begin{enumerate}
    \item If one of the following two conditions holds:
    \begin{itemize}
      \item $k=3$ or $k\geq4$ is even,
      \item $k\geq5$ is odd and $\ell\geq1$ is odd,
    \end{itemize}
    then
    \begin{align}
      \vartheta(\TG{k}{\ell}) = \left\lfloor\frac{k}{2}\right\rfloor + \left\lceil\frac{\ell}{2}\right\rceil.
    \end{align}
    \item If $k\geq5$ is odd and $\ell\geq0$ is even, then
    \begin{align}
      \vartheta(\TG{k}{\ell}) = \frac{k}{1+\sec\frac{\pi}{k}}+\frac{\ell}{2}.
    \end{align}
  \end{enumerate}
\end{lemma}
\begin{proof}
  First, if $k=3$, then $\CG{3}$ is a clique of order $3$. Since the clique-cover number of $\PathG{\ell}$ is $\left\lceil\frac{\ell}{2}\right\rceil$, it follows that
  \begin{align}
    \clcnum{\TG{3}{\ell}} = 1 + \left\lceil\frac{\ell}{2}\right\rceil.
  \end{align}
  By Lemma~\ref{lemma:independence_number_tadpole_graph} and Theorem~\ref{theorem: sandwich_theorem}, the lower and upper bounds on
  $\vartheta(\TG{k}{\ell})$ coincide, which gives
  \begin{align}
    \vartheta(\TG{3}{\ell}) = 1 + \left\lceil\frac{\ell}{2}\right\rceil.
  \end{align}
  Next, if $k>3$, then the clique number of $\TG{k}{\ell}$ is $2$ (since there are no triangles), thus every clique in $\TG{k}{\ell}$ is either
  an edge or a single vertex. In this case, if $k$ is even, it is possible to cover the cycle $\CG{k}$ by $\frac{k}{2}$ edges (cliques),
  the path $\PathG{\ell}$ can be covered by $\left\lceil\frac{\ell}{2}\right\rceil$ edges, and thus
  $\clcnum{\TG{k}{\ell}} = \frac{k}{2} + \left\lceil\frac{\ell}{2}\right\rceil$. Moreover, if $k$ and $\ell$ are odd, then one can cover
  $\TG{k}{\ell}$ by $\frac{k+\ell}{2}$ edges (by covering $\CG{k}$ by $\frac{k-1}{2}$ edges, excluding the vertex of $\CG{k}$ that is adjacent to
  a pendant vertex of $\PathG{\ell}$, and then covering
  the remaining vertices of $\TG{k}{\ell}$ by $\frac{\ell+1}{2}$ edges). Overall, in both instances we obtain
  \begin{align}
    \clcnum{\TG{k}{\ell}} = \left\lfloor\frac{k}{2}\right\rfloor + \left\lceil\frac{\ell}{2}\right\rceil,
  \end{align}
  which gives, by again using the sandwich theorem in \eqref{eq: sandwich_theorem} and Lemma~\ref{lemma:independence_number_tadpole_graph},
  \begin{align}
    \vartheta(\TG{k}{\ell}) = \left\lfloor\frac{k}{2}\right\rfloor + \left\lceil\frac{\ell}{2}\right\rceil.
  \end{align}
  Finally, assume that $k\geq5$ is odd and $\ell\geq2$ is even (if $\ell=0$, then $\TG{k}{0}=\CG{k}$, and the result follows from Theorem~\ref{theorem: cycle_invariants}).
  By deleting the edge that connects the cycle $\CG{k}$ with the path $\PathG{\ell}$, it follows that $\CG{k} + \PathG{\ell}$
  is a spanning subgraph of $\TG{k}{\ell}$.
  On the other hand, by deleting the leftmost vertex of $\PathG{\ell}$ (i.e., the vertex that is of distance~$1$ from the cycle $\CG{k}$)
  and the two edges that are incident to it in $\TG{k}{\ell}$, it follows that $\CG{k} + \PathG{\ell-1}$ is an induced subgraph of $\TG{k}{\ell}$.
  Hence, by Theorem~\ref{theorem: lovasz_function_subgraphs},
  \begin{align}
    \vartheta(\CG{k}+\PathG{\ell-1}) & \leq \vartheta(\TG{k}{\ell}) \nonumber \\
    & \leq \vartheta(\CG{k}+\PathG{\ell}).
  \end{align}
  Furthermore, by Theorems \ref{theorem: path_invariants} and \ref{theorem: cycle_invariants},
  \begin{align}
    \vartheta(\CG{k}+\PathG{\ell-1}) & = \frac{k}{1+\sec\frac{\pi}{k}}+\frac{\ell}{2} \nonumber \\
    & = \vartheta(\CG{k}+\PathG{\ell}).
  \end{align}
  Thus,
  \begin{align}
    \vartheta(\TG{k}{\ell}) = \frac{k}{1+\sec\frac{\pi}{k}}+\frac{\ell}{2}.
  \end{align}
\end{proof}

\begin{theorem}\label{theorem: capacity_tadpole_graph}
  Let $k \geq 3$ and $\ell \geq 0$ be integers.
  \begin{enumerate}
    \item If one of the following two conditions holds:
    \begin{itemize}
      \item $k=3$ or $k \geq 4$ is even,
      \item $k\geq5$ is odd and $\ell \geq 1$ is odd,
    \end{itemize}
    then
    \begin{align}\label{eq:capacity_tadpole_graph_1}
      \Theta(\TG{k}{\ell})=\left\lfloor\frac{k}{2}\right\rfloor + \left\lceil\frac{\ell}{2}\right\rceil.
    \end{align}
    \item If $k\geq5$ is odd and $\ell\geq0$ is even, then
    \begin{align}\label{eq:capacity_tadpole_graph_2}
      \Theta(\TG{k}{\ell}) = \Theta(\CG{k}) + \frac{\ell}{2}.
    \end{align}
  \end{enumerate}
\end{theorem}
\begin{proof}
  First, if $k=3$, $k\geq4$ is even, or $k\geq5$ is odd and $\ell\geq1$ is odd, then by Lemmas~\ref{lemma:independence_number_tadpole_graph} and \ref{lemma:lovasz_tadpole_graph},
  \begin{align}
    \Theta(\TG{k}{\ell})=\left\lfloor\frac{k}{2}\right\rfloor + \left\lceil\frac{\ell}{2}\right\rceil.
  \end{align}
  Second, if $k\geq5$ is odd and $\ell\geq0$ is even, then by using the same subgraphs of $\TG{k}{\ell}$ from the proof of Lemma~\ref{lemma:lovasz_tadpole_graph}, and by Theorem~\ref{theorem: shannon_capacity_subgraphs},
  \begin{align}  \label{eq1: 09.01.25}
    \Theta(\CG{k}+\PathG{\ell-1}) \leq \Theta(\TG{k}{\ell}) \leq \Theta(\CG{k}+\PathG{\ell}).
  \end{align}
  Furthermore, as path graphs are universal (see Theorem~\ref{theorem: path_invariants}), it follows from Theorem~\ref{theorem: structured_second_result} that
  \begin{align}
    \Theta(\CG{k}+\PathG{\ell-1}) & = \Theta(\CG{k}) + \Theta(\PathG{\ell-1}) \\
    & = \Theta(\CG{k}) + \frac{\ell}{2} \\
    & = \Theta(\CG{k}+\PathG{\ell}).  \label{eq2: 09.01.25}
  \end{align}
  Combining \eqref{eq1: 09.01.25} and \eqref{eq2: 09.01.25} gives equality \eqref{eq:capacity_tadpole_graph_2}.
\end{proof}

The next result follows from equality \eqref{eq:capacity_tadpole_graph_2}.
\begin{corollary}\label{corollary:capacity_tadpole_bounds}
  Let $\ell \geq 0$ be an even number and let $k \geq 5$ be an odd number.
  \begin{enumerate}
    \item If $k=5$, then
    \begin{align} \label{eq: 27.12.2025}
      \Theta(\TG{5}{\ell}) = \sqrt{5} + \frac{\ell}{2}.
    \end{align}
    \item If $k\geq7$, then
    \begin{align}
    \label{eq1: 19.08.25}
      \frac{k+\ell-1}{2}\leq\Theta(\TG{k}{\ell})\leq\frac{k}{1+\sec\frac{\pi}{k}}+\frac{\ell}{2}.
    \end{align}
  \end{enumerate}
\end{corollary}
\begin{proof}
  If $k=5$, then $\Theta(\CG{5}) = \sqrt{5}$ \cite{Lovasz79}, and by Theorem~\ref{theorem: capacity_tadpole_graph},
  \begin{align}
    \Theta(\TG{5}{\ell}) & = \Theta(\CG{5}) + \frac{\ell}{2} \nonumber \\
    & = \sqrt{5} + \frac{\ell}{2}.
  \end{align}
  If $k\geq7$, then by Theorem \ref{theorem: cycle_invariants},
  \begin{align}
    \frac{k-1}{2} & = \indnum{\CG{k}} \nonumber \\
    & \leq\Theta(\CG{k}) \nonumber \\
    & \leq\vartheta(\CG{k}) = \frac{k}{1+\sec\frac{\pi}{k}}.
  \end{align}
  Thus, by Theorem~\ref{theorem: capacity_tadpole_graph},
  \begin{align}
    \frac{k+\ell-1}{2} & \leq\Theta(\TG{k}{\ell}) \nonumber \\
    & \leq\frac{k}{1+\sec\frac{\pi}{k}}+\frac{\ell}{2}.
  \end{align}
\end{proof}

\begin{example}\label{example:capacity_tadpole_bounds}
  Let $\Gr{G}_1 = \TG{5}{6}$. Then, by Corollary~\ref{corollary:capacity_tadpole_bounds},
  \begin{align}
    \Theta(\TG{5}{6}) = 3 + \sqrt{5} = 5.23607\ldots
  \end{align}
  For comparison, using the SageMath software \cite{SageMath} gives the values
  \begin{align}
    & \sqrt{\indnum{\Gr{G}_1\boxtimes\Gr{G}_1}}=\sqrt{26}= 5.09902, \\
    & \sqrt[3]{\indnum{\Gr{G}_1\boxtimes\Gr{G}_1\boxtimes\Gr{G}_1}}=\sqrt[3]{136}= 5.14256.
  \end{align}
  Let $\Gr{G}_2 = \TG{7}{6}$. By the lower bound in \cite{PolakS19} on the Shannon capacity of the cycle $\CG{7}$,
  \begin{align}  \label{eq: Polak}
  \Theta(\CG{7}) \geq \sqrt[5]{\indnum{\CG{7}^5}} \geq \sqrt[5]{367},
  \end{align}
  which implies by Corollary~\ref{corollary:capacity_tadpole_bounds} and \eqref{eq: Polak} that
  \begin{align}
    6.2578659\ldots = \sqrt[5]{367} + 3 & \leq \Theta(\CG{7}) + 3 \nonumber \\
    & = \Theta(\TG{7}{6}) \nonumber \\
    & \leq \frac{7}{1+\sec\frac{\pi}{7}} + 3 = 6.3176672.
  \end{align}
  The above lower bound on $\Theta(\TG{7}{6})$ improves the previous lower bound in the leftmost term of \eqref{eq1: 19.08.25},
  thus closing the gap between the upper and lower bounds from $0.3176$ to $0.0598$.
\end{example}

The last example shows that the lower bound in Corollary~\ref{corollary:capacity_tadpole_bounds} is not tight. The following result provides an improved lower bound.
\begin{theorem}\label{theorem: capacity_tadpole_graph_bound}
  If $k\geq5$ is odd and $\ell\geq0$ is even, then
  \begin{align}\label{eq:capacity_tadpole_graph_bound}
    \sqrt{\left(\frac{k-1}{2}\right)^2+\left\lfloor\frac{k-1}{4}\right\rfloor}+\frac{\ell}{2}\leq\Theta(\TG{k}{\ell})\leq\frac{k}{1+\sec\frac{\pi}{k}}+\frac{\ell}{2}.
  \end{align}
\end{theorem}
\begin{proof}
  The upper bound was proved in Corollary~\ref{corollary:capacity_tadpole_bounds}. Next, we prove the improved lower bound.
  By Hales' result \cite[Theorem~7.1]{Hales73}, if $k \geq 3$ is odd, then
  \begin{align} \label{eq1: 0.1.01.26}
    \indnum{\CG{k}^2} = \left(\frac{k-1}{2}\right)^2+\left\lfloor\frac{k-1}{4}\right\rfloor,
  \end{align}
  which implies that
  \begin{align}
    \Theta(\CG{k}) & \geq \sqrt{\indnum{\CG{k}^2}} \nonumber \\
    & = \sqrt{\left(\frac{k-1}{2}\right)^2+\left\lfloor\frac{k-1}{4}\right\rfloor}.  \label{eq2: 01.01.26}
  \end{align}
  Thus, by combining \eqref{eq2: 01.01.26} with Theorem~\ref{theorem: capacity_tadpole_graph}, it follows that
  \begin{align}
    \Theta(\TG{k}{\ell}) & = \Theta(\CG{k}) + \frac{\ell}{2} \nonumber \\
    & \geq \sqrt{\left(\frac{k-1}{2}\right)^2 + \left\lfloor\frac{k-1}{4}\right\rfloor} + \frac{\ell}{2}. \label{eq3: 01.01.26}
  \end{align}
\end{proof}
The gap between the upper and lower bounds on $\Theta(\TG{k}{\ell})$ in \eqref{eq:capacity_tadpole_graph_bound}, for an arbitrary
odd $k \geq 7$ and even $\ell \geq 0$, satisfies
\begin{align}
  & \hspace*{-0.2cm} \left(\frac{k}{1+\sec\frac{\pi}{k}} + \frac{\ell}{2}\right) - \left(\sqrt{\left( \frac{k-1}{2}\right)^2 + \left\lfloor \frac{k-1}{4} \right\rfloor}
  + \frac{\ell}{2}\right) \nonumber \\
  & < \frac{k+\ell}{2}-\frac{k+\ell-1}{2} \label{eq1:27.6.25} \\[0.1cm]
  & = \frac{1}{2}, \nonumber
\end{align}
where inequality \eqref{eq1:27.6.25} follows by dropping the term $\left\lfloor \frac{k-1}{4} \right\rfloor$
inside the square root of the first line, and since $\sec\frac{\pi}{k}>1$.

\begin{remark}\label{remark:tadpole_bound_holzman}
  Theorem \ref{theorem: capacity_tadpole_graph_bound} gives an improved lower bound that relies on a lower
  bound on the Shannon capacity of odd cycles that was constructed by Hales (see \cite{Hales73}). Since Hales'
  result, better lower bounds for odd cycles were constructed, and can be used to improve the lower bound in
  special cases. For better lower bounds on higher powers of odd cycles see \cite{Bohman05b, BohmanHN13}.
  The next theorem shows an improvement to the bound in \eqref{eq:capacity_tadpole_graph_bound} for special
  cases using a result from \cite{Bohman05b}.
\end{remark}

\begin{theorem}\label{theorem: capacity_tadpole_graph_bound_bohman}
Let $n, d \in \naturals$, let $\ell \geq 0$ be an even number, and define
\begin{align}
k = n 2^d + 2^{d-1} + 1.  \label{eq: k}
\end{align}
Then,
\begin{align}\label{eq:capacity_tadpole_graph_bound_bohman}
\sqrt[d]{nk^{d-1}+\left(\frac{k-1}{2}\right)k^{d-2}}+\frac{\ell}{2}\leq\Theta(\TG{k}{\ell}).
\end{align}
\end{theorem}
\begin{proof}
By Theorem~\ref{theorem: capacity_tadpole_graph},
\begin{align}\label{eq1:31.8.25}
\Theta(\TG{k}{\ell}) = \frac{\ell}{2} + \Theta(\CG{k}),
\end{align}
and, by \cite[Theorem~1.4]{Bohman05b},
\begin{align}\label{eq2:31.8.25}
\Theta(\CG{k}) & \geq \sqrt[d]{\indnum{\CG{k}^d}} \nonumber \\[0.1cm]
& \geq \sqrt[d]{nk^{d-1}+\left(\frac{k-1}{2}\right)k^{d-2}}.
\end{align}
Finally, combining \eqref{eq1:31.8.25} and \eqref{eq2:31.8.25} gives \eqref{eq:capacity_tadpole_graph_bound_bohman}.
\end{proof}

\begin{remark}\label{remark:cycle_bound_bohman_holzman}
The proof of Theorem~\ref{theorem: capacity_tadpole_graph_bound_bohman} relies on the lower bound on $\indnum{\CG{k}^d}$ given
in \eqref{eq2:31.8.25}, as presented in \cite[Theorem~1.4]{Bohman05b}.
It is worth noting that in \cite{Bohman05b}, Bohman conjectured that the lower bound is the exact value of the independence
number (see \cite[Conjecture~1.5]{Bohman05b}). In a subsequent paper by Bohman, Holzman, and Natarajan \cite{BohmanHN13},
the conjecture was confirmed for a countably infinite subset of these values. In particular, by substituting $d=3$, which
gives $k=8n+5$, it was proved in \cite{BohmanHN13} that if $8n+5$ is a prime number, then the lower bound coincides with the exact
value of the independence number. Specifically, by \cite[Theorem~1]{BohmanHN13}, the independence number of the strong cubed power
of the cycle graph $\CG{8n+5}$ is given by
\begin{align}\label{eq:cycle_bound_bohman_holzman}
\indnum{\CG{8n+5}^3}= \tfrac12 (8n + 5) \, \bigl[(2n+1)(8n+5) - 1\bigr].
\end{align}
\end{remark}


\section{When the graph capacity is not attained by the independence number of any finite power?}
\label{section: unattainability of Shannon capacity at any finite power}

In Section~\ref{section:tadpole}, we proved that if $k=3$, or if $k \geq 4$ is even, or if $k \geq 5$ and $\ell \geq 1$ are odd,
then the capacity of the Tadpole graph $\TG{k}{\ell}$ coincides with its independence number. Additional families of graphs also
share this property; e.g., by \cite{Lovasz79}, the capacities of all Kneser graphs are equal to their independence numbers.
It is, however, well known that not all graphs posses this property; e.g., $\Theta(\CG{5}) = \sqrt{\indnum{\CG{5}^2}} = \sqrt{5}$
and $\indnum{\CG{5}} = 2$. The property that the Shannon capacity coincides with the square root of the independence number of the
second (strong) power of the graph was proved, more generally for all self-complementary vertex transitive graphs \cite{Lovasz79}
and for all self-complementary strongly regular graphs \cite{Sason24}, provided that the order of the graph is not a square of an integer.
Families of graphs whose capacity is attained at a finite strong power that is strictly larger than $2$ are yet unknown.
In this section, we explore sufficient conditions on a graph $\Gr{G}$ such that its Shannon capacity is unattainable by the
independence number of any finite strong power of $\Gr{G}$, i.e.,
\begin{align} \label{eq3:20.4.25}
  \Theta(\Gr{G}) > \sqrt[k]{\indnum{\Gr{G}^k}}, \quad \forall \, k \in \naturals.
\end{align}

This problem was explored by Guo and Watanabe \cite{GuoW90}, and a family of disconnected graphs that satisfies \eqref{eq3:20.4.25}
was constructed. In this section, we start by presenting the method from \cite{GuoW90}, and then we provide two original approaches.
The first original approach relies on Dedekind's lemma in number theory (see Lemma~\ref{lemma:dedekind}), and it uses a similar concept of proof to the one
in Example~\ref{example:unattinabilty_C5+K1}. The second original approach uses the result from \cite{GuoW90} to construct a countably
infinite family of {\em connected} graphs whose capacity is strictly larger than the independence number of any finite (strong) power
of the graph.

The following example exhibits a disconnected graph whose capacity is not attained by the independence number of any strong finite power
of the graph. The proof of this example relies on an important concept that is later employed in our first original approach. This example,
appearing explicitly in \cite{WigdersonZ23}, was essentially already contained in \cite{GuoW90} and was later explicitly pointed out in the
1998~paper on zero-error information theory by K\"{o}rner and Orlitsky \cite{KornerO98} (at the end of Section~IV), where it is attributed
to Arikan.
\begin{example}\label{example:unattinabilty_C5+K1}
  Let $\Gr{G} = \Gr{C}_5 + \Gr{K}_1$. We have $\Gr{C}_5 \boxtimes \Gr{K}_1 \cong \Gr{C}_5$, so
  \begin{align}
  \label{eq4: 01.01.26}
    \Theta(\Gr{C}_5 \boxtimes \Gr{K}_1) = \Theta(\Gr{C}_5) = \Theta(\Gr{C}_5)\,\Theta(\Gr{K}_1).
  \end{align}
  Consequently, by Theorem~\ref{theorem: shannon_union_equivalent_product},
  \begin{align}
  \label{eq5: 01.01.26}
    \Theta(\Gr{G}) = \Theta(\Gr{C}_5) + \Theta(\Gr{K}_1) = \sqrt{5} + 1,
  \end{align}
  and, for every $k\in\naturals$,
  \begin{align}
    \Theta(\Gr{G})^k & = \bigl(\sqrt{5} + 1\bigr)^k \nonumber \\
    & = \sum_{i = 0}^{k}{\binom{k}{i} \, 5^{\frac{i}{2}}} \nonumber \\
    & = \sum_{0\leq\ell\leq\lfloor\frac{k}{2}\rfloor} \binom{k}{2\ell} \, 5^\ell
    + \sum_{0\leq\ell\leq\lfloor\frac{k - 1}{2}\rfloor}\binom{k}{2\ell+1} \, 5^{\ell + \frac12} \nonumber \\
    & = c_k + d_k\sqrt{5} \notin \naturals, \label{eq6: 01.01.26}
  \end{align}
  where
  \begin{align}
    c_k = \sum_{\ell = 0}^{\bigl \lfloor \frac{k}{2} \bigr\rfloor} \binom{k}{2\ell} \, 5^\ell \in \naturals, \qquad
    d_k = \sum_{\ell = 0}^{\bigl \lfloor \frac{k - 1}{2} \bigr\rfloor}\binom{k}{2\ell + 1} \, 5^\ell \in \naturals. \label{eq8: 01.01.26}
  \end{align}
  Therefore, it follows that for all $k\in\naturals$, $\Theta(\Gr{G})^k \neq\indnum{\Gr{G}^k}$ since the independence
  number is an integer. Consequently, \eqref{eq3:20.4.25} holds.
\end{example}
The interested reader is referred to \cite[Section~5]{CharpenayT2020}, where Example~\ref{example:unattinabilty_C5+K1} is studied in the context
of zero-error variable-length coding.

Next, we consider three approaches to the construction of families of graphs whose Shannon capacities satisfy the condition in \eqref{eq3:20.4.25}.
The first approach is due to Guo and Watanabe \cite{GuoW90} with a simplified proof (the original proof from \cite{GuoW90} is provided in
Appendix~\ref{appendix: original proof of Watanabe}), and the second and third approaches of constructing such graph families are original.

\subsection{First approach}  \label{subsection: 1st approach}

\begin{theorem}\label{theorem: universal_watanabe}
  \cite{GuoW90} Let $\Gr{G}$ be a universal graph (see Definition~\ref{definition:universal_graph}), and let $\Gr{H}$ satisfy the inequality
  $\Theta(\Gr{H}) > \indnum{\Gr{H}}$. Then, the Shannon capacity of $\Gr{K} \triangleq \Gr{G}+\Gr{H}$ is not attained at any finite power of $\Gr{K}$.
\end{theorem}
We present here a new and shortened proof, while the proof from \cite{GuoW90} is presented in Section~\ref{appendix: original proof of Watanabe}.
\begin{proof}
  Let $k\in\naturals$. By the universality of $\Gr{G}$,
  \begin{align}
    \indnum{(\Gr{G}+\Gr{H})^k} & = \sum_{i=0}^{k} \binom{k}{i} \, \indnum{\Gr{G}^i \boxtimes \Gr{H}^{k-i}} \nonumber \\
    & = \sum_{i=0}^{k} \binom{k}{i} \, \indnum{\Gr{G}}^i \, \indnum{\Gr{H}^{k-i}} \nonumber \\
    & = \sum_{i=0}^{k} \binom{k}{i} \, \Theta(\Gr{G})^i \, \indnum{\Gr{H}^{k-i}},  \label{eq1}
  \end{align}
  where the last equality holds by Lemma~\ref{lemma:universal_capacity_independent}.
  Next, by the assumption that $\Theta(\Gr{H}) > \indnum{\Gr{H}}$ and since $\Theta(\Gr{H}^m) = \Theta(\Gr{H})^m \geq \indnum{\Gr{H}}^m$ for all $m \in \naturals$
  with strict inequality if $m=1$ (by assumption), it follows that for all $k \in \naturals$,
  \begin{align}
    \sum_{i=0}^{k} \binom{k}{i} \, \Theta(\Gr{G})^i \, \indnum{\Gr{H}^{k-i}}
    & < \sum_{i=0}^{k} \binom{k}{i} \, \Theta(\Gr{G})^i \, \Theta(\Gr{H})^{k-i} \nonumber \\
    & = \bigl( \Theta(\Gr{G}) + \Theta(\Gr{H}) \bigr)^k. \label{eq2}
  \end{align}
  Finally, by Shannon's inequality (see Theorem~\ref{theorem: shannon_disjoint_union}),
  \begin{align}
    \Theta(\Gr{G})+\Theta(\Gr{H}) \leq \Theta(\Gr{G}+\Gr{H}).   \label{eq3}
  \end{align}
  Combining \eqref{eq1}--\eqref{eq3} and raising both sides of the resulting inequality to the power of $\frac{1}{k}$ gives
  \begin{align}
    \sqrt[k]{\indnum{(\Gr{G}+\Gr{H})^k}} < \Theta(\Gr{G}+\Gr{H}),  \label{eq10: 01.01.26}
  \end{align}
  thus confirming the satisfiability of the condition in \eqref{eq3:20.4.25}.
\end{proof}

\begin{corollary} \label{corollary:universal_watanabe}
  \cite{GuoW90} Let $\Gr{G}$ be a universal graph, and let $\Gr{K} = \Gr{G} + \CG{2k + 1}$ with $k\geq 2$ be the disjoint
  union of a universal graph and an odd cycle of length at least $5$. Then, $\Theta(\Gr{K})$ is not attained at any finite
  strong power of $\Gr{K}$.
\end{corollary}
\begin{proof}
  This follows from Theorem~\ref{theorem: universal_watanabe} by showing that $\Theta(\CG{2k+1}) > \indnum{\CG{2k+1}}$. For an odd cycle graph of length $2k+1$, $\indnum{\CG{2k+1}} = k$. By \cite[Theorem 7.1]{Hales73}, for all integers $j$ and $k$ such that $2 \leq j \leq k$,
  \begin{align} \label{eq12: 01.01.26}
    \indnum{\CG{2j + 1}\boxtimes\CG{2k + 1}} = jk+\Big\lfloor\frac{j}{2}\Big\rfloor,
  \end{align}
  which yields
  \begin{align}
  \Theta(\CG{2k + 1}) \geq \sqrt{\indnum{\CG{2k + 1}^2}}
  = \sqrt{k^2 + \bigg\lfloor\frac{k}{2}\bigg\rfloor}
  > k = \indnum{\CG{2k + 1}}. \label{eq13: 01.01.26}
  \end{align}
\end{proof}

\begin{remark}
  Corollary~\ref{corollary:universal_watanabe} applies in particular to $\Gr{K}_1 + \CG{5}$ since $\Gr{K}_1$ is universal.
\end{remark}

\subsection{Second approach}  \label{subsection: 2nd approach}
The next approach is based on a classical number-theoretic lemma established by Dedekind in 1858.
The lemma is stated below, and for completeness we present the simple and elegant proof given in \cite[p.~309]{AignerZ18}.
\begin{lemma}\label{lemma:dedekind}
If the square-root of a natural number is rational, then it must be an integer; equivalently, the square-root of a natural number
is either an integer or an irrational number.
\end{lemma}

\begin{proof}
Let $m \in \naturals$ be such that $\sqrt{m} \in \Rationals$.
Let $n_0 \in \naturals$ be the smallest natural number such that $n_0 \sqrt{m} \in \naturals$.
Suppose that $\sqrt{m} \notin \naturals$. Then, there exists $\ell \in \naturals$ such that $0 < \sqrt{m}-\ell < 1$.
Define $n_1 = n_0 (\sqrt{m} - \ell)$. Since $n_0 \sqrt{m} \in \naturals$ and $n_0 \ell \in \naturals$, it follows that
$n_1 \in \naturals$, and also $0 < n_1 < n_0$. In addition, $n_1 \sqrt{m} = n_0 m - \ell (n_0 \sqrt{m}) \in \naturals$,
which contradicts the minimality of $n_0$. Hence, $\sqrt{n} \in \naturals$.
\end{proof}

Next, we show the main result of this subsection.
\begin{theorem}\label{theorem: original_unattainability}
  Let $r\geq 2$, and let $\Gr{G}_1, \Gr{G}_2, \ldots, \Gr{G}_r$ be graphs such that for every $\ell\in\OneTo{r}$:
  \begin{enumerate}
    \item The graph $\Gr{G}_\ell$ is either a Kneser graph, a self-complementary strongly regular graph, or a self-complementary vertex-transitive graph.
    \item There exists a single $\ell_0 \in \OneTo{r}$ such that $\Gr{G}_{\ell_0}$ is either self-complementary vertex-transitive or self-complementary
    strongly regular on $n_{\ell_0}$ vertices, where $n_{\ell_0}$ is not the square of an integer.
  \end{enumerate}
  Let $p \in \naturals[x_1 , \ldots , x_r]$ be a polynomial with nonzero coefficients in $\naturals$ such that the following hold:
  \begin{enumerate}
    \item There exists a monomial in $p(x_1,\ldots,x_r)$ whose degree in $x_{\ell_0}$ is odd.
    \item The variable $x_{\ell_0}$ does not occur in at least one monomial of $p(x_1,\ldots,x_r)$.
  \end{enumerate}
  Let $\Gr{G} = p(\Gr{G}_1 , \ldots , \Gr{G}_r)$. Then, the Shannon capacity $\Theta(\Gr{G})$ is not attained at any finite strong power of $\Gr{G}$
  (i.e., condition \eqref{eq3:20.4.25} holds).
\end{theorem}
\begin{proof}
  For all $\ell \in \OneTo{r}$, $\Theta(\Gr{G}_\ell) = \vartheta(\Gr{G}_\ell)$ (this follows from the first assumption on $\Gr{G}_\ell$). Thus, by Corollary~\ref{corollary:capacity_lovasz_polynomial_eq},
  \begin{align} \label{eq3: 30.12.25}
    \Theta(\Gr{G}) = p\bigl(\vartheta(\Gr{G}_1), \ldots , \vartheta(\Gr{G}_r) \bigr).
  \end{align}
  Furthermore, for $\ell \in \OneTo{r}$, if $\Gr{G}_\ell$ is a Kneser graph, then $\vartheta(\Gr{G}_\ell) = \binom{n_\ell - 1}{r_\ell - 1}\in\naturals$
  (by \cite[Theorem~13]{Lovasz79}), and if $\Gr{G}_\ell$ is self-complementary and vertex-transitive or self-complementary and
  strongly regular, then $\vartheta(\Gr{G}_\ell) = \sqrt{n_\ell}$ (by \cite[Theorem~3.26]{Sason24}). By the second assumption on the graphs,
  it follows that if $\ell \neq \ell_0$, then $\vartheta(\Gr{G}_\ell) \in \naturals$, and $\vartheta(\Gr{G}_{\ell_0}) \notin \naturals$.
  Next, by the assumptions of the polynomial $p \in \naturals[x_1 , \ldots , x_r]$, there exists a monomial in $p$ whose degree in $x_{\ell_0}$
  is odd, and there exists a monomial in $p$ where the variable $x_{\ell_0}$ does not occur.
  Thus, there exist $a, b \in \naturals$ such that
  \begin{align}  \label{eq14: 01.01.26}
    p\bigl(\vartheta(\Gr{G}_1), \ldots, \vartheta(\Gr{G}_r)\bigr) = a \sqrt{n_{\ell_0}} + b.
  \end{align}
  Similarly to the concept of proof in Example~\ref{example:unattinabilty_C5+K1}, for every $k \in \naturals$,
  it follows from \eqref{eq14: 01.01.26} that there exist $c_k, d_k \in \naturals$ such that
  \begin{align}  \label{eq15: 01.01.26}
    p\bigl(\vartheta(\Gr{G}_1), \ldots , \vartheta(\Gr{G}_r)\bigr)^k = c_k \sqrt{n_{\ell_0}} + d_k.
  \end{align}
  Since $\sqrt{n_{\ell_0}} \notin \naturals$ (by assumption), we have $\sqrt{n_{\ell_0}} \notin \Rationals$ by Lemma~\ref{lemma:dedekind}.
  By \eqref{eq3: 30.12.25} and \eqref{eq15: 01.01.26}, it implies that
  \begin{align} \label{eq16: 01.01.26}
    \Theta(\Gr{G})^k \notin \naturals,
  \end{align}
  so, for all $k \in \naturals$,
  \begin{align} \label{eq17: 01.01.26}
    \Theta(\Gr{G}) > \sqrt[k]{\indnum{\Gr{G}^k}}.
  \end{align}
  This shows that the Shannon capacity of $\Gr{G}$ is not attained by the independence number of any finite strong power of $\Gr{G}$.
\end{proof}

\begin{corollary}\label{corollary:unattainability_disjoint_union}
    Let $r\geq 2$, and let $\Gr{G}_1, \Gr{G}_2, \ldots, \Gr{G}_r$ be graphs such that for every $\ell \in \OneTo{r}$:
    \begin{enumerate}
        \item The graph $\Gr{G}_\ell$ is self-complementary on $n_\ell$ vertices, and it is either strongly regular or vertex-transitive.
        \item There exists a single $\ell_0 \in \OneTo{r}$ such that $n_{\ell_0}$ is not a square of an integer.
    \end{enumerate}
    Then, the Shannon capacity of the disjoint union of these graphs is not attained at any finite strong power.
\end{corollary}
\begin{proof}
  The result is achieved by applying Theorem~\ref{theorem: original_unattainability} to the linear polynomial
  \begin{align}  \label{eq18: 01.01.26}
    p(x_1,\ldots,x_r) = \sum_{j=1}^{r}x_j.
  \end{align}
\end{proof}

\begin{example}
    The special case of $\Gr{G} = \CG{5} + \CoG{1}$ (presented earlier) is obtained by selecting in the previous corollary
    \begin{itemize}
      \item $\Gr{G}_1 = \CG{5}$, which is a self-complementary and vertex-transitive graph.
      \item $\Gr{G}_2 = \CoG{1}$, which is a Kneser graph.
      \item $r=2$.
    \end{itemize}
    Hence, the Shannon capacity of $\Gr{G}$ is not attained at any finite power of $\Gr{G}$.
\end{example}

Next, we present an additional example based on Paley graphs.
\begin{example}
Let $q_1,\ldots,q_\ell$ be integer powers of prime numbers, $p_i\equiv 1\text{ mod 4}$ for
every $i\in\OneTo{\ell}$, where only one of the $q_i$'s is an odd power of the prime $p_i$.
Define $\Gr{G}_i = P(q_i)$ (where $P(q)$ is the Paley graph of order $q$) for every $i\in\OneTo{\ell}$.
Then, the Shannon capacity of the disjoint union of these graphs,
\begin{align} \label{eq19: 01.01.26}
\Gr{G}=\Gr{G}_1 + \Gr{G}_2 + \ldots + \Gr{G}_\ell
\end{align}
is not attained by the independence number of any finite strong power of $\Gr{G}$.
This follows from Corollary~\ref{corollary:unattainability_disjoint_union}, since the Paley
graphs are self-complementary and strongly regular, and the restrictions on $q_1, \ldots,q_\ell$
guarantee that the conditions of Corollary~\ref{corollary:unattainability_disjoint_union} hold.
\end{example}

\subsection{Third approach}   \label{subsection: 3rd approach}
In \eqref{eq: 27.12.2025}, one already observes a novel example of a countably infinite family of {\em connected} graphs,
where the Shannon capacity of every such graph is not achieved by the independence number of any finite strong power.
Indeed, this holds since raising the right-hand side of \eqref{eq: 27.12.2025} to any power $k \in \naturals$ does not result in an integer
(similarly to \eqref{eq6: 01.01.26}).
In the present last approach, we build a generalized family of {\em{connected}} graphs, whose capacity is not attained
at any of their finite strong powers. In particular, we prove that an extended infinite family of Tadpole graphs
is not attained at a finite strong power.
\begin{theorem}\label{theorem: connected_unattainable}
Let $\Gr{H}$ be a graph with $\indnum{\Gr{H}}<\Theta(\Gr{H})$, let $\ell\geq2$ be an even number,
and let $v\in\V{\Gr{H}}$ be an arbitrary vertex in $\Gr{H}$. Define the graph $\Gr{G}$ as the
disjoint union of $\Gr{H}$ and $\PathG{\ell}$ with an extra edge between $v$ and one of the two
endpoints of $\PathG{\ell}$. Then, the capacity of $\Gr{G}$ is unattainable by any of its finite
strong powers.
\end{theorem}
\begin{proof}
Since $\Gr{H}+\PathG{\ell-1}$ is an induced subgraph of $\Gr{G}$, and $\Gr{H}+\PathG{\ell}$ is a spanning
subgraph of $\Gr{G}$, it follows that
\begin{align} \label{eq4}
\Theta(\Gr{H}+\PathG{\ell-1}) \leq \Theta(\Gr{G}) \leq \Theta(\Gr{H}+\PathG{\ell}).
\end{align}
Path graphs are universal, and for an even $\ell\geq2$, the capacities of $\PathG{\ell-1}$ and $\PathG{\ell}$ coincide
(see Theorem~\ref{theorem: path_invariants}).
The latter holds because path graphs are bipartite and therefore perfect, so their Shannon capacities coincide with their
independence numbers, and also $\indnum{\PathG{\ell-1}} = \frac{\ell}{2} = \indnum{\PathG{\ell}}$ if $\ell \geq 2$ is even.
By the universality of path graphs, and since the independence numbers of $\PathG{\ell}$ and $\PathG{\ell-1}$ coincide
for an even $\ell \geq 2$ (see Theorem~\ref{theorem: path_invariants}), for every $k \in \naturals$,
\begin{align}
\indnum{(\Gr{H} + \PathG{\ell-1})^k} & = \sum_{i=0}^{k} \, \binom{k}{i} \, \indnumbig{\Gr{H}^i\boxtimes\PathG{\ell-1}^{k-i}} \nonumber \\
& = \sum_{i=0}^{k} \, \binom{k}{i} \, \indnum{\Gr{H}^i} \, \indnum{\PathG{\ell-1}}^{k-i} \nonumber \\
& = \sum_{i=0}^{k} \, \binom{k}{i} \, \indnum{\Gr{H}^i} \, \indnum{\PathG{\ell}}^{k-i} \nonumber \\
& = \indnumbig{(\Gr{H}+\PathG{\ell})^k}.  \label{eq5}
\end{align}
Consequently, by raising both sides of \eqref{eq5} to the power $\frac{1}{k}$ and letting $k \to \infty$, it follows that
$\Theta(\Gr{H}+\PathG{\ell-1}) = \Theta(\Gr{H}+\PathG{\ell})$. Combining the last equality with \eqref{eq4}, hence gives
\begin{align}\label{eq3:17.8.25}
\Theta(\Gr{G}) = \Theta(\Gr{H}+\PathG{\ell}).
\end{align}
By the same argument that yields \eqref{eq4}, for all $k \in \naturals$, $(\Gr{H}+\PathG{\ell-1})^k$ is an
induced subgraph of $\Gr{G}^k$, and $(\Gr{H}+\PathG{\ell})^k$ is a spanning subgraph of $\Gr{G}^k$, so
\begin{align} \label{eq6}
\indnum{(\Gr{H}+\PathG{\ell-1})^k} \leq \indnum{\Gr{G}^k} \leq \indnum{(\Gr{H}+\PathG{\ell})^k}.
\end{align}
Thus, by combining \eqref{eq5} and \eqref{eq6}, it follows that for every $k \in \naturals$,
\begin{align}\label{eq4:17.8.25}
\indnum{\Gr{G}^k} = \indnumbig{(\Gr{H}+\PathG{\ell})^k}.
\end{align}
Finally, by Theorem~\ref{theorem: universal_watanabe} and equalities~\eqref{eq3:17.8.25} and \eqref{eq4:17.8.25}, for every $k\in\naturals$,
\begin{align}
\Theta(\Gr{G})^{\, k} = \bigl(\Theta(\Gr{H}+\PathG{\ell})\bigr)^{k} > \indnumbig{(\Gr{H}+\PathG{\ell})^k} = \indnum{\Gr{G}^k},
\end{align}
so $\Gr{G}$ is not attained at any of its finite strong powers.
\end{proof}

\begin{corollary}\label{corollary:tadpole_unattainable}
  Let $k\geq5$ be an odd number, and let $\ell\geq2$ be an even number. Then, the Shannon capacity of the Tadpole graph $\TG{k}{\ell}$ is unattainable by any of its strong powers.
\end{corollary}
\begin{proof}
This follows directly from Theorem~\ref{theorem: connected_unattainable} and Definition~\ref{definition:tadpole_graph}, by selecting $\Gr{H} = \CG{k}$ for an odd $k \geq 5$.
In the latter case, $\indnum{\CG{k}} < \Theta(\CG{k})$, as required by Theorem~\ref{theorem: connected_unattainable}.
\end{proof}

\begin{remark}\label{remark:connected_unattainable}
  Corollary \ref{corollary:tadpole_unattainable} provides a countably infinite set of {\em{connected}} graphs whose Shannon capacities are unattainable by any of its strong powers.
  This is the first infinite family of connected graphs with that property. All previous constructions with that property were disconnected graphs.
\end{remark}


\section{The Shannon capacity of \texorpdfstring{$q$}{q}-Kneser graphs}
\label{section:q-kneser}

In this section, we determine the exact Shannon capacity of the family of $q$-Kneser graphs.
For the classical Kneser graphs, the Shannon capacity was determined by Lov\'{a}sz \cite[Theorem~13]{Lovasz79} (see \eqref{eq: Shannon capacity of Kneser graphs}).
\begin{definition}\label{definition:q-gaussian_coefficient}
Let $n, k \in \naturals$ with $k \leq n$, let $p$ be a prime, let $q = p^m$ be a prime power with $m \in \naturals$, and let
$\mathbb{F}_q$ denote the Galois field of order $q$. The Gaussian coefficient, denoted by $\gbinom{n}{k}{q}$, is given by
\begin{align}\label{eq:q-gaussian_coefficient}
\gbinom{n}{k}{q} = \frac{(q^n - 1)(q^n-q)\cdots(q^n - q^{k-1})}{(q^k - 1)(q^k - q)\cdots (q^k - q^{k-1})}.
\end{align}
\end{definition}
The Gaussian coefficient admits the following combinatorial interpretation.
Let $V$ be an $n$-dimensional vector space over $\mathbb{F}_q$. Then, the number of distinct $k$-dimensional subspaces of $V$
is equal to $\gbinom{n}{k}{q}$. Consequently, we define $\gbinom{0}{0}{q} \triangleq 1$.

Letting $q$ vary continuously and taking the limit $q \to 1$ gives
\begin{align}
\lim_{q \to 1} \, \gbinom{n}{k}{q} &= \frac{n}{k} \cdot \frac{n-1}{k-1} \cdots \frac{n-(k-1)}{k-(k-1)} \nonumber \\
\label{eq: limit to binomial coeff.}
& = \binom{n}{k},
\end{align}
thus converging to the binomial coefficient.

\begin{definition}[$q$-Kneser graphs]
\label{definition:q_kneser_graph}
Let $V(n,q)$ denote the $n$-dimensional vector space over the finite field $\mathbb{F}_q$, where $q$ is a prime power.
The $q$-Kneser graph $\qKG{n}{k}{q}$ has as its vertices the $k$-dimensional subspaces of $V(n,q)$, and any two vertices
are adjacent if and only if their intersection is the zero vector.
\end{definition}
Some of the properties of the $q$-Kneser graphs are presented next (see \cite[Proposition~3.4]{Mussche09}).
\begin{theorem}\label{theorem: q-kneser_graph_properties}
  Let $\Gr{G}=\qKG{n}{k}{q}$ be a $q$-Kneser graph, where $n \geq 2k$ and $q$ is a prime power. Then, the following hold:
  \begin{enumerate}
    \item The order of $\qKG{n}{k}{q}$ is given by
    \begin{align} \label{eq: order q-Kneser}
    \card{\V{\qKG{n}{k}{q}}} = \gbinom{n}{k}{q}.
    \end{align}
    \item The size of $\qKG{n}{k}{q}$ is given by
    \begin{align}  \label{eq: size q-Kneser}
    \card{\E{\qKG{n}{k}{q}}} = \tfrac{1}{2} \, q^{k^2} \, \gbinom{n-k}{k}{q} \, \gbinom{n}{k}{q}.
    \end{align}
    \item $\qKG{n}{k}{q}$ is $d$-regular with $d = q^{k^2} \, \gbinom{n-k}{k}{q}$.
    \item $\qKG{n}{k}{q}$ is vertex-transitive and edge-transitive.
  \end{enumerate}
\end{theorem}

Next, we determine the Shannon capacity of $\qKG{n}{k}{q}$ in a manner analogous to the Lov\'{a}sz calculation
of the Shannon capacity of the Kneser graph $\KG{n}{k}$ (see \cite[Theorem~13]{Lovasz79}). To that end,
we derive the independence number and the Lov\'{a}sz $\vartheta$-function of the $q$-Kneser graph and show
that they coincide, thereby yielding the Shannon capacity of the graph.
\begin{lemma}\label{lemma:q-kneser_graph_independence_number} \cite{GodsilK16}
Let $\Gr{G}=\qKG{n}{k}{q}$ be a $q$-Kneser graph, where $n\geq 2k$ and $q$ is a prime power. Then,
\begin{align}\label{eq:q-kneser_graph_independence_number}
\indnum{\qKG{n}{k}{q}}=\gbinom{n-1}{k-1}{q}.
\end{align}
\end{lemma}
\begin{proof}
By the version of the Erd\H{o}s-Ko-Rado theorem for finite vector spaces (see \cite[Theorem~9.8.1]{GodsilK16}), it follows that
\begin{align}
\indnum{\qKG{n}{k}{q}} \leq \gbinom{n-1}{k-1}{q}.
\end{align}
Moreover, we can construct a family of $k$-subspaces of $V(n,q)$ containing a fixed $1$-dimensional subspace of $V(n,q)$.
This family of subspaces has $\gbinom{n-1}{k-1}{q}$ subspaces (see Definition~\ref{definition:q-gaussian_coefficient}
and \cite[Theorem~9.8.1]{GodsilK16}), and it is an independent set in $\qKG{n}{k}{q}$, thus
\begin{align}
\indnum{\qKG{n}{k}{q}} \geq \gbinom{n-1}{k-1}{q},
\end{align}
which proves equality~\eqref{eq:q-kneser_graph_independence_number}.
\end{proof}

\begin{lemma}\label{lemma:q-kneser_graph_lovasz_function}
  Let $\Gr{G}=\qKG{n}{k}{q}$ be a $q$-Kneser graph, where $n\geq 2k$ and $q$ is a prime power. Then,
  \begin{align}\label{eq:q-kneser_graph_lovasz_function}
    \vartheta(\qKG{n}{k}{q})=\gbinom{n-1}{k-1}{q}.
  \end{align}
\end{lemma}
\begin{proof}
By Theorem~\ref{theorem: lovasz_hoffman_bound}, and since the $q$-Kneser graphs are edge-transitive
(see Theorem~\ref{theorem: q-kneser_graph_properties}), it follows that
\begin{align}\label{eq2:24.8.25}
\vartheta(\qKG{n}{k}{q}) = -\frac{\card{\V{\qKG{n}{k}{q}}}\;\lambda_{\min}}{\lambda_{\max} - \lambda_{\min}},
\end{align}
where $\lambda_{\max}$ and $\lambda_{\min}$ are, respectively, the largest and smallest eigenvalues of the adjacency matrix of $\qKG{n}{k}{q}$.
By \cite[Theorem 2]{LvW12}, these eigenvalues are given by
\begin{align}
\lambda_{\max} & = q^{k^2} \, \gbinom{n-k}{k}{q} \label{eq3:24.8.25}\\
\lambda_{\min} & = -q^{k^2 - k} \, \gbinom{n-k-1}{k-1}{q}. \label{eq4:24.8.25}
\end{align}
Substituting \eqref{eq: order q-Kneser}, \eqref{eq3:24.8.25}, and \eqref{eq4:24.8.25} into \eqref{eq2:24.8.25} gives
\begin{align}\label{eq5:24.8.25}
\vartheta(\qKG{n}{k}{q}) & = \frac{\gbinom{n}{k}{q} \cdot q^{k^2 - k} \, \gbinom{n-k-1}{k-1}{q}}{q^{k^2} \,
\gbinom{n-k}{k}{q} + q^{k^2 - k} \, \gbinom{n-k-1}{k-1}{q}} \nonumber \\
& = \frac{\gbinom{n}{k}{q}\cdot\gbinom{n-k-1}{k-1}{q}}{q^k \, \gbinom{n-k}{k}{q}+\gbinom{n-k-1}{k-1}{q}}.
\end{align}
This expression can be simplified, based on the following identity:
\begin{align} \label{eq6:24.8.25}
\gbinom{m}{r}{q} = \frac{q^m - 1}{q^r - 1} \, \gbinom{m-1}{r-1}{q},
\end{align}
where using \eqref{eq6:24.8.25} with $m=n-k$ and $r=k$ simplifies the denominator on the right-hand side of \eqref{eq5:24.8.25} to
\begin{align} \label{eq7:24.8.25}
& q^k \, \gbinom{n-k}{k}{q} + \gbinom{n-k-1}{k-1}{q} \nonumber \\
&= q^k \cdot \frac{q^{n-k}-1}{q^k -1} \cdot \gbinom{n-k-1}{k-1}{q} + \gbinom{n-k-1}{k-1}{q} \nonumber \\
&= \frac{q^n-1}{q^k -1} \cdot \gbinom{n-k-1}{k-1}{q}.
\end{align}
Combining \eqref{eq5:24.8.25} and \eqref{eq7:24.8.25} finally gives the simplified form
\begin{align}
\vartheta(\qKG{n}{k}{q}) & = \gbinom{n}{k}{q}\cdot\frac{q^k - 1}{q^n - 1} \nonumber \\
& = \gbinom{n-1}{k-1}{q},  \label{eq20: 01.01.26}
\end{align}
where \eqref{eq20: 01.01.26} holds by \eqref{eq6:24.8.25}.
\end{proof}
Due to the coincidence of the independence number and the Lov\'{a}sz $\vartheta$-function in
Lemmata~\ref{lemma:q-kneser_graph_independence_number} and~\ref{lemma:q-kneser_graph_lovasz_function},
their common value is equal to the Shannon capacity of the graph.
This gives the following closed-form expression for the Shannon capacity of $q$-Kneser graphs.
\begin{theorem}\label{theorem: q-kneser_graph_shannon_capacity}
The Shannon capacity of the $q$-Kneser graph $\qKG{n}{k}{q}$, where $n \geq 2k$ and $q$ is a prime power, is given by
\begin{align}\label{eq:q-kneser_graph_shannon_capacity}
\Theta(\qKG{n}{k}{q}) = \gbinom{n-1}{k-1}{q}.
\end{align}
\end{theorem}


\section{A new inequality for the capacity of graphs}\label{section:inequality}
The following result provides a relation between the Shannon capacity of any strong product of graphs, and the capacity of the disjoint
union of the component graphs. If these component graphs are connected, then their strong product is a connected graph on a number of
vertices that is equal to the product of the number of vertices in each component graph, whereas the disjoint union of these component
graphs is a disconnected graph on a number of vertices that is equal to the sum of the orders of the component graphs (the latter order
is typically much smaller than the former). The motivation for our inequality comes from the following result.
\begin{theorem}[Unique Prime Factorization for Connected Graphs]
\label{theorem: prime_factorization}
Every connected graph has a unique prime factor decomposition with respect to the strong product.
\end{theorem}

The proof of Theorem~\ref{theorem: prime_factorization} was introduced by D\"{o}rfler and Imrich \cite{DorflerImrich70}, and Mckenzie \cite{Mckenzie71}.
See Section~7.3 in the comprehensive book on graph products \cite{HammackIK11} (Theorem~7.14).
In a paper by Feigenbaum and Sch\"{a}ffer \cite{FeigenbaumS91}, a polynomial-time algorithm was introduced for finding that unique
prime factorization (with respect to strong products).

Next, we provide and prove the main result of this section.
\begin{theorem}\label{theorem: the_inequality}
  Let $\Gr{G}_1, \Gr{G}_2, \ldots, \Gr{G}_\ell$ be simple graphs. Then,
  \begin{align}\label{eq:the_inequality}
    \Theta(\Gr{G}_1\boxtimes\ldots\boxtimes \Gr{G}_\ell)\leq\left(\frac{\Theta(\Gr{G}_1+\ldots+\Gr{G}_\ell)}{\ell}\right)^\ell.
  \end{align}
  Furthermore, if $\Theta(\Gr{G}_i) = \vartheta(\Gr{G}_i)$ for every $i \in \OneTo{\ell}$, then inequality~\eqref{eq:the_inequality}
  holds with equality if and only if
  \begin{align}\label{eq2:11.09.2024}
    \Theta(\Gr{G}_1)=\Theta(\Gr{G}_2)=\ldots=\Theta(\Gr{G}_\ell).
  \end{align}
  In particular, if for every $i \in \OneTo{\ell}$, one of the following statements holds:
  \begin{itemize}
    \item $\Gr{G}_i$ is a perfect graph,
    \item $\Gr{G}_i = \KG{n}{r}$ for some $n,r \in \naturals$ with $n \geq 2r$,
    \item $\Gr{G}_i = \qKG{n}{r}{q}$ for a prime factor $q$ and some $n,r \in \naturals$ with $n \geq 2r$,
    \item $\Gr{G}_i$ is vertex-transitive and self-complementary,
    \item $\Gr{G}_i$ is strongly regular and self-complementary,
  \end{itemize}
  then inequality \eqref{eq:the_inequality} holds with equality if and only if the condition in \eqref{eq2:11.09.2024} is satisfied.
\end{theorem}

\begin{proof}
  Let $k\in \naturals$. By~\eqref{eq:shannon_power},
  \begin{align}
    & \Theta(\Gr{G}_1+\ldots +\Gr{G}_\ell)^{\ell k} \nonumber \\[0.1cm]
    & = \Theta((\Gr{G}_1+\ldots+\Gr{G}_\ell)^{\ell k}) \nonumber \\
    & = \Theta\left(\sum_{k_1, \ldots, k_\ell: \; k_1 + \ldots + k_{\ell} = \ell k}{\binom{\ell k}{k_1,\ldots,k_\ell} \;
    \Gr{G}_{1}^{k_1} \boxtimes \ldots \boxtimes \Gr{G}_{\ell}^{k_{\ell}}}\right).  \label{eq7}
  \end{align}
  By~\eqref{eq:shannon_power}--\eqref{eq:shannon_strong_product},
  \begin{align}
    & \Theta\left(\sum_{k_1, \ldots, k_\ell: \; k_1 + \ldots + k_{\ell} = \ell k}{\binom{\ell k}{k_1,\ldots,k_\ell}
    \; \Gr{G}_{1}^{k_1} \boxtimes \ldots \boxtimes \Gr{G}_{\ell}^{k_\ell}}\right) \nonumber \\
    & \geq \sum_{k_1, \ldots, k_\ell: \; k_1 + \ldots + k_{\ell} = \ell k}{\binom{\ell k}{k_1,\ldots,k_\ell} \;
    \Theta \bigl(\Gr{G}_{1}^{k_1} \boxtimes \ldots \boxtimes \Gr{G}_{\ell}^{k_\ell}\bigr)} \nonumber \\[0.1cm]
    & \geq \binom{\ell k}{k,\ldots,k} \; \Theta \bigl(\Gr{G}_{1}^{k} \boxtimes \ldots \boxtimes \Gr{G}_{\ell}^{k} \bigr) \nonumber \\[0.1cm]
    & = \binom{\ell k}{k,\ldots,k} \; \Theta\bigl(\Gr{G}_1 \boxtimes \ldots \boxtimes \Gr{G}_\ell \bigr)^k,  \label{eq8}
  \end{align}
  and, thus, by combining \eqref{eq7} and \eqref{eq8}, it follows that
  \begin{equation*}
    \binom{\ell k}{k,\ldots,k} \; \Theta(\Gr{G}_1 \boxtimes \ldots \boxtimes \Gr{G}_{\ell})^{k} \leq\Theta(\Gr{G}_1+\ldots+\Gr{G}_\ell)^{\ell k}.
  \end{equation*}
  Raising both sides of the inequality to the power $\frac{1}{k}$ and letting $k$ tend to infinity gives
  \begin{align}
    \Theta(\Gr{G}_1\boxtimes\ldots\boxtimes \Gr{G}_\ell)\leq\left(\frac{\Theta(\Gr{G}_1+\ldots+\Gr{G}_\ell)}{\ell}\right)^\ell,
  \end{align}
  which holds by the equality
  \begin{align}
  \lim_{k \to \infty} \sqrt[k]{\binom{\ell k}{k,\ldots,k}} = \ell^\ell, \quad \forall \, \ell \in \naturals.
  \end{align}

  If $\Theta(\Gr{G}_i) = \vartheta(\Gr{G}_i)$ for every $i=1,\ldots,\ell$, then by Theorem~\ref{theorem: structured_first_result},
  \begin{align} \label{eq10}
    \Theta(\Gr{G}_1+\ldots + \Gr{G}_\ell) = \Theta(\Gr{G}_1)+\ldots + \Theta(\Gr{G}_\ell).
  \end{align}
  By \eqref{eq10}, inequality~\eqref{eq:the_inequality} is equivalent to
  \begin{align} \label{eq11}
    \sqrt[\ell]{\Theta(\Gr{G}_1)\ldots\Theta(\Gr{G}_\ell)} \leq \frac{\Theta(\Gr{G}_1) + \ldots + \Theta(\Gr{G}_\ell)}{\ell},
  \end{align}
  and, by the conditions for equality in the AM-GM inequality, equality holds in \eqref{eq11} if and only if \eqref{eq2:11.09.2024}
  holds. Finally, all the graphs that are listed in this theorem (Theorem~\ref{theorem: the_inequality}) satisfy the equality
  $\Theta(\Gr{G}_{i}) = \vartheta(\Gr{G}_i)$ for $i \in \OneTo{\ell}$. Thus, if $\Gr{G}_i$ is one of these graphs for each
  $i \in \OneTo{\ell}$, then inequality \eqref{eq:the_inequality} holds with equality if and only if the condition in \eqref{eq2:11.09.2024}
  is satisfied.
\end{proof}

\begin{remark}
An anonymous reviewer brought to our attention that inequality \eqref{eq:the_inequality} can be related to Eq.~(1)
in the preprint \cite{CharpenayTR26} for the complementary graph entropy, via the results of \cite[Lemma 1]{Marton93} and
\cite[Equation (11.2)]{CsiszarK01}.
\end{remark}

The next result strengthens \cite[Theorem~2.1]{Alon98} by identifying several sufficient conditions under which the inequality holds with equality.  
\begin{corollary}\label{corollary:alon_inequality}
  Let $\Gr{G}$ be a graph on $n$ vertices. Then,
  \begin{align}\label{eq4:11.09.2024}
    \Theta(\Gr{G}+\CGr{G}) \geq 2\sqrt{n},
  \end{align}
  with equality in \eqref{eq4:11.09.2024} if the graph $\Gr{G}$ is either self-complementary and vertex-transitive,
  self-complementary and strongly regular, a conference graph, a Latin square graph, or the complement of any of these graphs.
\end{corollary}
\begin{proof}
By Theorem~\ref{theorem: the_inequality},
\begin{align}\label{eq5:11.09.2024}
\Theta(\Gr{G}+\CGr{G}) \geq 2 \sqrt{\Theta(\Gr{G} \boxtimes \CGr{G})}.
\end{align}
Since $\bigl\{(1,1), \ldots, (n,n)\bigr\}$ is a trivial independent set of $\Gr{G} \boxtimes \CGr{G}$, it follows that
\begin{align}\label{eq6:11.09.2024}
\Theta(\Gr{G} \boxtimes \CGr{G}) \geq \indnum{\Gr{G} \boxtimes \CGr{G}} \geq n.
\end{align}
Combining inequalities \eqref{eq5:11.09.2024} and \eqref{eq6:11.09.2024} yields \eqref{eq4:11.09.2024}.
Furthermore, suppose that the graph $\Gr{G}$ is one of the following: (1)~self-complementary and vertex-transitive,
(2)~self-complementary and strongly regular, (3)~a conference graph, (4)~a Latin square graph, or (5)~the complement
of any of the graphs in (1)--(4). Then, by Theorem~\ref{theorem: the_inequality} and \cite[Theorems 3.23, 3.26, and 3.28]{Sason24},
inequalities \eqref{eq5:11.09.2024} and \eqref{eq6:11.09.2024} hold with equality, and consequently inequality \eqref{eq4:11.09.2024}
also holds with equality.
\end{proof}

\begin{remark}\label{remark:alon_inequality}
The original proof of inequality \eqref{eq4:11.09.2024} in \cite[Theorem~2.1]{Alon98} is presented in Appendix~\ref{appendix: Noga's proof}.
Our proof of Corollary~\ref{corollary:alon_inequality} follows a different approach, which enables us to identify sufficient conditions for 
equality in \eqref{eq4:11.09.2024}.
\end{remark}

By \eqref{eq:the_inequality}, some additional inequalities are derived in the next two corollaries.
\begin{corollary}\label{corollary:general_the_inequality}
  Let $\Gr{G}_1,\ldots,\Gr{G}_\ell$ be simple graphs, let $m_1,\ldots,m_\ell \in \naturals$, and let $m=\sqrt[\ell]{m_1 \ldots m_\ell}$. Then,
  \begin{align}\label{eq:general_the_inequality}
    \Theta(\Gr{G}_1\boxtimes\ldots\boxtimes \Gr{G}_\ell)\leq (m\ell)^{-\ell}\,\Theta(m_1 \Gr{G}_1 +\ldots+m_\ell \Gr{G}_\ell)^\ell.
  \end{align}
\end{corollary}
\begin{proof}
  By~\eqref{eq:shannon_scalar}, inequality \eqref{eq:general_the_inequality} is equivalent to
  \begin{align}
    \Theta(m_1\Gr{G}_1\boxtimes\ldots\boxtimes m_\ell \Gr{G}_\ell) \leq \ell^{-\ell}\,\Theta(m_1 \Gr{G}_1 +\ldots+m_\ell \Gr{G}_\ell)^\ell.
  \end{align}
  Thus, by Theorem~\ref{theorem: the_inequality}, inequality~\eqref{eq:general_the_inequality} holds.
\end{proof}

\begin{theorem} \label{theorem: entropy_inequality}
  Let $\Gr{G}_1,\ldots, \Gr{G}_\ell$ be simple graphs, with some $\ell \in \naturals$, and let
  $\underline{\alpha} = (\alpha_1, \ldots, \alpha_\ell)$ be a probability vector with
  $\alpha_j \in \Rationals$ for all $j\in \OneTo{\ell}$. Let
  \begin{align} \label{eq: K set}
    K(\underline{\alpha}) \triangleq \bigl\{ k \in \naturals: \, \alpha_j \, k \in \naturals, \; \forall \, j \in \OneTo{\ell} \bigr\}.
  \end{align}
  Then, for all $k\in K(\underline{\alpha})$,
  \begin{align}\label{eq:entropy_the_inequality}
    \Theta(\Gr{G}_{1}^{\alpha_1 \, k} \boxtimes \ldots \boxtimes \Gr{G}_{\ell}^{\alpha_\ell \, k})
    \leq \exp\bigl(-k \, H(\underline{\alpha}) \bigr) \, \Theta(\Gr{G}_1 + \ldots + \Gr{G}_\ell)^k,
  \end{align}
  where the entropy function $H$ is given by
  \begin{align} \label{eq: entropy function}
    H(\underline{\alpha})\triangleq -\sum_{j=1}^{\ell}\alpha_j \log{\alpha_j}, \quad \forall \, \underline{\alpha}: \,
    \alpha_1, \ldots, \alpha_\ell \geq 0, \; \; \sum_{j=1}^\ell \alpha_j = 1.
  \end{align}
\end{theorem}

\begin{proof}
  Let
  \begin{equation*}
    A \triangleq \lcm(\alpha_{1}k,\ldots,\alpha_{\ell}k),
  \end{equation*}
  and, for every $j \in \OneTo{\ell}$, define
  \begin{align}
    & n_j \triangleq \frac{A}{\alpha_j \, k}, \label{eq1:21.11.2024} \\
    & a_j \triangleq \sum_{i=1}^j \alpha_i \, k \in \naturals,  \label{eq3:28.11.2024}
  \end{align}
  with $a_0 \triangleq 0$.
  For each $i \in \OneTo{k}$, let $j \in \OneTo{\ell}$ be the index that satisfies
  \begin{align}
  a_{j-1} + 1 \leq i \leq a_j,  \label{eq: j index}
  \end{align}
  and define
  \begin{align} \label{eq: m_i}
    m_i \triangleq n_j, \quad \forall \, i \in \OneTo{k}.
  \end{align}
  By Corollary~\ref{corollary:general_the_inequality} (note that the strong product in the next line involves $k$ graphs),
  \begin{align}
    & \Theta\bigl(\Gr{G}_{1}^{\alpha_1 k} \boxtimes \ldots \boxtimes \Gr{G}_\ell^{\alpha_\ell k} \bigr) \nonumber \\
    & \leq (mk)^{-k}\,\Theta \bigl(m_1 \Gr{G}_1+\ldots+m_{a_1}\Gr{G}_1+ \ldots + m_{a_{\ell-1} +1}\Gr{G}_\ell +\ldots +m_k \Gr{G}_\ell \bigr)^k \nonumber \\
    & = (mk)^{-k} \, \Theta\bigl((m_1 + \ldots + m_{a_1}) \, \Gr{G}_1 + \ldots + (m_{a_{\ell - 1} + 1} + \ldots + m_k) \, \Gr{G}_\ell \bigr)^k  \nonumber \\
    & = (mk)^{-k} \, \Theta(\alpha_1 \, n_1 \, k \Gr{G}_1+\ldots + \alpha_{\ell} \, n_\ell \, k \Gr{G}_\ell)^k,  \label{eq: UB Theta}
  \end{align}
  where the last equality holds since, by \eqref{eq1:21.11.2024} and \eqref{eq3:28.11.2024},
  \begin{align}
  & m_1 = \ldots = m_{a_1} = n_1, \label{eq1: m} \\
  & m_{a_1+1} = \ldots = m_{a_2} = n_2, \label{eq2: m} \\
  & \vdots \nonumber \\
  & m_{a_{\ell-1}+1} = \ldots = m_k = n_\ell, \quad a_\ell=k. \label{eq3: m}
  \end{align}
  So, for every $j \in \OneTo{\ell}$,
  \begin{align}
  m_{a_{j-1} + 1} + \ldots + m_{a_j} &= (a_j - a_{j-1}) \, n_j \nonumber \\
  &= \alpha_j \, n_j \, k.  \label{eq4: m}
  \end{align}
  Calculation of $m$ on the right-hand side of \eqref{eq: UB Theta} gives, by \eqref{eq1:21.11.2024} and \eqref{eq1: m}--\eqref{eq4: m},
  \begin{align}
    m & = \sqrt[k]{m_1 \ldots m_k} \nonumber \\
    & = \sqrt[k]{n_{1}^{a_1 - a_0} \ldots n_{\ell}^{a_{\ell} - a_{\ell-1}}} \nonumber \\
    & = \sqrt[k]{n_{1}^{\alpha_{1}k}\ldots n_{\ell}^{\alpha_{\ell}k}} \nonumber \\
    & = \sqrt[k]{\left(\frac{A}{\alpha_{1}k}\right)^{\alpha_{1}k}\ldots \left(\frac{A}{\alpha_{\ell}k}\right)^{\alpha_{\ell}k}} \nonumber \\
    & = \frac{A}{k} \left(\prod_{i=1}^{\ell}{\alpha_{i}^{\alpha_{i}k}}\right)^{-\frac{1}{k}},  \label{eq5: m}
  \end{align}
  and, also by \eqref{eq:shannon_scalar} and \eqref{eq1:21.11.2024},
  \begin{align}
    \Theta(n_1 \alpha_{1}k\Gr{G}_1+\ldots+n_\ell \alpha_{\ell}k\Gr{G}_\ell)^k & = \Theta(A\Gr{G}_1 + \ldots + A\Gr{G}_\ell)^k \nonumber \\
    & = A^k \, \Theta(\Gr{G}_1+\ldots+\Gr{G}_\ell)^k. \label{eq6: 30.12.25}
  \end{align}
  Finally, combining \eqref{eq: UB Theta}, \eqref{eq5: m}, and \eqref{eq6: 30.12.25}, we get
  \begin{align}
    \Theta(\Gr{G}_{1}^{\alpha_{1}k}\boxtimes\ldots\boxtimes \Gr{G}_{\ell}^{\alpha_{\ell}k}) & \leq (mk)^{-k} A^k\,\Theta(\Gr{G}_1+\ldots+\Gr{G}_\ell)^k \nonumber \\
    & = \left(A\left(\prod_{i=1}^{\ell}{\alpha_{i}^{\alpha_{i}k}}\right)^{-\frac{1}{k}}\right)^{-k} A^k\,\Theta(\Gr{G}_1+\ldots+\Gr{G}_\ell)^k \nonumber \\
    & = \left(\prod_{i=1}^{\ell}{\alpha_{i}^{\alpha_{i}k}}\right) \Theta(\Gr{G}_1+\ldots+\Gr{G}_\ell)^k \nonumber \\
    & = \exp\bigl(-k \, H(\underline{\alpha})\bigr)\,\Theta(\Gr{G}_1+\ldots+\Gr{G}_\ell)^k.
  \end{align}
\end{proof}

\begin{remark}
  Let $\alpha_j = \frac{p_j}{q_j}$ with $(p_j , q_j) = 1$ for all $j\in \OneTo{\ell}$. Then,
  \begin{equation*}
    K(\underline{\alpha})=\lcm(q_1,\ldots,q_\ell) \, \naturals.
  \end{equation*}
  This is true because if we choose $k$ that is not a multiple of $\lcm(q_1,\ldots,q_\ell)$,
  then there exists $i \in \OneTo{\ell}$ such that $\alpha_i \notin \naturals$.
\end{remark}

In analogy to Corollary~\ref{corollary:alon_inequality}, we next derive upper and lower bounds on $\vartheta(\Gr{G}+\CGr{G})$.
\begin{theorem}\label{lovasz_complement_union_bound}
  For every simple graph $\Gr{G}$ on $n$ vertices
  \begin{align}\label{eq1:24.11.2024}
    \vartheta(\Gr{G}+\CGr{G})\geq 2\sqrt{n} + \frac{(\vartheta(\Gr{G})-\sqrt{n})^2}{\vartheta(\Gr{G})},
  \end{align}
  and for every $d$-regular graph $\Gr{G}$ whose eigenvalues are ordered in non-increasing order as
  \begin{align}
  \lambda_1 \geq\lambda_2 \geq\ldots\geq\lambda_n,
  \end{align}
  the following holds:
  \begin{align}
    2+\frac{n-d-1}{1+\lambda_2}-\frac{d}{\lambda_n} & \leq \vartheta(\Gr{G}+\CGr{G}) \nonumber \\
    & \leq  \frac{n(1+\lambda_2)}{n-d+\lambda_2}-\frac{n\lambda_n}{d-\lambda_n}.  \label{eq9:28.11.2024}
  \end{align}
  Furthermore,
  \begin{enumerate}
    \item If $\Gr{G}$ is either vertex-transitive or strongly regular, then inequality \eqref{eq1:24.11.2024} holds with equality.
    \item If $\Gr{G}$ is strongly regular, then the first inequality in \eqref{eq9:28.11.2024} is attained with equality.
    \item If both $\Gr{G}$ and $\CGr{G}$ are edge-transitive, or if $\Gr{G}$ is strongly regular, then the second inequality
    in \eqref{eq9:28.11.2024} is attained with equality.
  \end{enumerate}
\end{theorem}

\begin{proof}
  Using Theorems~\ref{theorem: lovasz_disjoint_union} and \ref{theorem: lovasz_bigger_n}, we get
  \begin{align}\label{eq2:24:11:2024}
    \vartheta(\Gr{G} + \CGr{G}) & = \vartheta(\Gr{G}) + \vartheta(\CGr{G})\nonumber \\[0.1cm]
     & \geq \vartheta(\Gr{G}) + \frac{n}{\vartheta(\Gr{G})}\nonumber \\[0.1cm]
     & = 2\sqrt{n} + \frac{(\vartheta(\Gr{G})-\sqrt{n})^2}{\vartheta(\Gr{G})}.
  \end{align}
  In addition, by Theorem~\ref{theorem: lovasz_bigger_n}, inequality~\eqref{eq2:24:11:2024} holds with equality if $\Gr{G}$ is vertex-transitive or strongly regular.
  Furthermore, if $\Gr{G}$ is a $d$-regular graph, then by \cite[Proposiotion~1]{Sason23},
  \begin{align}\label{eq8:28.11.24}
    \vartheta(\Gr{G}+\CGr{G}) = \vartheta(\Gr{G}) + \vartheta(\CGr{G})
    \leq \frac{n(1+\lambda_2)}{n-d+\lambda_2}-\frac{n\lambda_n}{d-\lambda_n},
  \end{align}
  and
  \begin{align}\label{eq1:8.12.24}
    \vartheta(\Gr{G}+\CGr{G}) & = \vartheta(\Gr{G}) + \vartheta(\CGr{G})\nonumber \\
    & \geq 2+\frac{n-d-1}{1+\lambda_2}-\frac{d}{\lambda_n},
  \end{align}
  which proves \eqref{eq9:28.11.2024}.
  Finally, by \cite[Proposition~1]{Sason23}, if both $\Gr{G}$ and $\CGr{G}$ are edge-transitive, or if $\Gr{G}$ is strongly regular, then
  inequality \eqref{eq8:28.11.24} holds with equality. Moreover, if $\Gr{G}$ is strongly regular, then inequality \eqref{eq1:8.12.24} also holds with equality.
  This completes the proof of the stated sufficient conditions for the attainability with equality of the two inequalities in \eqref{eq9:28.11.2024}.
\end{proof}

\section{Outlook} \label{section: outlook}
This section suggests some potential directions for further research that are related to the findings in this paper.
\begin{enumerate}
\item In Theorem~\ref{theorem: original_unattainability}, a construction of graphs was provided whose Shannon
capacity is not attained by the independence number of any of their finite strong powers. The proof relied on
Dedekind's lemma from number theory (see Lemma~\ref{lemma:dedekind}), showing that the capacity of this construction
equals a value whose finite powers are not natural numbers, and hence cannot correspond to the finite root of an
independence number. This method of proof also applies to rational numbers that are not integers, since none of
their finite powers are natural numbers either. This raises an interesting question: {\em Does there exist a graph
whose Shannon capacity is a rational number that is not an integer?}. At present, no such graphs are known. However,
if the answer is positive, it would immediately follow that these graphs also possess the property of having a
Shannon capacity that is not attained by the independence number of any of their finite strong powers (thus potentially
leading to a fourth approach in Section~\ref{section: unattainability of Shannon capacity at any finite power}).
\item It was proved in \cite{CsonkaS24} that if the Shannon capacity of a graph is attained at some finite power,
then the Shannon capacity of its Mycielskian is strictly larger than that of the original graph. In view of the
constructions presented in Section~\ref{section: unattainability of Shannon capacity at any finite power}, which
yield graph families whose Shannon capacity is not attained at any finite power, it would be interesting to determine
whether this property also holds in such cases. If it does not, then the graph constructions from
Section~\ref{section: unattainability of Shannon capacity at any finite power} could provide potential candidates
for a counterexample.
\item In light of earlier studies on the Shannon capacity of graphs, as well as the present work, it remains unknown
whether there exists a finite, undirected, and simple graph whose Shannon capacity is attained by the independence number
of some finite strong power, but not by that of the first or second strong powers. Further study of this problem is therefore of interest.
\item By combining Theorems~\ref{theorem: fractional_independence_strong_product} and~\ref{theorem: universal_findnum},
the equality $\indnum{\Gr{G}\boxtimes\Gr{H}} = \indnum{\Gr{G}}\,\indnum{\Gr{H}}$ holds for every simple graph $\Gr{H}$ if
$\indnum{\Gr{G}} = \findnum{\Gr{G}}$. Moreover, by Theorem~\ref{theorem: shannon_union_equivalent_product}
and Lemma~\ref{lemma:structured_second_result}, if $\Theta(\Gr{G})=\findnum{\Gr{G}}$, then
$\Theta(\Gr{G}\boxtimes\Gr{H}) = \Theta(\Gr{G})\,\Theta(\Gr{H})$ holds for all $\Gr{H}$. From these results,
a natural question arises: {\em Is it true that the equality $\Theta(\Gr{G}\boxtimes\Gr{H}) =
\Theta(\Gr{G})\,\Theta(\Gr{H})$ holds for all $\Gr{H}$ if and only if $\Theta(\Gr{G})=\findnum{\Gr{G}}$?}.
This question was already discussed to some extent in \cite{Acin17}, showing that
\begin{align}  \label{eq12}
\sup_{\Gr{H}}\frac{\Theta(\Gr{G}\boxtimes\Gr{H})}{\Theta(\Gr{H})} \leq \findnum{\Gr{G}},
\end{align}
while raising a question about a possible gap between the two sides of~\eqref{eq12}.
\item The new inequality in Theorem~\ref{theorem: the_inequality} can be viewed as an analogue of the arithmetic–geometric mean
inequality for the Shannon capacity of graphs, relating the Shannon capacity of the strong product of graphs to that of their
disjoint union. Part of the interest in this inequality arises from the fact that every finite, undirected, simple, and connected
graph admits a unique prime factorization with respect to the strong product (see Theorem~\ref{theorem: prime_factorization}).
Several applications of the inequality in Theorem~\ref{theorem: the_inequality} are presented in Section~\ref{section:inequality},
and additional applications are likely to follow.
\item This last open issue, while not directly related to the results of this paper, is mentioned here because of its fundamental importance.
The random graph $G(n, \tfrac12)$ is a graph on $n$ vertices in which each pair of distinct vertices is joined by
an edge independently with probability $\tfrac12$. For sufficiently large $n$, with high probability, the independence number of
such a graph is equal to $2 \log_2 n + O(\log \log n)$ \cite[Chapter~11]{Bollobas01}, and its Lov\'{a}sz $\vartheta$-function
ranges between $\bigl(\tfrac12 + o(1) \bigr) \sqrt{n}$ and $\bigl(2 + o(1) \bigr) \sqrt{n}$ \cite{Juhasz82}.
In contrast, a comparable probabilistic result for the Shannon capacity of $G(n, \tfrac12)$ is yet unknown. It has been conjectured
by Alon that there exists a positive constant $b$, independent of $n$, such that the Shannon capacity of $G(n, \tfrac12)$ is at most
$b \log_2 n$ almost surely (i.e., with probability tending to~1 as $n \to \infty$) \cite[Conjecture~2.4]{Alon19}.
\end{enumerate}
This sample of open problems illustrates that, despite the existence of classical results on the Shannon capacity of graphs and a
growing body of recent work, many fundamental avenues of research remain open.

\bigskip
\begin{center}
{\bf{\Large{Appendices}}}
\end{center}

\appendix


\section{Proof of the duality result in Theorem~\ref{theorem: shannon_union_equivalent_product}}
\label{appendix: proof of duality theorem}
\setcounter{equation}{0}
\renewcommand{\theequation}{\thesection.\arabic{equation}}

  By Theorems~\ref{theorem: shannon_strong_product} and \ref{theorem: shannon_disjoint_union}, the claim is equivalent to
  \begin{align}\label{eq1:28.11.2024}
    \Theta(\Gr{G}+\Gr{H}) > \Theta(\Gr{G}) + \Theta(\Gr{H}) \iff \Theta(\Gr{G}\boxtimes \Gr{H}) > \Theta(\Gr{G}) \, \Theta(\Gr{H}).
  \end{align}
  Next, we prove both directions of the equivalence in \eqref{eq1:28.11.2024}.
  \begin{enumerate}
  \item Assume that $\Theta(\Gr{G}\boxtimes \Gr{H}) > \Theta(\Gr{G}) \, \Theta(\Gr{H})$. By equality~\eqref{eq:shannon_power},
  \begin{align}
    & \hspace*{-0.2cm} \Theta(\Gr{G}+\Gr{H})^2 \nonumber \\
    & = \Theta((\Gr{G}+\Gr{H})^2) \nonumber \\
    & = \Theta(\Gr{G}^2 + 2\Gr{G}\boxtimes\Gr{H} + \Gr{H}^2),  \label{eq2: 24.12.25}
  \end{align}
  where \eqref{eq2: 24.12.25} holds by the commutative semiring argument in \cite{Schrijver23}.
  By Theorems~\ref{theorem: shannon_strong_product} and \ref{theorem: shannon_disjoint_union},
  \begin{align}
    \Theta(\Gr{G}^2 + 2\Gr{G}\boxtimes\Gr{H} + \Gr{H}^2) \geq \Theta(\Gr{G})^2 + 2\Theta(\Gr{G}\boxtimes \Gr{H}) + \Theta(\Gr{H})^2.
  \end{align}
  Using the above assumption, we get
  \begin{align}
    & \hspace*{-0.2cm} \Theta(\Gr{G})^2 + 2\Theta(\Gr{G}\boxtimes \Gr{H}) + \Theta(\Gr{H})^2 \nonumber  \\
    & > \Theta(\Gr{G})^2 +2\Theta(\Gr{G}) \, \Theta(\Gr{H}) + \Theta(\Gr{H})^2 \nonumber \\
    & = (\Theta(\Gr{G}) + \Theta(\Gr{H}))^2,
  \end{align}
  which gives
  $\Theta(\Gr{G}+\Gr{H}) > \Theta(\Gr{G}) + \Theta(\Gr{H})$.
  \item Next, assume $\Theta(\Gr{G}\boxtimes \Gr{H}) = \Theta(\Gr{G}) \, \Theta(\Gr{H})$. Then, for all $i,j \in \naturals$, we get
  \begin{align}
    & \hspace*{-0.2cm} \Theta(\Gr{G}^i\boxtimes \Gr{H}^j)\,\Theta(\Gr{G})^j\,\Theta(\Gr{H})^i & \nonumber \\
    & = \Theta(\Gr{G}^i\boxtimes \Gr{H}^j)\,\Theta(\Gr{G}^j)\,\Theta(\Gr{H}^i) \nonumber \\
    & \leq \Theta(\Gr{G}^{i+j}\boxtimes \Gr{H}^{i+j}) \nonumber \\
    & = \Theta(\Gr{G}\boxtimes \Gr{H})^{i+j} \nonumber \\
    & = \Theta(\Gr{G})^{i+j}\,\Theta(\Gr{H})^{i+j}.
  \end{align}
  Thus,
  $\Theta(\Gr{G}^i\boxtimes \Gr{H}^j) \leq \Theta(\Gr{G})^i \, \Theta(\Gr{H})^j$,
  which implies that for all $k\in \naturals$,
  \begin{align}
    \indnumbig{(\Gr{G}+\Gr{H})^k} & =\indnumBigg{\sum_{\ell = 0}^{k}\binom{k}{\ell}\Gr{G}^\ell\boxtimes \Gr{H}^{k-\ell}} \nonumber \\
    & = \sum_{\ell = 0}^{k}\binom{k}{\ell}\indnum{\Gr{G}^\ell\boxtimes \Gr{H}^{k-\ell}} \nonumber \\
    & \leq \sum_{\ell = 0}^{k}\binom{k}{\ell}\Theta\left(\Gr{G}^\ell\boxtimes \Gr{H}^{k-\ell}\right) \nonumber \\
    & \leq \sum_{\ell = 0}^{k}\binom{k}{\ell}\Theta\left(\Gr{G}^\ell\right)\,\Theta\left(\Gr{H}^{k-\ell}\right) \nonumber \\
    & = \sum_{\ell = 0}^{k}\binom{k}{\ell}\Theta(\Gr{G})^\ell\,\Theta(\Gr{H})^{k-\ell} \nonumber \\
    & = \left(\Theta(\Gr{G})+\Theta(\Gr{H})\right)^k.
  \end{align}
  Finally, letting $k\to\infty$ gives
  $\Theta(\Gr{G}+\Gr{H})\leq\Theta(\Gr{G})+\Theta(\Gr{H})$,
  and by Theorem~\ref{theorem: shannon_disjoint_union},
  \begin{align}
  \Theta(\Gr{G}+\Gr{H})=\Theta(\Gr{G})+\Theta(\Gr{H}).
  \end{align}
  \end{enumerate}


\section{A new proof of Theorem~\ref{theorem: lovasz_disjoint_union}}
\label{appendix: original proof of Knuth}
\setcounter{equation}{0}
\renewcommand{\theequation}{\thesection.\arabic{equation}}

\begin{theorem}
  \cite{Knuth94} Let $\Gr{G}$ and $\Gr{H}$ be simple graphs. Then,
  \begin{align}
    \vartheta(\Gr{G} + \Gr{H}) = \vartheta(\Gr{G}) + \vartheta(\Gr{H}).
  \end{align}
\end{theorem}
\begin{proof}
  By Theorem~\ref{theorem: lovasz_equivalent_complement_formula}, let $(\mathbf{u}_1,\ldots,\mathbf{u}_n)$ be an
  orthonormal representation of $\CGr{G}$, and let $\mathbf{c}$ be a unit vector such that
  \begin{align}
  \label{eq1: 06.05.25}
    \vartheta(\Gr{G}) = \sum_{i=1}^{n}(\mathbf{c}^{\mathrm{T}} \mathbf{u}_i)^2.
  \end{align}
  Likewise, let $(\mathbf{v}_1,\ldots,\mathbf{v}_r)$ be an orthonormal representation of $\CGr{H}$, and let $\mathbf{d}$
  be a unit vector such that
  \begin{align}
  \label{eq2: 06.05.25}
    \vartheta(\Gr{H}) = \sum_{i=1}^{r}(\mathbf{d}^{\mathrm{T}} \mathbf{v}_i)^2.
  \end{align}
  Assume without loss of generality that the dimensions of $\mathbf{u}_i,\mathbf{c},\mathbf{v}_j$, and $\mathbf{d}$ are identical and equal to $m$
  (if the dimensions are distinct, the vectors of the lower dimension can be padded by zeros). Next, let $\mathbf{A}$ be an
  orthogonal matrix of order $m \times m$ such that
  ${\mathbf{A}} {\mathbf{d}} = {\mathbf{c}}$.
  Such a matrix $\mathbf{A}$, satisfying ${\mathbf{A}}^{\mathrm{T}} {\mathbf{A}} = \Id{m}$, exists (e.g., the householder
  matrix defined as
  ${\mathbf{A}} = \Id{m} - \frac{2(\mathbf{c}-\mathbf{d})(\mathbf{c}-\mathbf{d})^{\mathrm{T}}}{\|\mathbf{c}-\mathbf{d}\|^2}$
  provided that $\mathbf{c} \neq \mathbf{d}$, and ${\mathbf{A}} = \Id{m}$ if $\mathbf{c} = \mathbf{d}$).
   Let $w = (\mathbf{w}_1,\ldots,\mathbf{w}_r)$ be defined as $\mathbf{w}_i = \mathbf{A}\mathbf{v}_i$ for every $i\in\OneTo{r}$.
   Since the pairwise inner products are preserved under an orthogonal (orthonormal) transformation and $(\mathbf{v}_1, \ldots, \mathbf{v}_r)$
   is an orthonormal representation of $\CGr{H}$, so is $(\mathbf{w}_1,\ldots,\mathbf{w}_r)$. Likewise, we get
  \begin{align}
  \label{eq4: 06.05.25}
    \sum_{i=1}^{r}(\mathbf{c}^{\mathrm{T}} \mathbf{w}_i)^2=\sum_{i=1}^{r}((\mathbf{A}\mathbf{d})^{\mathrm{T}} (\mathbf{A}\mathbf{v}_i))^2
    = \sum_{i=1}^{r}(\mathbf{d}^{\mathrm{T}} \mathbf{v}_i)^2 = \vartheta(\Gr{H}).
  \end{align}
  Next, the representation $(\mathbf{x}_1,\ldots,\mathbf{x}_{n+r}) = (\mathbf{u}_1, \ldots, \mathbf{u}_n, \mathbf{w}_1, \ldots, \mathbf{w}_r)$
  is an orthonormal representation of $\overline{\Gr{G}+\Gr{H}}$ (since there are no additional non-adjacencies in the graph $\overline{\Gr{G}+\Gr{H}}$
  in comparison to the disjoint union of the pairs of nonadjacent vertices in $\CGr{G}$ and $\CGr{H}$). Hence, by
  Theorem~\ref{theorem: lovasz_equivalent_complement_formula}, the equality
  \begin{align}
  \label{eq5: 06.05.25}
    \vartheta(\Gr{G}) + \vartheta(\Gr{H}) = \sum_{i=1}^{n}(\mathbf{c}^{\mathrm{T}} \mathbf{u}_i)^2 + \sum_{i=1}^{r}(\mathbf{c}^{\mathrm{T}} \mathbf{w}_i)^2 = \sum_{i=1}^{n+r}(\mathbf{c}^{\mathrm{T}} \mathbf{x}_i)^2,
  \end{align}
  yields the inequality
  \begin{align}
  \label{eq1:5.5.25}
    \vartheta(\Gr{G}) + \vartheta(\Gr{H}) \leq \vartheta(\Gr{G} + \Gr{H}).
  \end{align}
  Next, by Theorem~\ref{theorem: lovasz_equivalent_complement_formula}, let $(\mathbf{u}_1,\ldots,\mathbf{u}_n,\mathbf{v}_1,\ldots,\mathbf{v}_r)$
  be an orthonormal representation of $\overline{\Gr{G}+\Gr{H}}$, where the vectors $(\mathbf{u}_1,\ldots,\mathbf{u}_n)$ correspond to the
  vertices of $\Gr{G}$ and the vectors $(\mathbf{v}_1,\ldots,\mathbf{v}_r)$ correspond to the vertices of $\Gr{H}$, and let $\mathbf{c}$ be a
  unit vector such that
  \begin{align}
  \label{eq6: 06.05.25}
    \vartheta(\Gr{G}+\Gr{H}) = \sum_{i=1}^{n}(\mathbf{c}^{\mathrm{T}} \mathbf{u}_i)^2
    + \sum_{i=1}^{r} (\mathbf{c}^{\mathrm{T}} \mathbf{v}_i)^2.
  \end{align}
   By definition, since $(\mathbf{u}_1,\ldots,\mathbf{u}_n,\mathbf{v}_1,\ldots,\mathbf{v}_r)$ is an orthonormal representation of
   $\overline{\Gr{G}+\Gr{H}}$, it follows that if $i$ and $j$ are nonadjacent vertices in $\CGr{G}$, then they are nonadjacent in
   $\overline{\Gr{G}+\Gr{H}}$ and thus, $\mathbf{u}_i^{\mathrm{T}} \mathbf{u}_j = 0$. Hence, $(\mathbf{u}_1,\ldots,\mathbf{u}_n)$
   is an orthonormal representation of $\CGr{G}$. Similarly, it follows that $(\mathbf{v}_1,\ldots,\mathbf{v}_r)$ is an orthonormal
   representation of $\CGr{H}$. Thus, by Theorem~\ref{theorem: lovasz_equivalent_complement_formula},
  \begin{align}\label{eq2:5.5.25}
    \vartheta(\Gr{G}+\Gr{H}) = \sum_{i=1}^{n} (\mathbf{c}^{\mathrm{T}}\mathbf{u}_i)^2 + \sum_{i=1}^{r} (\mathbf{c}^{\mathrm{T}}
    \mathbf{v}_i)^2 \leq \vartheta(\Gr{G}) + \vartheta(\Gr{H}).
  \end{align}
  Combining inequalities~\eqref{eq1:5.5.25} and \eqref{eq2:5.5.25} gives the equality in \eqref{eq:lovasz_disjoint_union}.
\end{proof}


\section{Proof of Theorem~\ref{theorem: shannon_condition}}
\label{appendix: original proof by Shannon}
\setcounter{equation}{0}
\renewcommand{\theequation}{\thesection.\arabic{equation}}

Theorem~\ref{theorem: shannon_condition} is a direct corollary of \cite[Theorem 4]{Shannon56}. To show this, we define an adjacency-reducing mapping.
\begin{definition}
  Let $\Gr{G}$ be a simple graph, and let $f \colon \V{\Gr{G}} \to \set{A}$ be a mapping from the vertices of $\Gr{G}$ to a subset $\set{A}$ of $\V{\Gr{G}}$.
  The mapping $f$ is called {\em adjacency-reducing} if for every pair of nonadjacent vertices $u, v\in \V{\Gr{G}}$, the vertices $f(u)$ and $f(v)$ are
  also nonadjacent.
\end{definition}

\begin{proof}
Let $\set{U} = \{u_1,u_2,\ldots,u_k\}$ be a maximal independent set of vertices in $\Gr{G}_1$. By assumption,
$\V{\Gr{G}_1}$ can be partitioned into $k$ cliques. Let $\set{C} = \{C_1,\ldots,C_k\}$ be such cliques.
Obviously, every $u_i,u_j\in \set{U}$ cannot be in the same clique because they are nonadjacent. So, every clique
in $\set{C}$ contains exactly one vertex from $\set{U}$. For simplicity, assume that $\set{U}$ and $\set{C}$ are
ordered in a way that $u_j \in C_j$ for every $j \in \OneTo{k}$.
Next, define the mapping $f:\V{\Gr{G}_1}\to \set{U}$ as follows. Let $v\in \V{\Gr{G}_1}$ be a vertex, and let
$C_j$ be the clique that has $v\in C_j$. Then, $f(v) = u_j$.
If $v_1$ and $v_2$ are nonadjacent, then they belong to different cliques in $\set{C}$. Thus, $f(v_1)$ and $f(v_2)$
are mapped to different vertices in $\set{U}$. Since $\set{U}$ is an independent set, $f(v_1)$ and $f(v_2)$ are
nonadjacent. Thus, $f$ is an adjacency-reducing mapping of $\Gr{G}_1$ into $\set{U}$. Then, by \cite[Theorem~4]{Shannon56},
equality~\eqref{eq:lemma_structured_second_result} holds.
\end{proof}


\section{The original proof from \texorpdfstring{\cite{GuoW90}}{} of Theorem~\ref{theorem: universal_watanabe}}
\label{appendix: original proof of Watanabe}
\setcounter{equation}{0}
\renewcommand{\theequation}{\thesection.\arabic{equation}}

\begin{theorem}
  \cite{GuoW90} Let $\Gr{G}$ be a universal graph, and let $\Gr{H}$ satisfy $\Theta(\Gr{H}) > \indnum{\Gr{H}}$. The Shannon capacity of $\Gr{K}\triangleq\Gr{G}+\Gr{H}$ is not attained at any finite power of $\Gr{K}$.
\end{theorem}
\begin{proof}
  Let $\Gr{G}$ be a universal graph, and let $k \in \naturals$. Then, since $\Gr{G}$ is universal,
  \begin{align}
    \indnum{\Gr{K}^{2k}} & = \sum_{\ell=0}^{2k} \, \binom{2k}{\ell} \, \indnum{\Gr{G}^{2k-\ell}} \, \indnum{\Gr{H}^\ell} \nonumber \\
    & = \sum_{\ell=0}^{2k}\sum_{0\leq i,j\leq k ; i+j=\ell}\binom{k}{i} \, \binom{k}{j} \, \indnum{\Gr{G}^{2k-(i+j)}} \, \indnum{\Gr{H}^{i+j}} \nonumber \\
    & = \sum_{\ell=0}^{2k}\sum_{0\leq i,j\leq k ; i+j=\ell}\binom{k}{i} \, \binom{k}{j} \, \indnum{\Gr{G}}^{2k-(i+j)} \, \indnum{\Gr{H}^{i+j}} \nonumber \\
    & = \sum_{i=0}^{k}\sum_{j=0}^{k} \binom{k}{i} \, \binom{k}{j} \, \indnum{\Gr{G}}^{2k-i-j} \, \indnum{\Gr{H}^{i+j}}.  \label{eq1: 07.01.26}
  \end{align}
  Similarly,
  \begin{align}
    \indnum{\Gr{K}^k}^2 & =\left(\sum_{i=0}^{k} \binom{k}{i} \, \indnum{\Gr{G}}^{k-i} \, \indnum{\Gr{H}^i}\right)^2 \nonumber \\
    & = \sum_{i=0}^{k}\sum_{j=0}^{k} \binom{k}{i} \, \binom{k}{j} \, \indnum{\Gr{G}}^{2k-i-j} \, \indnum{\Gr{H}^i} \, \indnum{\Gr{H}^j}. \label{eq2: 07.01.26}
  \end{align}
  Subtracting \eqref{eq2: 07.01.26} from \eqref{eq1: 07.01.26} gives
  \begin{align}
  \label{eq3: 07.01.26}
    \indnum{\Gr{K}^{2k}} - \indnum{\Gr{K}^k}^2
    = \sum_{i=0}^k \sum_{j=0}^k \binom{k}{i} \, \binom{k}{j} \, \indnum{\Gr{G}}^{2k-i-j} \, \left(\indnum{\Gr{H}^{i+j}}-\indnum{\Gr{H}^i} \, \indnum{\Gr{H}^j} \right).
  \end{align}
  Since by assumption $\Theta(\Gr{H}) > \indnum{\Gr{H}}$, there exists $i_0,j_0\in\naturals$ such that
  \begin{align}
    \indnum{\Gr{H}^{i_0+j_0}}-\indnum{\Gr{H}^{i_0}} \, \indnum{\Gr{H}^{j_0}} > 0.  \label{eq4: 07.01.26}
  \end{align}
  Otherwise, by \eqref{eq:independence_number_strong_product} and \eqref{eq4: 07.01.26}, for all $m \in \naturals$,
  $\indnum{\Gr{H}^m}=\indnum{\Gr{H}}^m$, so $\Theta(\Gr{H}) = \indnum{\Gr{H}}$, contradicting our assumption.
  Hence, $\indnum{\Gr{K}^{2k}} > \indnum{\Gr{K}^k}^2$ for all $k \geq \max\{i_0,j_0\}$. Set $k_0 \triangleq \max\{i_0,j_0\}$.

  Suppose by contradiction that $\Theta(\Gr{K})=\sqrt[k]{\indnum{\Gr{K}^{k}}}$ for some $k \geq k_0$. Then,
  by \eqref{eq:independence_number_strong_product}, \eqref{eq3: 07.01.26}, and \eqref{eq4: 07.01.26}, we get
  $\indnum{\Gr{K}^{2k}} > \indnum{\Gr{K}^{k}}^2$, which gives
  \begin{align}
    \Theta(\Gr{K}) \geq \sqrt[2k]{\indnum{\Gr{K}^{2k}}} > \sqrt[k]{\indnum{\Gr{K}^{k}}} = \Theta(\Gr{K}),
  \end{align}
  thus leading to a contradiction. Hence, the equality $\Theta(\Gr{K})=\sqrt[k]{\indnum{\Gr{K}^{k}}}$ cannot hold for any $k \geq k_0$.

  Note that if we assume that the equality $\Theta(\Gr{K})=\sqrt[k]{\indnum{\Gr{K}^{k}}}$ holds for some $k < k_0$, then for every $n \in \naturals$,
  \begin{align}
    \Theta(\Gr{G}) \geq \sqrt[nk]{\indnum{\Gr{K}^{nk}}}\geq\sqrt[k]{\indnum{\Gr{K}^k}} = \Theta(\Gr{G}),
  \end{align}
  which, using the same argument for $nk \geq k_0$, leads to a contradiction. Hence, the Shannon capacity of the graph $\Gr{K}$ is not attained
  by the independence number of any finite strong power of this graph.
\end{proof}


\section{The original proof from \texorpdfstring{\cite{Alon98}}{} of inequality \texorpdfstring{\eqref{eq4:11.09.2024}}{(Eq.~\ref*{eq4:11.09.2024})}}
\label{appendix: Noga's proof}

\begin{theorem}
  \cite{Alon98} Let $\Gr{G}$ be a simple graph on $n$ vertices. Then,
  \begin{align}\label{eq20:24.11.2024}
    \Theta(\Gr{G}+\CGr{G}) \geq 2 \sqrt{n}.
  \end{align}
\end{theorem}
\begin{proof}
  Let $\set{A} \cup \set{B} = \{a_1, \ldots, a_n, b_1, \ldots, b_n\}$ be the vertex set of $\Gr{G}+\CGr{G}$, where $\set{A}$ and $\set{B}$ 
  are the vertex sets of a copy of $\Gr{G}$ and of its complement $\CGr{G}$, respectively. The labeling is chosen so that, for $i \neq j$, 
  $\{a_i, a_j\} \in \E{\Gr{G}}$ if and only if $\{b_i, b_j\} \notin \E{\CGr{G}}$. Moreover, since $\Gr{G} + \CGr{G}$ is a disjoint union, 
  there are no edges of the form $\{a_i, b_j\}$.
  For each $k \in \naturals$, we construct an independent set in the strong power $(\Gr{G}+\CGr{G})^{2k}$, which is then used to bound the 
  Shannon capacity of $\Gr{G}+\CGr{G}$.
  For every $k$, define $\set{S}_k$ as the set of all vectors $\textbf{v}=(v_1,v_2,\ldots,v_{2k})$ with coordinates in $\set{A} \cup \set{B}$
  that satisfy the following two conditions:
  \begin{enumerate}
    \item $\bigcard{\{i: v_i \in \set{A}\}} = \bigcard{\{j: v_j \in \set{B}\}} = k$,
    \item For every $i \in \OneTo{k}$, if $a_r$ is the $i$th coordinate of $\textbf{v}$ (from left to right) belonging to $\set{A}$, and $b_s$ 
    is the $i$th coordinate of $\textbf{v}$ belonging to $\set{B}$, then $r = s$.
  \end{enumerate}
  Next, we prove that $\set{S}_k$ is an independent set in $(\Gr{G}+\CGr{G})^{2k}$. Let $\textbf{u}$ and $\textbf{v}$ be two distinct vectors in $\set{S}_k$.
  We consider the following two cases:
  \begin{itemize}
  \item
  Case 1: If there exists an index $t$ such that $u_t \in \set{A}$ and $v_t \in \set{B}$ (or vice versa), then the vertices $\textbf{u}$ and $\textbf{v}$ are clearly nonadjacent in $(\Gr{G}+\CGr{G})^{2k}$.
  \item
  Case 2: If there is no such an index, then there exist two indices $1 \leq i,j \leq k$, and two positions $1 \leq r,s \leq 2k$ such that $u_r = a_i$, $u_s = b_i$,
  $v_r = a_j$, and $v_s = b_j$. Since $\{a_i,a_j\} \in \E{\Gr{G}}$ if and only if $\{b_i,b_j\} \notin \E{\CGr{G}}$, the vertices $\textbf{u}$ and $\textbf{v}$ are
  again nonadjacent in $(\Gr{G}+\CGr{G})^{2k}$.
  \end{itemize}
  Hence, $\set{S}_k$ is an independent set in $(\Gr{G}+\CGr{G})^{2k}$.
  The cardinality of $\set{S}_k$ is $\binom{2k}{k} \, n^{k}$. Indeed, there are $\binom{2k}{k}$ ways to choose the positions of the coordinates from $\set{A}$,
  which uniquely determines the positions of the coordinates from $\set{B}$. For each such choice, there are $n$ possible values for each of the $k$ coordinates,
  yielding $\card{\set{S}_k} = \binom{2k}{k} \, n^{k}$. Consequently, 
  \begin{align}
    \indnum{(\Gr{G}+\CGr{G})^{2k}} \geq \card{\set{S}_k} = \binom{2k}{k} \, n^k,
  \end{align}
  which yields 
  \begin{align}
    \Theta(\Gr{G}+\CGr{G}) & \geq \lim_{k\to\infty}{\sqrt[2k]{\indnum{(\Gr{G}+\CGr{G})^{2k}}}} \nonumber \\
    & \geq \lim_{k \to \infty}{\sqrt[2k]{\binom{2k}{k} \, n^k}} \nonumber \\
    & = 2\sqrt{n}.
  \end{align}
\end{proof}

\begin{remark} \label{remark: appendix E}
It is noted in \cite{Alon98} that the inequality $\Theta(\Gr{G} + \CGr{G}) \geq \sqrt{2n}$, which is a loosened 
version of inequality \eqref{eq20:24.11.2024}, easily follows for every simple graph $\Gr{G}$ on $n$ vertices. 
This holds since 
\[
\bigl\{ (a_i, b_i) \bigr\}_{i \in \OneTo{n}} \bigcup \bigl\{ (b_i, a_i) \bigr\}_{i \in \OneTo{n}}
\]
is an independent set of the square $(\Gr{G} + \CGr{G})^2$, consisting of $2n$ elements. 
Furthermore, it is also noted in \cite{Alon98} that inequality \eqref{eq20:24.11.2024} is easily obtained for every 
self-complementary and vertex-transitive graph $\Gr{G}$. Indeed, by
\cite[Theorem~12]{Lovasz79}, $\Theta(\Gr{G}) = \sqrt{n} = \Theta(\CGr{G})$ for every such graph on $n$ vertices, 
so it follows from Shannon's inequality \eqref{eq:shannon_disjoint_union} that $\Theta(\Gr{G} + \CGr{G}) \geq 2 \sqrt{n}$.
Our strengthened result in Corollary~\ref{corollary:alon_inequality} asserts that inequality \eqref{eq20:24.11.2024} holds with equality
for all self-complementary vertex-transitive graphs, as well as for some additional graph families. 
\end{remark}

\subsection*{Acknowledgments}
The authors gratefully acknowledge the timely and constructive reports of the three anonymous referees,
which contributed to improving the presentation.



\begin{thebibliography}{99}

\bibitem{Shannon56}
C. E. Shannon, The zero error capacity of a noisy channel, {\em IEEE T. Inform. Theory}, 2 (1956), 8--19.
\href{https://doi.org/10.1109/TIT.1956.1056798}{here}

\bibitem{Alon02}
N. Alon, {\em Graph powers}, In:  Contemporary Combinatorics (B. Bollobás, Ed.), Bolyai Soc. Math. Stud.,
Springer, Budapest, Hungary, 10 (2002), 11--28.
\href{https://www.tau.ac.il/~nogaa/PDFS/cap2.pdf}{here}

\bibitem{Alon19}
N. Alon, {\em Lov\'{a}sz, vectors, graphs and codes}, In: Building Bridges II---Mathematics of
L\'{a}szl\'{o} Lov\'{a}sz (I. B\'{a}r\'{a}ny, G. O. H. Katona and A. Sali, Eds.), Bolyai Soc. Math. Stud.,
Springer, Budapest, Hungary, 28 (2019), 1--16.
\href{https://doi.org/10.1007/978-3-662-59204-5_1}{here}

\bibitem{Jurkiewicz14}
M. Jurkiewicz, A survey on known values and bounds on the Shannon capacity, in {\em Selected Topics in Modern Mathematics - Edition 2014},
eds. G. Gancarzewicz, M. Skrzy\'{n}ski, Publishing House AKAPIT, Krak\'{o}w, Poland, 2014, 115--128.
\href{https://repozytorium.biblos.pk.edu.pl/resources/25729}{here}

\bibitem{KornerO98}
J. K\"{o}rner, A. Orlitsky, Zero-error information theory, {\em IEEE T. Inform. Theory}, 44 (1998), 2207--2229.
\href{https://doi.org/10.1109/18.720537}{here}

\bibitem{CsiszarK01}
I. Csisz\'{a}r, J. K\"{o}rner, {\em Information Theory, Coding Theorems for Discrete Memoryless Systems},
2 Eds., Cambridge University Press, 2011.

\bibitem{Alon98}
N. Alon, The Shannon capacity of a union, {\em Combinatorica}, 18 (1998), 301–310.
\href{https://doi.org/10.1007/PL00009824}{here}

\bibitem{AlonL06}
N. Alon, E. Lubetzky, The Shannon capacity of a graph and the independence numbers of its powers,
{\em IEEE T. Inform. Theory}, {\bf 52} (2006), 2172--2176.
\href{https://doi.org/10.1109/TIT.2006.872856}{here}

\bibitem{BocheD25}
H. Boche, C. Deppe, Computability of the zero-error capacity of noisy channels. Information, 16 (2025), paper~571.
\href{https://doi.org/10.1109/ITW48936.2021.9611383}{here}

\bibitem{GuoW90}
F. Guo, Y. Watanabe, On graphs in which the Shannon capacity is unachievable by finite product, {\em IEEE T. Inform. Theory}, 36 (1990), 622--623.
\href{https://doi.org/10.1109/18.54907}{here}

\bibitem{Lovasz79}
L. Lov\'{a}sz, On the Shannon capacity of a graph, {\em IEEE T. Inform. Theory}, 25 (1979), 1--7.
\href{https://doi.org/10.1109/TIT.1979.1055985}{here}

\bibitem{Lovasz19}
L. Lov\'{a}sz, {\em Graphs and Geometry}, American Mathematical Society, vol.~65, 2019.
\href{https://doi.org/10.1090/coll/065}{here}

\bibitem{BiT19}
Y. Bi, A. Tang, On upper bounding Shannon capacity of graph through generalized conic programming,
{\em Optim. Lett.}, 13 (2019), 1313--1323.
\href{https://doi.org/10.1007/s11590-019-01436-7}{here}

\bibitem{BukhC19}
B. Bukh, C. Cox, On a fractional version of Haemers’ bound, {\em IEEE T. Inform. Theory}, 65 (2019), 3340--3348.
\href{https://doi.org/10.1109/TIT.2018.2889108}{here}

\bibitem{Haemers79}
W. H. Haemers, On some problems of Lov\'{a}sz concerning the Shannon capacity of a graph, {\em IEEE T. Inform. Theory},
25 (1979), 231--232. \href{https://doi.org/10.1109/TIT.1979.1056027}{here}

\bibitem{HuTS18}
S. Hu, I. Tamo, O. Sheyevitz, A bound on the Shannon capacity via a linear programming variation,
{\em SIAM J. Discrete Math.}, 32 (2018), 2229--2241.
\href{https://doi.org/10.1137/17M115565X}{here}

\bibitem{Knuth94}
D. E. Knuth, The sandwich theorem, {\em Electron. J. Combin.}, 1 (1994), 1--48.
\href{https://doi.org/10.37236/1193}{here}

\bibitem{CsonkaS24}
B. Csonka, G. Simonyi, Shannon capacity, Lov\'{a}sz $\vartheta$ number, and the Mycielski construction,
{\em IEEE T. Inform. Theory}, 70 (2024), 7632--7646.
\href{https://doi.org/10.1109/TIT.2024.3394775}{here}

\bibitem{Simonyi21}
G. Simonyi, Shannon capacity and the categorical product, {\em Electron. J. Combin.}, 28 (2021), 1--23.
\href{https://doi.org/10.37236/9113}{here}

\bibitem{BoerBZ24}
D. de Boer, P. Buys, J. Zuiddam, The asymptotic spectrum distance, graph limits,
and the Shannon capacity, {\em preprint}, 2024.
\href{https://doi.org/10.48550/arXiv.2404.16763}{here}

\bibitem{BuysPZ25}
P. Buys, S. Polak, J. Zuiddam, A group-theoretic approach to Shannon capacity of graphs and
a limit theorem from lattice packings, {\em preprint}, 2025.
\href{https://doi.org/10.48550/arXiv.2506.14654}{here}

\bibitem{Zuiddam19}
J. Zuiddam, The asymptotic spectrum of graphs and the Shannon capacity, {\em Combinatorica}, 39 (2019), 1173--1184.
\href{https://doi.org/10.1007/s00493-019-3992-5}{here}

\bibitem{WigdersonZ23}
A. Wigderson, J. Zuiddam, Asymptotic spectra: Theory, applications and extensions, {\em preprint}, October 2023.
\href{https://staff.fnwi.uva.nl/j.zuiddam//papers/convexity.pdf}{here}

\bibitem{Strassen88}
V. Strassen, The asymptotic spectrum of tensors, {\em J. Reine Angew. Math.}, 384 (1988), 102--152.
\href{https://doi.org/10.1515/crll.1988.384.102}{here}

\bibitem{AbiadDF26}
A. Abiad, C. Dalf\'{o}, M. A. Fiol, The Shannon capacity of graph powers,
{\em IEEE T. Inform. Theory}, 72 (2026), 100--111.
\href{https://doi.org/10.1109/TIT.2025.3628011}{here}

\bibitem{PolakS19}
S. C. Polak, A. Schrijver, New lower bound on the Shannon capacity of $\CG{7}$ from circular graphs, {\em Inform. Process. Lett.}, 143 (2019), 37--40.
\href{https://doi.org/10.1016/j.ipl.2018.11.006}{here}

\bibitem{Bohman05b}
T. Bohman, A limit theorem for the Shannon capacities of odd cycles. II, {\em Proc. Amer. Math. Soc.}, 133 (2005), 537--543.
\href{https://www.jstor.org/stable/4097960}{here}

\bibitem{Schrijver23}
A. Schrijver, On the Shannon capacity of sums and products of graphs, {\em Indag. Math.}, 34 (2023), 37--41.
\href{https://doi.org/10.1016/j.indag.2022.08.009}{here}

\bibitem{GodsilK16}
C. Godsil, K. Meagher, Erd\H{o}s-Ko-Rado Theorems, {\em Cambridge University Press}, 2016.
\href{https://doi.org/10.1017/CBO9781316414958}{here}

\bibitem{LvW12}
B. Lv, K. Wang, The eigenvalues of $q$-Kneser graphs, {\em Discrete Math}, 312 (2012) 1144-1147.
\href{https://doi.org/10.1016/j.disc.2011.11.042}{here}

\bibitem{HammackIK11}
R. Hammack, W. Imrich, S. Klav\v{z}ar, {\em Handbook of Product Graphs}, second edition, 2011.
\href{https://doi.org/10.1201/b10959}{here}

\bibitem{Hales73}
R. S. Hales, Numerical invariants and the strong product of graphs, {\em J. Combin. Theory Ser. B}, 15 (1973), 146--155.
\href{https://doi.org/10.1016/0095-8956(73)90014-2}{here}

\bibitem{GodsilR}
C. Godsil, G. Royle, {\em Algebraic graph theory}, Springer, 2001.
\href{https://doi.org/10.1007/978-1-4613-0163-9}{here}

\bibitem{Lovasz26}
L. Lov\'{a}sz, {\em personal communication}, January 2026.

\bibitem{Sason25}
I. Sason, An example showing that Schrijver's $\vartheta$-function need not upper bound the Shannon capacity of a graph,
{\em AIMS Mathematics}, vol. 10, no. 7, paper 695, pp. 15294-15301, July 2025.
\href{https://doi.org/10.3934/math.2025685}{here}

\bibitem{Sason23}
I. Sason, Observations on Lov\'{a}sz $\vartheta$-function, graph capacity, eigenvalues, and strong products, {\em Entropy}, 25 (2023), 104, 1--40.
\href{https://doi.org/10.3390/e25010104}{here}

\bibitem{ScheinermanU13}
E. R. Scheinerman, D. H. Ullman, {\em Fractional Graph Theory: A Rational Approach to the Theory of Graphs}, Dover Publications, 2013.
\href{https://www.ams.jhu.edu/ers/wp-content/uploads/2015/12/fgt.pdf}{here}

\bibitem{BrouwerS79}
A. E. Brouwer, A. Schrijver, Uniform hypergraphs, in: A. Schrijer (ed.), Packing and Covering in Combinatorics, Mathematical Centre Tracts 106, Amsterdam (1979), 39-73.

\bibitem{Sason24}
I. Sason, Observations on graph invariants with the Lov\'{a}sz $\vartheta$-function, {\em AIMS Mathematics}, vol.~9, no.~6, pp.~15385--15468, April 2024.
\href{https://doi.org/10.3934/math.2024747}{here}

\bibitem{SageMath}
The Sage Developers, SageMath, the Sage Mathematics Software System, Version 9.3, August 2021.
\href{https://www.sagemath.org/}{here}

\bibitem{BohmanHN13}
T. Bohman, R. Holzman, V. Natarajan, On the independence numbers of the cubes of odd cycles, {\em Electron. J. Combin.}, 20 (2013), P10, 1--19.
\href{https://doi.org/10.37236/2598}{here}

\bibitem{CharpenayT2020}
N. Charpenay, M. Le Treust, Zero-error coding with a generator set of variable-length words, {\em Proceedings
of the 2020 IEEE International Symposium on Inform. Theory}, pp.~78--83, Los Angeles, CA, USA, 21--26 June, 2020.
\href{https://doi.org/10.1109/ISIT44484.2020.9174059}{here}

\bibitem{AignerZ18}
M. Aigner, G. M. Ziegler, {\em Proofs from THE BOOK}, 6th ed., Springer, Berlin, Germany, 2018.
\href{https://doi.org/10.1007/978-3-662-57265-8}{here}

\bibitem{DorflerImrich70}
W. D{\"o}rfler, W. Imrich, \"{U}ber das starke produkt von graphen, {\em Monatshefte für Mathematik},
74 (1970), 97--102.

\bibitem{Mckenzie71}
R. McKenzie, Cardinal multiplication of structures with a reflexive relation, {\em Fund. Math.}, 70 (1971), 59--101.
\href{https://scispace.com/pdf/cardinal-multiplication-of-structures-with-a-reflexive-1bw5w0tzrt.pdf}{here}

\bibitem{FeigenbaumS91}
J. Feigenbaum, A. A. Sch\"{a}ffer, Finding the prime factors of strong direct product graphs in polynomial time,
{\em Discrete Mathematics}, vol.~109, no.~1--3, pp.~77--102, November 1992.
\href{https://doi.org/10.1016/0012-365X(92)90280-S}{here}

\bibitem{CharpenayTR26}
N. Charpenay, M. Le Treust, A. Roumy, On the additivity of optimal rates for independent zero-error source and channel
problems, to appear in {\em IEEE T. Inform. Theory}, preprint from December~4th, 2025. \href{https://doi.org/10.1109/TIT.2025.3639292}{here}

\bibitem{Marton93}
K. Marton, On the Shannon capacity of probabilistic graphs, {\em Journal of Combinatorial Theory, Series B}, vol.~57, no.~2, pp.~183--195,
March 1993. \href{https://doi.org/10.1006/jctb.1993.1015}{here}

\bibitem{GreeneK78}
C. Greene, D. J. Kleitman, Proof techniques in the theory of finite sets, in MAA Studies in Math. Assoc.
of Amer., Washington, D.C., 1978.
\href{https://www.semanticscholar.org/paper/Proof-techniques-in-the-theory-of-finite-sets-Greene-Kleitman/d6d94511006c1f9a7e9212cbb05f636060cebd78}{here}

\bibitem{Jukna11}
S. Jukna, Extremal Combinatorics with Applications in Computer Science, second edition, {\em Springer}, 2011.
\href{https://doi.org/10.1007/978-3-642-17364-6}{here}

\bibitem{Lovasz75}
L. Lov\'{a}sz, On the ratio of optimal integral and fractional covers, {\em Discrete Math.}, 13 (1975), 383--390.
\href{https://doi.org/10.1016/0012-365X(75)90058-8}{here}

\bibitem{Mussche09}
T. J. J. Mussche, Extremal combinatorics in generalized Kneser graphs, {\em Discrete Algebra and Geometry},
Technische Universiteit Eindhoven, 2009.
\href{https://doi.org/10.6100/IR642440}{here}

\bibitem{Acin17}
A. Ac\'{i}n, R. Duanc, D. E. Roberson, A. B. Sainz, A. Winter, A new property of the Lov\'{a}sz number
and duality relations between graph parameters, {\em Discrete Appl. Math.}, 216 (2017), 489--501.
\href{https://doi.org/10.1016/j.dam.2016.04.028}{here}

\bibitem{Bollobas01}
B. Bollob\'{a}s, {\em Random Graphs}, second edition, Cambridge University Press, 2001.
\href{https://doi.org/10.1017/CBO9780511814068}{here}

\bibitem{Juhasz82}
F. Juh{\'a}sz, On the asymptotic behavior of the Lov{\'a}sz theta function for random graphs, {\em Combinatorica},
vol.~2, no.~2, pp.~153--155, June 1982. \href{https://doi.org/10.1007/BF02579314}{here}

\end{thebibliography}
\end{document}